\documentclass[a4paper,12pt]{amsart}
\usepackage{yfonts}

\usepackage{amssymb}

\def\hbar{\bar{h}}

\def\mapsfrom{\leftarrow\!\shortmid}

\def\iso{\buildrel \sim\over\to}

\def\GS{{\mathfrak{S}}}

\def\Gg{{\mathfrak{g}}}

\def\Gsl{{\mathfrak{sl}}}

\def\CA{{\mathcal{A}}}
\def\CB{{\mathcal{B}}}
\def\CC{{\mathcal{C}}}
\def\CD{{\mathcal{D}}}
\def\CE{{\mathcal{E}}}

\def\CH{{\mathcal{H}}}

\def\CL{{\mathcal{L}}}
\def\CM{{\mathcal{M}}}

\def\CO{{\mathcal{O}}}
\def\CP{{\mathcal{P}}}

\def\CS{{\mathcal{S}}}
\def\CT{{\mathcal{T}}}

\def\CV{{\mathcal{V}}}
\def\CW{{\mathcal{W}}}

\def\BA{{\mathbf{A}}}

\def\BC{{\mathbf{C}}}

\def\BG{{\mathbf{G}}}

\def\BQ{{\mathbf{Q}}}

\def\BU{{\mathbf{U}}}

\def\BZ{{\mathbf{Z}}}

\def\Bk{{\mathbf{k}}}

\def\eps{\varepsilon}

\def\Ab{\operatorname{Ab}\nolimits}

\def\can{{\mathrm{can}}}

\def\Comp{\operatorname{Comp}\nolimits}

\def\End{\operatorname{End}\nolimits}

\def\Ext{\operatorname{Ext}\nolimits}

\def\mfree{\operatorname{\!-free}\nolimits}

\def\GL{\operatorname{GL}\nolimits}

\def\gr{{\operatorname{gr}\nolimits}}

\def\Ho{\operatorname{Ho}\nolimits}

\def\Hom{\operatorname{Hom}\nolimits}

\def\Id{\operatorname{Id}\nolimits}
\def\id{\operatorname{id}\nolimits}

\def\Ind{\operatorname{Ind}\nolimits}

\def\lcm{\operatorname{lcm}\nolimits}

\def\mMod{\operatorname{\!-mod}\nolimits}

\def\mMOD{\operatorname{\!-Mod}\nolimits}

\def\Ob{{\operatorname{Ob}\nolimits}}

\def\opp{{\operatorname{opp}\nolimits}}

\def\mproj{\operatorname{\!-proj}\nolimits}

\def\Res{\operatorname{Res}\nolimits}

\def\ie{{\em i.e.}}

\def\ts{{\tilde{s}}}
\def\tt{{\tilde{t}}}

\def\tI{{\tilde{I}}}

\def\tS{{\tilde{S}}}

\def\CHom{{{\mathcal H}om}}

\def\idun{\mathbf{1}}

\newtheorem{thm}{Theorem}[section]
\newtheorem{lemma}[thm]{Lemma}
\newtheorem{cor}[thm]{Corollary}
\newtheorem{prop}[thm]{Proposition}
\newtheorem{defi}[thm]{Definition}
\newtheorem{conj}[thm]{Conjecture}

\theoremstyle{definition}
\newtheorem{rem}[thm]{Remark}

\usepackage[all,2cell]{xy}
\UseAllTwocells
\xyoption{all}
\def\FA{{\textgoth{A}}}
\def\FB{{\textgoth{B}}}
\def\FC{{\textgoth{C}}}

\def\FH{{\textgoth{H}}}
\def\FL{{\textgoth{L}}}
\def\FO{{\textgoth{O}}}
\def\FT{{\textgoth{T}}}

\begin{document}
\title{$2$-Kac-Moody algebras}
\author{Rapha\"el Rouquier}
\address{Mathematical Institute,
University of Oxford, 24-29 St Giles', Oxford, OX1 3LB, UK}
\email{rouquier@maths.ox.ac.uk}

\begin{abstract}
We construct a $2$-category associated with a
Kac-Moody algebra and we study its $2$-representations. This generalizes
earlier work \cite{ChRou} for $\Gsl_2$. We relate
categorifications relying on $K_0$ properties as
in the approach of \cite{ChRou} and $2$-representations.
\end{abstract}

\date{29 December 2008}
\maketitle
\setcounter{tocdepth}{3}
\tableofcontents

\section{Introduction}

Over the past ten years, we have advocated the idea that there should exist
monoidal categories (or $2$-categories)
with an interesting ``representation theory'': we propose
to call ``$2$-representation theory''
this higher version of representation theory and to call ``$2$-algebras''
those ``interesting'' monoidal additive categories. The difficulty in
pinning down what is a $2$-algebra (or a Hopf version)
should be compared with the difficulty in defining
precisely the meaning of quantum groups (or quantum algebras). The analogy
is actually expected to be meaningful: while quantization turns certain
algebras into quantum algebras, ``categorification'' should turn those
algebras into $2$-algebras. Dequantization is specialization $q\to 1$,
while ``decategorification'' is the Grothendieck group construction ---
in the presence of gradings, it leads to a quantum object.
A large part of geometric representation theory should, and can, be 
viewed as a construction of ``irreducible'' $2$-representations as
categories of sheaves.

\smallskip
The starting point of the study of $2$-representation theory of Lie
algebras was the
definition in 2003 of $\Gsl$-categorifications with
Joseph Chuang and its study for $\Gsl_2$ in \cite{ChRou}.

\medskip
A crucial feature of $2$-representation theory is the construction
of a machinery that produces new categories out of some given categories
(with extra structure). We believe this should be viewed as an algebraic
counterpart of the construction of moduli spaces as families of
sheaves or other objects on a variety. The following oversimplified
diagram explains how our algebraic constructions would reproduce the
various counting invariants based on moduli spaces, bypassing the moduli
spaces and the difficulties of their construction and the construction
of their invariants
$$\xymatrix{
\text{Variety} X \ar@{~>}[r] \ar[d]&
 \text{Moduli space }\CM \text{ of objects on }X \ar[d]\\
\text{Category of sheaves on }X \ar@{-->}[r] &
\text{Category of sheaves on }\CM
}$$

\smallskip
While our focus here is on classical algebraic objects (related in some
way to $2$-dimensional geometry), it is our belief that
there should be $2$-algebras associated with $3$-dimensional
geometry, possibly non-commutative, and that their higher representation
theory would provide the proper algebraic framework for the various
couting invariants (Gromov-Witten, Donaldson-Thomas,...). 

\smallskip
In this paper, we define a $2$-category $\FA(\Gg)$ associated with a
Kac-Moody algebra $\Gg$. Modulo some
Hecke algebra isomorphisms, the generalization from type $A$ (finite or
affine) defined in joint work with Joseph Chuang is quite natural.

In  \cite{Rou1}, we define and study tensor structures on the
$2$-category of $2$-representations of $\FA(\Gg)$ on dg-categories, with
aim the construction of $4$-dimensional topological quantum field theories.
Our $2$-categories associated with Kac-Moody algebras
provide a solution to the question raised by Crane and Frenkel \cite{CrFr}
of the construction of ``Hopf categories''.

\medskip
The $2$-category $\FA(\Gg)$ ``categorifies'' (a completion of)
the $\BZ$-form $U_\BZ(\Gg)$
of the enveloping algebra of $\Gg$.
 Consequently, a $2$-representation of
$\FA(\Gg)$ on an exact or a triangulated category $\CV$ gives rise to
an action of $U_\BZ(\Gg)$ on $K_0(\CV)$. This gives a hint at the
very non-semi-simplicity of the theory of $2$-representations of
$\FA(\Gg)$. The presence of gradings actually gives rise to a
``categorification'' of the associated quantum group.

The Hecke algebras used in \cite{ChRou} are replaced by nil Hecke algebras
associated with Cartan matrices.
Some of their specializations
occur naturally as endomorphisms
of correspondences for quiver varieties \cite{Rou3}.
In type $A$, they 
occur when decomposing representations of (degenerate) affine Hecke algebras
according to the spectrum of the polynomial subalgebra,
and not just the center. These nil Hecke algebras can be defined by generators
and relations
and they also
have a simple construction as a subalgebra of a wreath product algebra.

We construct more generally a flat family of ``Hecke'' algebras over the space
of matrices over $k[u,v]$ which are hermitian with respect to
$u\leftrightarrow v$. They are filtered with associated graded
algebra a wreath product of a polynomial algebra by a nil Hecke algebra. They
satisfy the PBW property.

\smallskip
Consider a monoidal category or a $2$-category
defined by generators and relations.
A difficulty in $2$-representation theory is to check the 
defining relations in examples. The philosophy of \cite{ChRou} was,
instead of defining first the monoidal category, to describe directly
what a $2$-representation should be, using the action on the Grothendieck
group. A key result of this paper is to provide a similar approach
for Kac-Moody algebras.
We show, under certain finiteness assumptions, that it is enough to
check the relations $[e_i,f_j]=\delta_{ij}h_i$ on $K_0$.
This is needed to show that the earlier
definition of Chuang and the author of type $A$-categorifications coincides
with the more general notion defined here. It is also a crucial ingredient
for the construction of algebraic and geometric $2$-representations
in \cite{Rou3}.

\medskip
Let us describe in more detail the constructions and results of the paper.
\smallskip

We set up some of the formalism to deal with $2$-categories, presentations
by generators and relations and $2$-representations in
\S \ref{se:2cat}. An important role is played by biadjoint pairs and
\S \ref{se:symmalg} develops the theory of symmetric algebras over
non-commutative rings. In \S \ref{se:classicalHecke}, we gather classical
results on Hecke algebras of type $A$: affine, degenerate affine and
nil affine. We introduce Hecke algebras associated with hermitian matrices
in \S \ref{se:nilmatrix} and show they satisfy a PBW Theorem. We provide 
specializations associated with Cartan matrices and further specializations
associated with quivers (with automorphisms).

Given a Cartan datum, we construct in \S \ref{se:def2km} a $2$-category $\FA$ 
with set of objects the weight lattice $X$ and with $1$-arrows generated by
$E_s:\lambda\to\lambda+\alpha_s$ and 
$F_s:\lambda\to\lambda-\alpha_s$. The $2$-arrows are generated by
units and counits of dual pairs $(E_s,F_s)$ and by $x_s\in\End(E_s)$ and
$\tau_{st}\in\Hom(E_sE_t,E_tE_s)$. We impose relations so that 
there is an action of the nil Hecke algebras associated with the Cartan
matrix on products $E_{s_1}\cdots E_{s_n}$ induced by $x_s$ and
$\tau_{st}$. Finally, we invert certain maps relating
$E_sF_t$, $F_tE_s$ and a multiple of $\idun$ --- this accounts
for the decomposition of $[e_s,f_t]$ in the corresponding
Kac-Moody algebra $\Gg$. There is a morphism of algebras from a completion
of the $\BZ$-form of the
enveloping algebra of $\Gg$ to the Grothendieck group of $\FA$. 
The category $\FA$ is defined over a base ring with indeterminates and
a specialization of these leads to a graded category. There is
a morphism from the completed quantized enveloping algebra of $\Gg$ to
the graded Grothendieck group.

We introduce integrable $2$-representations of $\FA$ in
\S \ref{se:defintegrable}.
We show that for integrable $2$-representations of $\FA$,
there is a canonical adjunction $(F_s,E_s)$, giving rise to an action of
a $2$-category $\FA'$ (\S \ref{se:otherversions} and
Theorem \ref{th:critsigmast}). We provide a construction of
a $2$-representation $\CV(\lambda)$ with lowest weight $\lambda\in -X^+$ 
(\S \ref{se:simple2rep}) and
show that lowest weight integrable $2$-representations admit
Jordan-H\"older filtrations (Theorem \ref{th:JordanHolder}). The case
of $\Gsl_2$ is crucial for several proofs and \S \ref{se:sl2} is a study
of its $2$-representations $\CV(\lambda)$. We also introduce
three involutions $I$, $D$ and $\iota$ that allow to swap $E_s$ and $F_s$
in particular (\S \ref{se:symmetries} and \S \ref{se:iota}).

In \S \ref{se:defsl2cat}, we show that in the case of abelian
categories over a field with finite composition series, the notion of
$\Gsl_2$-categorifications of \cite{ChRou} coincides with that of
a $2$-representation of $\FA(\Gsl_2)$. We generalize this to type
$A$ (finite or affine) in \S \ref{se:slcat}. This builds on the
isomorphisms between (degenerate) affine Hecke algebras and Hecke algebras
associated with Cartan matrices of type $A$ constructed in
\S \ref{se:idempotents}.
This provides a 
powerful way to construct $2$-representations. We extend to general Kac-Moody
algebras two key facts: the relations of type ``$[e_s,f_t]=0$ when
$s\not=t$'' are a consequence of the other axioms (\S \ref{se:rank2}) and
for abelian categories as above, the relations of type
``$[e_s,f_s]=h_s$'' follow from their $K_0$ version (\S \ref{se:controlK0}).

\medskip
The main results of this paper have been announced at
seminars in Orsay, Paris and Kyoto in the Spring 2007. Certain
specializations of the nil Hecke algebras associated with quivers and
the resulting monoidal categories associated with ``half'' Kac-Moody
algebras have been introduced independently by Khovanov and Lauda
\cite{KhoLau1,KhoLau2}.
The relations between Hecke algebras associated with affine
type $A$ Cartan matrices and representations of finite Hecke algebras of type
$A$ have been studied independently by Brundan and Kleshchev \cite{BrKl}.

\section{Preliminaries}
\subsection{Notations and conventions}
Given $n\in\BZ$, we put $[n]=\frac{v^n-v^{-n}}{v-v^{-1}}$,
$[n]!=\prod_{i=1}^n [i]$ for $n\in\BZ_{\ge 0}$. We put also
$$\begin{bmatrix} i\\n\end{bmatrix}=\frac{\prod_{a=0}^{n-1} (v^{i-a}-
v^{a-i})}{\prod_{a=1}^n (v^a-v^{-a})} \text{ for }i\in\BZ.$$

\smallskip
Given $\Omega$ a finite interval of $\BZ$, we denote by
$\GS(\Omega)$ the symmetric group on
$\Omega$, viewed as a Coxeter group with generating
set $\{s_i=(i,i+1)\}$ where $i$ runs over the non-maximal elements of
$\Omega$.
We denote by $w(\Omega)$ the longest element of $\GS(\Omega)$. Given
$E$ a family of disjoint intervals of $\Omega$, we put
$\GS(E)=\prod_{\Omega'\in E}\GS(\Omega')$ and we denote by
$\GS(\Omega)^E$ (resp. ${^E\GS}(\Omega)$)
the set of minimal length representatives
of $\GS(\Omega)/\GS(E)$ (resp. $\GS(E)\setminus\GS(\Omega)$).
We put $\GS_n=\GS[1,n]$. Given $w\in\GS_n$, we put
$\delta_w=\delta_{1,w}$.

\smallskip
Let $k$ be a commutative ring. We write $\otimes$ for $\otimes_k$. Given
$M$ a graded $k$-module and $i$ an integer, we denote by $M(i)$
the graded $k$-module given by $M(i)_n=M_{n+i}$.

Given $P=\sum_{i\in\BZ}p_iv^i \in \BZ_{\ge 0}[v^{\pm 1}]$ a Laurent polynomial
with non-negative
coefficients, we put $Pk=\bigoplus_{i\in\BZ}k^{p_i}(-i)$.
Given $k'$ a $k$-algebra and $M$ a $k$-module, we put
$k'M=k'\otimes M$. We also put $PM=Pk\otimes M$.

Given $A$ a $k$-algebra, $\gamma$ an automorphism of $A$ and $M$ a
right $A$-module, we denote by $M_\gamma$ the right $A$-module
$\gamma^*M$: it is equal to $M$ as a $k$-module and the action of
$a\in A$ on $M_\gamma$ is given by $M_\gamma\ni m\mapsto m\cdot \gamma(a)$.
Given $M$ an $(A,A)$-bimodule, we put
$M^A=\{m\in M~|~am=ma,\ \forall a\in A\}$.

An $A$-algebra is an algebra $B$ endowed with a morphism of algebras
$A\to B$.
Given $B$ an $A$-algebra, we say that a $B$-module is {\em relatively
$A$-projective} if it is a direct summand of
$B\otimes_A M$ for some $A$-module $M$.

\smallskip
Categories are denoted by calligraphic letters $\CA,\CB,\CC$, etc. and
$2$-categories are denoted by gothic letters $\FA,\FB,\FC$, etc.

We denote by $\Ob(\CA)$ or by $\CA$ the set of objects of a category
(or of a $2$-category) $\CA$.
Given $a$ an object, we will denote by $a$ or $\idun_a$ or $1_a$ the
identity of $a$.

Given $F,G:\CA\to\CB$ two functors, a morphism $F\to G$ is the
data of a compatible
collection of arrows $F(a)\to G(a)$ for $a\in\CA$ and we call
these {\em natural morphisms}.

We say that an endofunctor $F$ of an additive category $\CC$
is {\em locally nilpotent} if for every $M\in\CC$, there is $n>0$ such that
$F^n(M)=0$.

\smallskip
We denote by $\CS ets$ (resp. $\CA b$)
 the category of sets (resp. of abelian groups).
We denote by $A\mMOD$ the category of $A$-modules, by
 $A\mMod$ the category of finitely generated $A$-modules and
by $A\mfree$ is full subcategory of free $A$-modules of finite rank.
Here, module means left module. Given $\CA$ an additive category, we
denote by $\Comp^b(\CA)$ 
the category of bounded complexes of objects of $\CA$ and by
$\Ho^b(\CA)$ the associated homotopy category.

\smallskip

 We denote by
$\FC at$ (resp. $\FA dd$, $\FL in_k$,
$\FA b$, $\FT ri$) the strict $2$-category of
categories (resp. of additive categories, of $k$-linear categories,
of abelian categories
with exact functors, of triangulated categories). When $k$ is a field,
we denote by $\FA b_k^f$ the $2$-category of $k$-linear
abelian categories all of whose objects have finite
composition series and such that $k=\End(V)$ for any simple object
$V$ ($1$-arrows are $k$-linear exact functors).

\subsection{$2$-Categories}
\label{se:2cat}
We set up in this section the appropriate formalism for $2$-representation
theory. At first, we recall the more classical setting of representation
theory as a study of functors.

\subsubsection{Categories}
Let $\CA$ and $\CB$ be two categories.
We denote by $\CHom(\CA,\CB)$ the category of functors $\CA\to\CB$:
we think of these as representations of $\CA$ in $\CB$. For example,
if $\CA$ has a unique object $\ast$ and $\CB=\CS ets$,
the category $\CHom(\CA,\CB)$ is equivalent to the category of
sets acted on by the monoid $\End(\ast)$.

Given $a\in\CA$, we have a functor $\Hom(a,-):\rho_a:\CA\to\CS ets$
(the regular representation when $\CA$ has a unique object).

We put $\CA^\vee=\CHom(\CA^\opp,\CS ets^\opp)$. The functor
$$\CA\to\CA^\vee,\ M\mapsto \Hom(M,-)$$
 is fully faithful (Yoneda's Lemma) and we
identify $\CA$ with a full subcategory of $\CA^\vee$ through this embedding.

Assume $\CA$ is enriched in abelian groups.
The {\em additive closure} of $\CA$ is the full additive subcategory $\CA^a$ of
the category of functors $\CA^\opp\to\CA b^\opp$ generated by objects
of $\CA$. Given $\CA'$ an additive category, the restriction functor
gives an equivalence from the category of additive functors $\CA^a\to\CA'$
to the category of functors enriched in abelian groups $\CA\to\CA'$.

\smallskip
Assume $\CA$ is an additive category.
We denote by $\CA^i$ the idempotent completion of $\CA$. Given $\CA'$ an
idempotent-complete additive category, restriction gives an equivalence
from the category of additive functors $\CA^i\to\CA'$ to the category
of additive functors $\CA\to\CA'$.

 Let $M\in\CA$ and let
$L$ be a right $\End(M)$-module. We denote by $L\otimes_{\End(M)}M$ the
object of $\CA^\vee$ defined by
$\Hom_{\End(M)^\opp}(L,\Hom(M,-))$.

Given $A$ a ring, the category of $A$-modules in $\CA$ is the category
of additive functors $A\to \CA$,
where $A$ is the category with one object $\ast$ and with $\End(\ast)=A$.
An object of that category
is an object $M$ of $\CA$ endowed with a morphism of rings $A\to \End(M)$.

Given an $A$-module $M$ in $\CA$
and $L$ a right $A$-module, we put
$L\otimes_A M=(L\otimes_A \End(M))\otimes_{\End(M)}M$. For example,
there is a canonical isomorphism $\BZ^n\otimes_\BZ M\iso M^n$.

Let $B$ be a commutative
ring endowed with a morphism $B\to Z(\CA)$ and let $A$ be a
$B$-algebra. We denote by $\CA\otimes_B A$ the additive category
with same objects as $\CA$ and
$\Hom_{\CA\otimes_B A}(M,N)=\Hom_\CA(M,N)\otimes_B A$, where
$B$ acts via $Z(\CA)$. Let $\CA'$ be 
$B$-linear category. We denote by
$\CA\otimes_B \CA'$ the additive closure of the category with set
of objects $\Ob(\CA)\times\Ob(\CA')$ and with
$\Hom((M,M'),(N,N'))=\Hom_{\CA}(M,N)\otimes_B \Hom_{\CA'}(M',N')$.
Given $\CA''$ a $B$-linear category, there is an equivalence between
$\CHom_{\FL in_B}(\CA\otimes_B \CA',\CA'')$ and the category of
$B$-bilinear functors $\CA\times\CA'\to \CA''$.

\medskip
An equivalence relation $\sim$ on a category is a relation on arrows
such that $f\sim f'$ implies $fg\sim f'g$ and $gf\sim gf'$ (whenever this
makes sense).
Given $\CA$ a category and $\sim$ a relation on arrows of $\CA$, 
we have a quotient category $\CA/\!\!\sim$ with same objects as $\CA$.
The quotient functor $\CA\to\CA/\!\!\sim$ induces a fully faithful functor
$\CHom(\CA/\!\!\sim,\CB)\to\CHom(\CA,\CB)$ for any category $\CB$.
A functor is in the image if and
only if two equivalent arrows have the same image under the functor. The
construction depends only on the equivalence relation on $\CA$
generated by $\sim$.

Let $k$ be a commutative ring and $\CA$ a $k$-linear category.
Given $S$ a set of arrows of $\CA$, let $\sim=\sim_S$ be the coarsest
equivalence
relation on $\CA$ such that $f\sim 0$ for every $f\in S$ and
$\{(f,g)\ |\ f\sim g\}$ is a $k$-submodule
of $\Hom(a,a')\oplus\Hom(a,a')$.
We denote by $\CA/S=\CA/\!\!\sim$ the quotient
$k$-linear category: a $k$-linear functor $\CA\to\CB$ factors
through $\CA/S$, and then the factorization is unique up to unique
isomorphism, if and only if it sends arrows in $S$ to $0$.

\smallskip
Given $\CA$ a category,
we denote by $k\CA$ the $k$-linear
category associated with $\CA$: there is a canonical functor
$\CA\to k\CA$ and given a $k$-linear category $\CB$ and a functor
$F:\CA\to\CB$, there is a $k$-linear functor $G:k\CA\to\CB$ unique up to
unique isomorphism such that $F=G\cdot\can$.

\smallskip
Let $I=(I_0,I_1,s,t)$ be a quiver: this is the data of
\begin{itemize}
\item a set $I_0$ (vertices) and a set $I_1$ (arrows)
\item maps $s,t:I_1\to I_0$ (source and target).
\end{itemize}

We denote by $\CP=\CP(I)$ the set of paths in $I$, \ie, sequences
$(b_1,\ldots,b_n)$ of elements of $I_1$ such that $t(b_i)=s(b_{i-1})$ for
$1<i\le n$. It comes with maps
$s:\CP\to I_0,\ (b_1,\ldots,b_n)\mapsto s(b_n)$ (source) and
$t:\CP\to I_0,\ (b_1,\ldots,b_n)\mapsto t(b_1)$ (target). We write
$b_1\cdots b_n$ for the element $(b_1,\ldots,b_n)$ of $\CP$.

We denote by $\CC(I)$ the category generated
by $I$. Its set of objects is $I_0$ and $\Hom(i,j)=(s,t)^{-1}(i,j)$.
Composition is concatenation
of paths.

Let $\CA$ be a category. The category of diagrams of type $I$ in
$\CA$ is canonically isomorphic to the category of functors $\CC(I)\to\CA$
(the isomorphism is given by restricting the functor).

\medskip
A {\em graded category} is a category endowed with a self-equivalence $T$.
Given $M$ an object with isomorphism class $[M]$, we put
$v [M]=[T^{-1}(M)]$.

The $2$-category of graded $k$-linear categories is
equivalent to the $2$-category of $k$-linear categories enriched in
graded $k$-modules:
\begin{itemize}
\item Let $\CC$ be a graded $k$-linear category. We define $\CD$ as the category
with objects those of $\CC$ and with $\Hom_\CD(V,W)=\bigoplus_i
\Hom_\CC(V,T^iW)$. The composition of the maps of $\CD$ coming from maps
$f:V\to T^i W$ and $g:W\to T^j X$ of $\CC$ is the map coming from
$T^i(g)\circ f:V\to T^{i+j}X$.

\item Let $\CD$ be a $k$-linear category enriched in graded $k$-modules.
Define $\CC$ as the category with objects families $\{V_i\}_{i\in \BZ}$ with
$V_i$ an object of $\CD$ and $V_i=0$ for almost all $i$. We put
$\Hom_\CC(\{V_i\},\{W_i\})=\bigoplus_{i,j}\Hom_\CD(V_i,W_j)_{j-i}$. We define
$T(\{V_i\})_n=V_{n+1}$.
\end{itemize}

\subsubsection{Definitions}
\label{se:def2cat}
Our main reference for basic definitions and results on $2$-categories
is \cite{Gra} (cf also \cite{Le} for the basic definitions).

\begin{defi}
A $2$-category $\FA$ is the data of
\begin{itemize}
\item a set $\FA_0$ of objects
\item categories $\CHom(a,a')$ for $a,a'\in\FA_0$
\item functors $\CHom(a_1,a_2)\times\CHom(a_2,a_3)\to
\CHom(a_1,a_3),\ (b_1,b_2)\mapsto b_2b_1$ for $a_1,a_2,a_3\in\FA$
\item functors $I_a\in\CE nd(a)$ for $a\in\FA$
\item natural isomorphisms $(b_3b_2)b_1\iso b_3(b_2b_1)$
for $b_i\in\CHom(a_i,a_{i+1})$ and $a_1,\ldots,a_4\in\FA$.
\item natural isomorphisms $b I_a\iso b$ for $b\in\CHom(a,a')$ and
$a,a'\in\FA$
\item natural isomorphisms $I_a b\iso b$ for $b\in\CHom(a',a)$ and
$a,a'\in\FA$
\end{itemize}
such that the following diagrams commute
$$\xymatrix{
& \bigl((b_4b_3)b_2\bigr)b_1 \ar[rr]^{\can(b_4,b_3,b_2)\cdot b_1}
  \ar[dl]_{\can(b_4b_3,b_2,b_1)} &&
 \bigl(b_4(b_3b_2)\bigr) b_1\ar[dr]^{\can(b_4,b_3b_2,b_1)} \\
(b_4b_3)(b_2b_1) \ar[drr]_{\can(b_4,b_3,b_2b_1)} &&&& b_4\bigl( (b_3b_2)b_1\bigr) 
  \ar[dll]^{b_4\cdot\can(b_3,b_2,b_1)}\\
&& b_4\bigl(b_3(b_2b_1)\bigr)
}$$
$$\xymatrix{
(b_2 I_a)b_1 \ar[rr]^{\can(b_2,I_a,b_1)} \ar[dr]_{\can(b_2)\cdot b_1} &&
 b_2(I_ab_1) \ar[dl]^{b_2\cdot\can(b_1)} \\
& b_2b_1
}$$
\end{defi}

Note that $2$-categories are called bicategories in \cite{Gra}.
A {\em strict $2$-category}
is a $2$-category where
the associativity and unit isomorphisms are identity maps:
$(b_3b_2)b_1=b_3(b_2b_1)$ and $bI_a=b$, $I_ab=b$ (called
$2$-category in \cite{Gra}).

\smallskip
Let $\FA$ be a $2$-category. Its $1$-arrows (resp. $2$-arrows)
are the objects (resp. arrows) of the categories $\CHom(a,a')$ 

\smallskip
Given $b:a\to a'$ and $b':a'\to a''$ two $1$-arrows, we denote by
$b'b:a\to a''$ their composition.
The composition of $2$-arrows $c$ and $c'$ (viewed as arrows in a category
$\CHom(a,a')$)
is denoted by $c'\circ c$.
Given $a$, $a'$ and $a''$ three objects of $\FA$, $b_1,b_2:a\to a'$, 
$c:b_1\to b_2$ and $b'_1,b'_2:a'\to a''$, $c':b'_1\to b'_2$, we denote
by $c'c:b'_1b_1\to b'_2b_2$ the ``juxtaposition''.

\smallskip
We say that a $1$-arrow $b:a_1\to a_2$ is
\begin{itemize}
\item an {\em equivalence} if there is a
$1$-arrow $b':a_2\to a_1$ and isomorphisms $I_{a_1}\iso b'b$ and
$bb'\iso I_{a_2}$
\item {\em fully faithful} if
given any object $a''$, the functor
$\CHom(a'',b):\CHom(a'',a_1)\to\CHom(a'',a_2)$ is fully faithful.
\end{itemize}
Note that these notions coincide with the usual notions for
$\FA=\FC at$, $\FA=\FA dd$, $\FA=\FA b$ or $\FA=\FT ri$.

\smallskip
Given a $2$-category $\FA$, we denote by $\FA_{\le 1}$ the 
category with objects those of $\FA$ and with arrows the isomorphism
classes of $1$-arrows of $\FA$.

\smallskip
The {\em opposite} $2$-category $\FA^\opp$ of $\FA$ has same set of objects
as $\FA$ and $\CHom_{\FA^\opp}(a,a')=\CHom_{\FA}(a,a')^\opp$, while the
rest of the structure is inherited from that of $\FA$.

The {\em reverse} $2$-category $\FA^{\mathrm{rev}}$ of $\FA$ has same set of
objects
as $\FA$ and $\CHom_{\FA^\opp}(a,a')=\CHom_{\FA}(a',a)$. The composition
$$\CHom_{\FA^{\mathrm{rev}}}(a_1,a_2)\times
\CHom_{\FA^{\mathrm{rev}}}(a_2,a_3)\to
\CHom_{\FA^{\mathrm{rev}}}(a_1,a_3)$$
is given by $(b_1,b_2)\mapsto b_1b_2$ (composition in $\FA$). The
rest of the structure is inherited from that of $\FA$. 

\begin{defi}
A {\em $2$-functor} $R:\FA\to\FB$ between $2$-categories is
the data of
\begin{itemize}
\item a map $R:\Ob(\FA)\to\Ob(\FB)$
\item functors
$R:\FH om(a,a')\to\FH om(R(a),R(a'))$ for $a,a'\in\FA$
\item natural isomorphisms
$R(b_2)R(b_1)\iso R(b_2b_1)$ for $b_1,b_2$ $1$-arrows of $\FA$
\item invertible $2$-arrows $I_{R(a)}\iso R(I_a)$ for $a\in\FA$
\end{itemize}
such that the following diagrams commute
$$\xymatrix{
\bigl(R(b_3)R(b_2)\bigr)R(b_1) \ar[rr]^-{\can(b_3,b_2)\cdot R(b_1)} 
  \ar[d]_{\can(R(b_3),R(b_2),R(b_1))}&&
 R(b_3b_2)R(b_1) \ar[rr]^-{\can(b_3b_2,b_1)} && R\bigl( (b_3b_2)b_1\bigr)
  \ar[d]^{R(\can(b_3,b_2,b_1))} \\
R(b_3)\bigl(R(b_2)R(b_1)\bigr)\ar[rr]_-{R(b_3)\cdot \can(b_2,b_1)} &&
 R(b_3)R(b_2b_1) \ar[rr]_-{\can(b_3,b_2b_1)} && R\bigl(b_3(b_2b_1)\bigr)
}$$
$$\xymatrix{
R(b) I_{R(a)} \ar[rr]^-{R(b)\cdot\can(a)} \ar[d]_-{\can(R(b))} && R(b) R(I_a)
   \ar[d]^-{\can(b,I_a)}\\
R(b) && R(b I_a)\ar[ll]^-{R(\can(b))}
}
\hskip 5pt
\xymatrix{
I_{R(a')}R(b)\ar[rr]^-{\can(a')\cdot R(b)} \ar[d]_{\can(R(b))}&& R(I_{a'}) R(b) 
  \ar[d]^{\can(I_{a'},b)} \\
R(b) && R(I_{a'}b) \ar[ll]^-{R(\can(b))}
}$$
\end{defi}

When the $2$-arrows are identity maps $I_{R(a)}=R(I_a)$,
we say that the $2$-functor is {\em strict} (called strict pseudo-functor
in \cite{Gra}). 

\begin{defi}
A {\em morphism of $2$-functors} $\sigma:R\to R'$ is the data of 
\begin{itemize}
\item $1$-arrows $\sigma(a):R(a)\to R'(a)$
\item natural isomorphisms $R'(b)\sigma(a_1)\iso\sigma(a_2)R(b)$
for all $1$-arrows $b:a_1\to a_2$
\end{itemize}
such that the following diagrams commute
$$\xymatrix{
\bigl(R'(b_2) R'(b_1)\bigr)\sigma(a_1) \ar[d]^{\can(b_2,b_1)\cdot \sigma(a_1)}
  \ar[rrr]^-{\can(R'(b_2),R'(b_1),\sigma(a_1))} &&&
 R'(b_2)\bigl(R'(b_1)\sigma(a_1)\bigr) \ar[rrr]^-{R'(b_2)\cdot\can(b_1)} &&&
 R'(b_2)\bigl(\sigma(a_2)R(b_1)\bigr) 
  \ar[d]_{\can(R'(b_2),\sigma(a_2),R(b_1))^{-1}}\\
R'(b_2b_1)\sigma(a_1) \ar[d]^{\can(b_2b_1)} &&&&&&
 \bigl(R'(b_2)\sigma(a_2)\bigr)R(b_1) \ar[d]_{\can(b_2)R(b_1)} \\
\sigma(a_3) R(b_2b_1) &&& \sigma(a_3)\bigl(R(b_2)R(b_1)\bigr)
  \ar[lll]^{\sigma(a_3)\cdot\can(b_2,b_1)}&&&
 \bigl(\sigma(a_3)R(b_2)\bigr)R(b_1)\ar[lll]^-{\can(\sigma(a_3),R(b_2),R(b_1))}
}$$
$$\xymatrix{
I_{R'(a)}\sigma(a)\ar[r]^-{\can} \ar[d]_{\can\cdot\sigma(a)}
 & \sigma(a)\ar[r]^-{\can^{-1}} &
 \sigma(a)I_{R(a)} \ar[d]^{\sigma(a)\cdot\can} \\
R'(I_a)\sigma(a) \ar[rr]_{\can} && \sigma(a)R(I_a)
}$$
\end{defi}
These are quasi-natural transformations
with invertible $2$-arrows in \cite{Gra}.

\begin{defi}
A {\em morphism} $\gamma:\sigma\to \tilde{\sigma}$,
where $\sigma,\tilde{\sigma}:R\to R'$
are morphisms of $2$-functors, is the data of
$2$-arrows $\gamma(a):\sigma(a)\to \tilde{\sigma}(a)$ for $a\in\FA$
such that the following diagrams commute
$$\xymatrix{
R'(b)\sigma(a_1) \ar[rr]^-{R'(b)\gamma(a_1)} \ar[d]_{\can}^\sim &&
 R'(b)\tilde{\sigma}(a_1) \ar[d]^{\can}_\sim \\
\sigma(a_2)R(b) \ar[rr]_-{\gamma(a_2)R'(b)} && \tilde{\sigma}(a_2)R(b)
}$$
\end{defi}
These are called modifications in \cite{Gra}.

\smallskip
We denote by
$\FH om(\FA,\FB)$ denotes the $2$-category of $2$-functors $\FA\to\FB$.
When $\FB$ is a strict $2$-category, then $\FH om(\FA,\FB)$ is strict as
well.

\smallskip
Given a property of functors, we say that a $2$-functor $F:\FA\to\FB$
has {\em locally}
 that property if the functors $\CHom(a,a')\to\CHom(F(a),F(a'))$
have the property for all $a,a'$ objects of $\FA$.

A $2$-functor $F:\FA\to\FB$ is a {\em $2$-equivalence} if there is
a $2$-functor $G:\FB\to\FA$ and equivalences $\id_{\FA}\iso GF$ and
$FG\iso \id_{\FB}$. This is equivalent to the requirement that $F$ is
locally an equivalence and every object of $\FB$ is equivalent to an
object in the image of $F$.

\smallskip
Every $2$-category is $2$-equivalent to a strict $2$-category, but
there are $2$-functors between strict $2$-categories that are not
equivalent to strict ones.

\smallskip
Given $a$ an object of $\FA$, then
$\CE nd(a)$ is a monoidal category. Conversely, a monoidal category
gives rise to a $2$-category with a single object $\ast$, and the
notion of monoidal functor coincides with that of
$2$-functor (\ie, there is a $1,2,3$-fully faithful strict $3$-functor from
the $3$-category of monoidal categories to that of $2$-categories).

\smallskip
Let $k$ be a commutative ring. A {\em $k$-linear} $2$-category is a
$2$-category $\FA$ that is locally $k$-linear and such that juxtaposition
is $k$-linear. Given $\FA$ and $\FB$ two $k$-linear $2$-categories,
we denote by $\FH om(\FA,\FB)$ the $2$-category of $k$-linear $2$-functors
$\FA\to\FB$: this is the locally full sub-$2$-category of the category of
$2$-functors
obtained by requiring the functors in the definition of
$2$-functors to be $k$-linear.

Given $\FA$ a $2$-category, we denote by $k\FA$ the $k$-linear closure
of $\FA$: its objects are those of $\FA$ and
$\CHom_{k\FA}(a,a')=k\CHom_{\FA}(a,a')$.

\smallskip
Let $b:a\to a'$ be a $1$-arrow. A {\em right adjoint} (or {\em right dual})
 of $b$
is a triple $(b^\vee,\eps_b,\eta_b)$ where $b^\vee:a'\to a$ is a $1$-arrow and
$\eps_b:bb^\vee\to I_{a'}$ and $\eta_b:I_a\to b^\vee b$ are $2$-arrows
such that
the compositions
$$\xymatrix{
b\ar[r]^-{\can} & bI_a\ar[r]^-{b\eta_b} & b(b^\vee b)\ar[r]^-{\can} & (bb^\vee)b
\ar[r]^-{\eps_b b} & I_{a'}b\ar[r]^-{\can} & b}$$
and
$$\xymatrix{
b^\vee \ar[r]^-{\can} & I_ab^\vee\ar[r]^-{\eta_b b^\vee} & (b^\vee b) b^\vee
\ar[r]^-{\can} & b^\vee (b b^\vee) \ar[r]^-{b^\vee\eps_b} & b^\vee I_{a'}
\ar[r]^-{\can} & b^\vee}$$
are identities.
We also say that $(b,\eps_b,\eta_b)$ is a left adjoint (or dual) of
$b'=b^\vee$ (and we write $b={^\vee b}'$)
and we say
that $(b,b^\vee,\eps_b,\eta_b)$ (or simply $(b,b^\vee)$) is an adjoint
quadruple (resp. an adjoint pair).

To simplify the exposition, let us assume for the reminder of
\S \ref{se:def2cat} that $\FA$ is strict.
Let $(b^\vee,\eps_b,\tilde{\eta}_b)$ be a triple such that
$(\eps_b b)\circ (b\tilde{\eta}_b)$ and
$(b^\vee\eps_b)\circ (\tilde{\eta}_b b^\vee)$ are invertible.
Then, 
$(b^\vee,\eps_b,
(b^\vee ((\eps_b b)\circ (b\tilde{\eta}_b)))\circ\tilde{\eta}_b)$ is a
right adjoint of $b$.

\smallskip
Given $b_1:a\to a'$ a $1$-arrow and $(b_1,b_1^\vee)$ an adjoint pair, we have a
canonical isomorphism
$$\Hom(b,b_1)\iso \Hom(b_1^\vee,b^\vee),\
f\mapsto f^\vee=(b^\vee\eps_{b_1})\circ(b^\vee f b_1^\vee)\circ
(\eta_b b_1^\vee).$$

\smallskip
Assume now there are dual pairs $(b,b^\vee)$ and $(b^\vee,b)$.
We have an automorphism
\begin{equation}
\label{eq:Nakayama2cat}
\End(b)\iso \End(b),\ f\mapsto (f^\vee)^\vee.
\end{equation}

\subsubsection{Generators and relations}
An {\em equivalence relation} $\sim$ on $\FA$ is the data for
every $a,a'$ objects, for every $b,b':a\to a'$ of an equivalence
relation on $\Hom(b,b')$ compatible with composition and juxtaposition,
\ie, if $c_1\sim c_2$, then given a $2$-arrow $c$,
we have
$c_1\circ c\sim c_2\circ c$, $c\circ c_1\sim c\circ c_2$,
$c_1 c\sim c_2 c$ and $cc_1\sim cc_2$, whenever this makes sense.
Given a relation $\sim$ on $2$-arrows of $\CC$, the equivalence relation 
generated by $\sim$ is the coarsest refinement of $\sim$ that is an equivalence
relation.

Let $\FA$ be a $2$-category and $\sim$ an equivalence relation.
We denote by $\FA/\!\!\sim$ the $2$-category with same objects as $\FA$ and
with $\CHom_{\FA/\!\sim}(a,a')=\CHom_{\FA}(a,a')/\!\!\sim$ (so,
$\FA/\!\!\sim$ has the same $1$-arrows as $\FA$). The local quotient
functors induce a strict quotient $2$-functor $\FA\to\FA/\!\!\sim$.
 Given
a $2$-category $\FB$, the
quotient strict $2$-functor $\FA\to\FA/\!\!\sim$ induces a strict $2$-functor
$\FH om(\FA/\!\!\sim,\FB)\to \FH om(\FA,\FB)$ that is
locally an isomorphism. A $2$-functor $R$ is
in the image if and only if two equivalent $2$-arrows have the same image
under $R$.

\smallskip
 Given $S$ a set of $2$-arrows of
$\FA$, we denote by $\tS$ the smallest set of $2$-arrows of $\FA$ closed
under juxtaposition and composition
and containing $S$ and the invertible $2$-arrows.

 We denote
by $\FA[S^{-1}]$ the $2$-category with same objects as $\FA$ and with
$\CHom_{\FA[S^{-1}]}(a,a')=\CHom_{\FA}(a,a')[S(a,a')^{-1}]$,
where $S(a,a')$ are the $2$-arrows of $\tS$ that are in $\CHom_{\FA}(a,a')$
(so, $\FA[S^{-1}]$ has the same $1$-arrows as $\FA$).

The canonical strict $2$-functor $\FA\to\FA[S^{-1}]$ induces
a strict $2$-functor $\FH om(\FA[S^{-1}],\FB)\to\FH om(\FA,\FB)$ that is
locally an isomorphism. A $2$-functor $R$
is in the image if and only if the image under $R$
of any $2$-arrow in $S$ is invertible.

\smallskip
Assume $\FA$ is a $k$-linear $2$-category. Let $S$ be a set of $2$-arrows
of $\FA$. Given $a,a'$ objects of $\FA$, we consider the
equivalence relation $\sim_{S(a,a')}$ on $\CHom(a,a')$. Let $\sim$ be the
coarsest equivalence relation on $\FA$
that refines the relations $\sim_{S(a,a')}$.
We put $\FA/S=\FA/\!\!\sim$.

\medskip
A {\em $2$-quiver} $I=(I_0,I_1,I_2,s,t,s_2,t_2)$ is the data of
\begin{itemize}
\item three sets $I_0$ (vertices), $I_1$ ($1$-arrows) and $I_2$ ($2$-arrows)
\item maps $s,t:I_1\to I_0$ (source and target)
\item maps $s_2,t_2:I_2\to \CP=\CP(I_0,I_1,s,t)$ (source and target of
$2$-arrows) such that 
$s(s_2(c))=s(t_2(c))$ and
$t(s_2(c))=t(t_2(c))$ for all $c\in I_2$.
\end{itemize}

Let $I$ be a $2$-quiver. Let $a,a'\in I_0$.
We define a quiver $I(a,a')=(\tI_0,\tI_1,\ts,\tt)$. We put
$\tI_0=(s,t)^{-1}(a,a')$, the set of paths from $a$ to $a'$.
The set $\tI_1$ is given by triples $(b,c,b')$ where
$b,b'\in \CP$, $c\in I_2$ satisfy $t(b')=s(s_2(c))$, $t(s_2(c))=s(b)$,
$s(b')=a$, $t(b)=a'$. We put
$\ts(b,c,b')=bs_2(c)b'$ and $\tt(b,c,b')=bt_2(c)b'$.
We introduce a relation $\sim$ on $\CP(I(a,a'))$ by
$$(b_1t_2(c_1)b_2,c_2,b_3)(b_1,c_1,b_2s_2(c_2)b_3)\sim
(b_1,c_1,b_2t_2(c_2)b_3)(b_1s_2(c_1)b_2,c_2,b_3)$$
(whenever this makes sense).

The strict $2$-category $\FC(I)$ generated by $I$ is
defined as follows. Its set of objects is $I_0$. 
We put $\CHom(a,a')=\CC(I(a,a'))/\!\!\sim$. Composition of $1$-arrows is
concatenation of paths. Juxtaposition is given by
$$(b_1,c_1,b'_1)(b_2,c_2,b'_2)=
(b_1,c_1,b'_1b_2t_2(c_2)b'_2)\circ (b_1s_2(c_1)b'_1b_2,c_2,b'_2).$$

Note that the category $\FC(I)_{\le 1}$ is $\CC(I_0,I_1,s,t)$.

\medskip
Let $\FB$ be a strict $2$-category. An {\em $I$-diagram}
$D$ in $\FB$ is the data of 
\begin{itemize}
\item an object $a_i$ of $\FB$ for any $i\in I_0$
\item a $1$-arrow $b_j:a_{s(j)}\to a_{t(j)}$ for any $j\in I_1$
\item a $2$-arrow $c_k:b_{s_2(k)}\to b_{t_2(k)}$ for any $k\in I_2$
\end{itemize}
where given $p=(p_1,\ldots,p_n)\in\CP$, we put
$b_p=b_{p_1}\cdots b_{p_n}$.
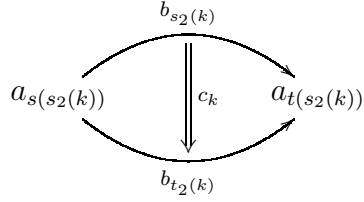
$$\xymatrix{a_{s(s_2(k))}\ar@/^2pc/[rr]^{b_{s_2(k)}}_{}="1"
\ar@/_2pc/[rr]_{b_{t_2(k)}}^{}="2"
\ar@{=>}"1";"2" ^{c_k}&& a_{t(s_2(k))}}$$
The data of $b_j$'s and $c_k$'s is the same as the data, for $i,i'\in I_0$,
of an $I(i,i')$-diagram in $\CHom(a_i,a_{i'})$.

\medskip
A morphism $\sigma:D\to D'$ is the data of
\begin{itemize}
\item $1$-arrows $\sigma_i:a_i\to a'_i$ for $i\in I_0$
\item invertible $2$-arrows $\sigma_j:b'_j\sigma_{s(j)}\iso
\sigma_{t(j)}b_j$ for every $j\in I_1$
\end{itemize}
$$
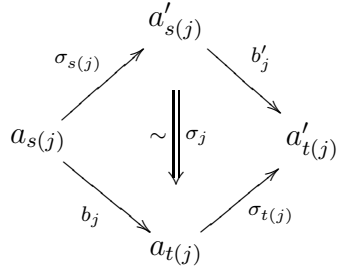
\xymatrix{
& a'_{s(j)} \ar[dr]^{b'_j}
\ar@{}[dd]|(.3){}="3"
\ar@{}[dd]|(.7){}="4"
 \\
a_{s(j)} \ar[dr]_{b_j} \ar[ur]^{\sigma_{s(j)}} 
&& a'_{t(j)} \\
& a_{t(j)} \ar[ur]_{\sigma_{t(j)}}
\ar@{=>}"3";"4"^{\sigma_j}_\sim
}$$
such that for every $k\in I_2$
with $s_2(k)=(j_1,\ldots,j_n)$ and 
$t_2(k)=(\bar{j}_1,\ldots,\bar{j}_{\bar{n}})$,
the following $2$-arrows $b'_{s_2(k)}\sigma_{s(j_n)}\to
\sigma_{t(j_1)}b_{t_2(k)}$ are equal:
$$
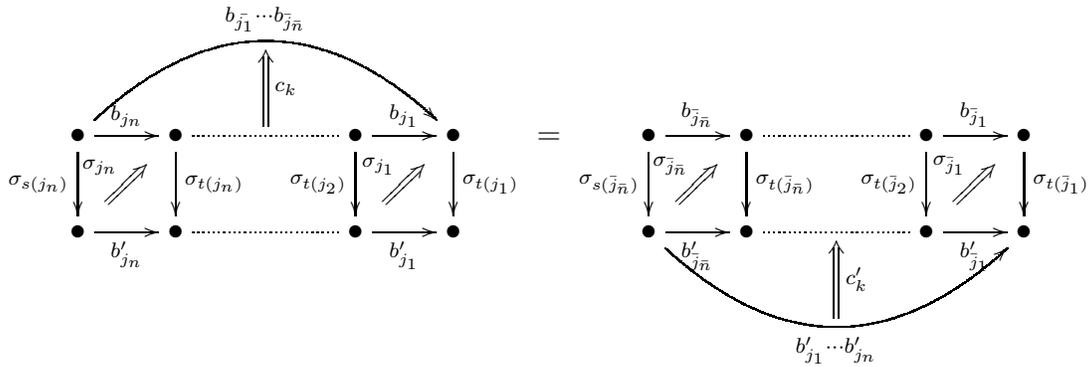
\xymatrix{
\bullet\ar[d]_{\sigma_{s(j_n)}} \ar[r]^{b_{j_n}} 
\ar@/^3pc/[rrrr]^{b_{\bar{j_1}}\cdots b_{\bar{j}_{\bar{n}}}}^{}="4" &
 \bullet\ar[d]^{\sigma_{t(j_n)}} \ar@{.}[rr]^{}="3" &&
 \ar@{=>}"3";"4"_{c_k}
 \bullet\ar[d]_{\sigma_{t(j_2)}}\ar[r]^{b_{j_1}} &
 \bullet\ar[d]^{\sigma_{t(j_1)}} \\
\bullet \ar[r]_{b'_{j_n}}
\ar@{}[ur]|(.3){}="1" \ar@{}[ur]|(.7){}="2" \ar@{=>}"1";"2"^-{\sigma_{j_n}}  &
 \bullet\ar@{.}[rr] &&
 \bullet\ar[r]_{b'_{j_1}}
\ar@{}[ur]|(.3){}="1" \ar@{}[ur]|(.7){}="2" \ar@{=>}"1";"2"^-{\sigma_{j_1}}  &
 \bullet 
}=
\xymatrix{
\bullet\ar[r]^{b_{\bar{j}_{\bar{n}}}}\ar[d]_{\sigma_{s(\bar{j}_{\bar{n}})}} &
 \bullet\ar[d]^{\sigma_{t(\bar{j}_{\bar{n}})}} \ar@{.}[rr] &&
 \bullet\ar[r]^{b_{\bar{j}_1}} \ar[d]_{\sigma_{t(\bar{j}_2)}} &
 \bullet\ar[d]^{\sigma_{t(\bar{j}_1)}} \\
\bullet\ar[r]_{b'_{\bar{j}_{\bar{n}}}} 
\ar@{}[ur]|(.3){}="1" \ar@{}[ur]|(.7){}="2" 
 \ar@{=>}"1";"2"^-{\sigma_{\bar{j}_{\bar{n}}}}
  \ar@/_3pc/[rrrr]_{b'_{j_1}\cdots b'_{j_n}}^{}="4" &
 \bullet \ar@{.}[rr]^{}="3" \ar@{=>}"4";"3"_{c'_k}&&
 \bullet\ar[r]_{b'_{\bar{j}_1}} 
\ar@{}[ur]|(.3){}="1" \ar@{}[ur]|(.7){}="2" 
\ar@{=>}"1";"2"^-{\sigma_{\bar{j}_1}}
& \bullet 
}$$

\smallskip
A morphism $\gamma:\sigma\to\tilde{\sigma}$ is the data of $2$-arrows
$\gamma_i:\sigma_i\to\tilde{\sigma}_i$ for $i\in I_0$ such that
for every $j\in I_1$, we have 
$(\gamma_{t(j)}b_j)\circ\sigma_j=\tilde{\sigma}_j\circ(b'_j\gamma_{s(j)})$,
\ie, the following diagram of $2$-arrows is commutative:
$$
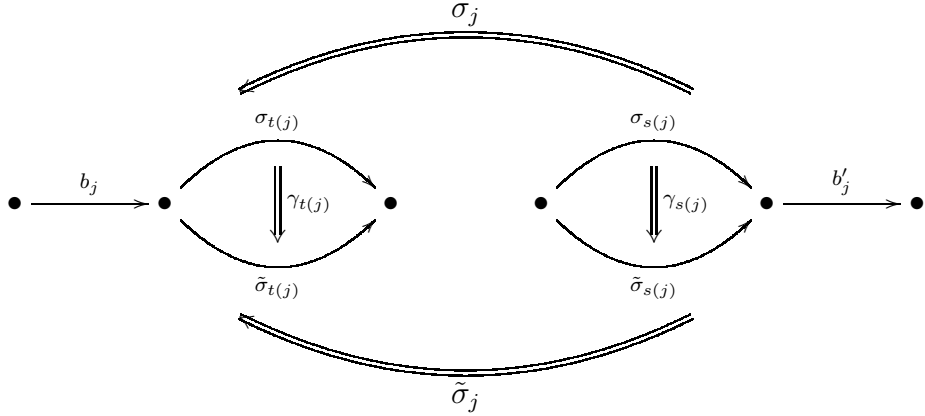
\xy
(-60,0) *+{\bullet}="1",
(-40,0) *+{\bullet}="2",
(-10,0) *+{\bullet}="3",
(10,0)  *+{\bullet}="4",
(40,0)  *+{\bullet}="5",
(60,0)  *+{\bullet}="6",
 (30,15)*{};(-30,15)*{} **\crv{~**\dir2{-}(0,30)}?*+!D{\sigma_j}
 ?>*\dir2{>},
 (30,-15)*{};(-30,-15)*{} **\crv{~**\dir2{-}(0,-30)}?*+!U{\tilde{\sigma}_j}
 ?>*\dir2{>},
 \ar^{b_j}"1";"2",
 \ar@/^2pc/^{\sigma_{t(j)}}_{}"2";"3",
 \ar@/_2pc/_{\tilde{\sigma}_{t(j)}}^{}"2";"3",
 \ar@{=>}^{\gamma_{t(j)}} (-25,5)*{};(-25,-5)*{},
 \ar@/^2pc/^{\sigma_{s(j)}}"4";"5",
 \ar@/_2pc/_{\tilde{\sigma}_{s(j)}}"4";"5",
 \ar@{=>}^{\gamma_{s(j)}} (25,5)*{};(25,-5)*{},
 \ar^{b'_j}"5";"6",
\endxy$$
This gives rise to a strict $2$-category $\FH om(I,\FB)$
of $I$-diagrams in $\FB$.

\smallskip
Restriction gives a strict $2$-functor $H:\FH om(\FC(I),\FB)\to \FH om(I,\FB)$.
It is locally an isomorphism and it is surjective on objects, so it
is a $2$-equivalence.

\subsubsection{$2$-Representations}
Let $\FA$ and $\FB$ be two $2$-categories.
We will consider $2$-representations of
$\FA$ in $\FB$, \ie, $2$-functors 
$R:\FA\to\FB$.
We put $\FA\mMOD(\FB)=\FH om(\FA,\FB)$, a $2$-category. Given 
$R:\FA\to\FB$, a {\em sub-$2$-representation} is a $2$-functor
$R':\FA\to\FB$ equiped with a fully faithful morphism
$R'\to R$.
There is a canonical $2$-equivalence
$\FH om(\FA^\opp,\FB^\opp)\iso \FH om(\FA,\FB)^\opp$.

\smallskip
Let $S$ be a collection of objects of $\FB$. An {\em action} of
$\FA$ on $S$ is a $2$-representation of $\FA$ in $\FB$ with image contained
in $S$. Note that if $\FA$ has only one object and is viewed as a monoidal
category $\CA$ and $S=\{\CC\}$, we recover the usual
notion of an action of $\CA$ on $\CC$.

\smallskip
Let $a\in\CA$. We define a $2$-functor $\CHom(a,-):\FA\to\FC at$ by
$a'\mapsto \CHom(a,a')$. The functor
$\CHom(a',a'')\to \CHom(\CHom(a,a'),\CHom(a,a''))$ is given by juxtaposition.
The associativity and unit maps of $\FA$ provide the required $2$-arrows.

\smallskip
Let $R:\FA\to\FC at$ be a $2$-functor. 
Given $a$ an object of $\FA$, there is an equivalence of categories
from $R(a)$ to the category of morphisms $\CHom(a,-)\to R$:
\begin{itemize}
\item Given
$M$ an object of the category $R(a)$, we define a morphism
$\sigma:\CHom(a,-)\to R$. The functor
$\CHom(a,a')\to R(a')$ is $b\mapsto R(b)(M)$. The required natural isomorphisms
come from the natural isomorphisms $R(b)R(f)\iso R(bf)$.
\item
Conversely, given $\sigma:\CHom(a,-)\to R$, we put $M=\sigma(I_a)$.
\end{itemize}

\smallskip
Assume from now on that our $2$-categories are $k$-linear.

Let $b:a\to a'$ be a $1$-arrow. A {\em cokernel} of $b$
is the data of an object $\CC oker(b)$ and of a $1$-arrow $b':a'\to \CC oker(b)$
such that for any object $a''$, the functor
$\CHom(b',a''):\CHom(\CC oker(b),a'')\to \CHom(a',a'')$ is fully
faithful with image equivalent to the full subcategory of
$1$-arrows $b'':a'\to a''$ such that $b''b=0$. When a cokernel of $b$
exists, it is unique up to an equivalence unique up to a unique isomorphism.

We say that
$\FA$ {\em admits cokernels} if all $1$-arrows admit cokernels.
This is the case for the $2$-category of
$k$-linear categories, of abelian categories or of triangulated
categories.

We define {\em kernels} as cokernels taken in $\FA^{\mathrm{rev}}$.

\smallskip
Assume $\FB$ admits kernel and cokernels.

Let $b:a\to a'$ be a fully faithful $1$-arrow. We say that it
is {\em thick} if $b$ is a kernel of $a\to\CC oker(b)$.

When $\FB\subset \FL in_k$, the notion of thickness corresponds to
\begin{itemize}
\item $\FL in_k$ or $\FT ri$: $a$ is closed under direct summands
\item $\FA b$: $a$ is closed under extensions, subobjects and quotients
\end{itemize}

\smallskip
Let
$R,R':\FA\to\FB$ be two $2$-functors and
$\sigma:R'\to R$.
Assume $\sigma$ is locally fully faithful.
We define
$R''=
\CC oker\ \sigma:\FA\to\FB$ (denoted also by $R/R'$ when there is no ambiguity) 
by $R'':a\mapsto \CC oker\ \sigma(a)$. The composition
$\CH om(a,a')\xrightarrow{R(a,a')} \CH om(R(a),R(a'))\xrightarrow{\can}
\CH om(R(a),\CC oker\ \sigma(a'))$ factors uniquely through
$\CH om(\CC oker\ \sigma(a),\CC oker\ \sigma(a'))$ and this defines 
a functor $\CH om(a,a')\to\CH om(\CC oker\ \sigma(a),\CC oker\ \sigma(a'))$.
The
constraints are obtained by taking quotients.
$$\xymatrix{
&&& \CHom(\CC oker\ \sigma(a),\CC oker\ \sigma(a'))\ar[d] \\
&\CHom(a,a') \ar[r]\ar[d] \ar@{-->}[urr] 
 \ar `l[ld] `[dd] `[rr]^0 [drr]&
 \CHom(R(a),R(a')) \ar[r] \ar[d] &
 \CHom(R(a),\CC oker\ \sigma(a')) \ar[d] \\
&\CHom(R'(a),R'(a')) \ar[r] 
 \ar@{}[ru]|(.3){}="1" \ar@{}[ru]|(.7){}="2" \ar@{=>}"1";"2"^\sim&
 \CHom(R'(a),R(a')) \ar[r] &
 \CHom(R'(a),\CC oker\ \sigma(a')) \\
 &&&&
}$$

\smallskip
We have a Grothendieck group functor $K_0:{\FT ri}_{\le 1}\to\CA b$. When
$\FB$ is endowed with a canonical $2$-functor to the
$2$-category of triangulated categories, we will still denote
by $K_0$ the composite functor $\FB_{\le 1}\to\Ab$. For example,
$\FB$ is the category of exact categories or of dg-categories
and we consider the derived category $2$-functor. Viewing additive
categories as exact categories for the split structure provides another
example (this is the homotopy category functor).
This gives a ``decategorification''
functor $\FA\mMOD(\FB)_{\le 1}\to \CH om(\FA_{\le 1},\CA b)$.

\smallskip
Let $\FA$ and $\FB$ be $k$-linear $2$-categories. Assume $\FB$ is locally
idempotent-complete. Let $\FA^i$ be the idempotent completion
of $\FA$. The canonical strict $2$-functor $\FA\to\FA^i$ induces a
$2$-equivalence $\FA^i\mMOD(\FB)\iso \FA\mMOD(\FB)$.

\subsection{Symmetric algebras}
\label{se:symmalg}
The theory of symmetric or Frobenius algebras
is classical (cf eg \cite{Bro}).
We need here a version over a non-commutative base algebra and
we study transitivity properties.

\subsubsection{Serre functors}
Let $k$ be a field and $\CT_1$, $\CT_2$ be two $k$-linear categories.
Let $E:\CT_1\to\CT_2$ be a functor and
$(E,F)$ an adjoint pair, provided with bifunctorial isomorphisms
$$\alpha(M,N):
\Hom(EM,N)\iso \Hom(M,FN)\text{ for } M\in\CT_1 \text{ and }N\in\CT_2$$

Let $S_i$ be a Serre functor for $\CT_i$, for $i=1,2$:
we have bifunctorial isomorphisms
$$\gamma_i(M,N):\Hom(M,N)^*\iso\Hom(N,S_iM)\text{ for } M,N\in\CT_i.$$

Then, $(S_2^{-1}FS_1,E)$ is an adjoint pair with defining isomorphisms
given by the following commutative diagram
$$\xymatrix{
\Hom(S_2^{-1}FS_1N,M) \ar[rr]^-{\sim} && \Hom(N,EM) \\
\Hom(M,FS_1N)^*\ar[u]^{\gamma_2(S_2^{-1}FS_1N,M)^*}_\sim
\ar[rr]^-\sim_-{\alpha(M,S_1N)^*} && \Hom(EM,S_1N)^*
\ar[u]^\sim_{\gamma_1(N,EM)^*}
}$$

\begin{lemma}
Let $(E',F')$ be an adjoint pair, with $E':\CT_1\to\CT_2$. Given
$f\in\Hom(E,E')$, we have
${^\vee f}=S_2^{-1}f^\vee S_1$.
\end{lemma}

\begin{proof}
Given $M\in\CT_1$ and $N\in\CT_2$, we have a commutative diagram
$$\xymatrix{
\Hom(N,EM)\ar[d]^{\Hom(N,fM)} &
\Hom(EM,S_1N)^*\ \ \ar[l]_-{\gamma_1^*}^-\sim\ar[d]^{\Hom(fM,S_1N)^*} & 
\Hom(M,FS_1N)^* \ar[r]^-{\gamma_2^*}_-\sim\ar[d]^{\Hom(M,f^\vee S_1N)^*} 
\ar[l]_-{\alpha^*}^-\sim&
 \Hom(S_2^{-1}FS_1N,M)\ar[d]^{\Hom(S_2^{-1}f^\vee S_1N,M)} \\
\Hom(N,E'M) &
\Hom(E'M,S_1N)^* 
\ar[l]^-{\gamma_1^*}_-\sim& 
\Hom(M,F'S_1N)^*\ar[r]_-{\gamma_2^*}^-\sim
\ar[l]^-{\alpha^{\prime *}}_-\sim
 &
 \Hom(S_2^{-1}F'S_1N,M) 
}$$
and the result follows.
\end{proof}

\subsubsection{Frobenius forms}
Let $B$ be a $k$-algebra and $A$ a $B$-algebra.
We denote by $m:A\otimes_B A\to A$ the multiplication map.

The canonical isomorphism of $(A,B)$-bimodules
$\Hom_B(A,B)\iso\Hom_A(A,\Hom_B(A,B))$ restricts to an isomorphism
$$t\mapsto \hat{t}:\Hom_{B,B}(A,B)\iso\Hom_{A,B}(A,\Hom_B(A,B)).$$

Let us describe this explicitely.
Given $t:A\to B$ a morphism of $(B,B)$-bimodules, we have the
morphism of $(A,B)$-bimodules
\begin{align*}
\hat{t}:A&\to \Hom_B(A,B)\\
a&\mapsto (a'\mapsto t(a'a)).
\end{align*}

Conversely, given $f:A\to \Hom_B(A,B)$ a morphism of $(A,B)$-bimodules,
then $f(1):A\to B$ is a morphism of $(B,B)$-bimodules and
we have $f=\widehat{f(1)}$.

\begin{defi}
Let $t:A\to B$ be a morphism of $(B,B)$-bimodules. We say that $t$ is a 
Frobenius form if $A$ is a projective $B$-module of finite type and
$\hat{t}:A\to \Hom_B(A,B)$ is an isomorphism.
\end{defi}

Let $t:A\to B$ be a Frobenius form. It defines an
automorphism of $Z(B)$-algebras, the Nakayama automorphism:
$$\gamma_t:A^B\iso A^B,\ a\mapsto \hat{t}^{-1}\left(a'\mapsto t(aa')\right).$$
We have
$$t(aa')=t(a'\gamma_t(a)) \text{ for all }a\in A^B \text{ and }a'\in A.$$
This makes $\hat{t}$ into an isomorphism of $(A,B\otimes_{Z(B)} A^B)$-modules
$$\hat{t}:A_{1\otimes\gamma_t}\iso \Hom_B(A,B).$$

\smallskip
We say that $t$ is {\em symmetric} if $\gamma_t=\id_{A^B}$.

\begin{rem}
Note that
 if $t(aa')=t(a'a)$ for all $a,a'\in A$, then $A^B=A$.
\end{rem}

Given $t$ and $t'$ two Frobenius forms, there is a unique element 
$z\in (A^B)^\times$ such that $t'(a)=t(az)$ for all $a\in A$. If in addition
$t$ and $t'$ are symmetric, then $z\in Z(A^B)^\times$.

\subsubsection{Adjunction $(\Res,\Ind)$}
Let $B$ be a $k$-algebra and $A$ a $B$-algebra.

The data of an adjunction $(\Res_B^A,\Ind_B^A)$ is the same as the data
of an isomorphism $A\otimes_B-\iso \Hom_B(A,-)$ of functors
$B\mMOD\to A\mMOD$.

Assume there is such an adjunction. The functor $\Hom_B(A,-)$ is
right exact, hence $A$ is projective as a $B$-module.
The functor $\Hom_B(A,-)$ commutes with direct sums,
hence $A$ is a finitely generated projective $B$-module.

Assume now $A$ is a finitely generated projective $B$-module. We have a
canonical isomorphism
$$\Hom_B(A,B)\otimes_B -\iso \Hom_B(A,-).$$
So, the data of an adjunction $(\Res_B^A,\Ind_B^A)$ is the same as the data
of an isomorphism $f:A\iso \Hom_B(A,B)$ of $(A,B)$-bimodules.
Given $f$, let $t=f(1):A\to B$. This is the morphism of $(B,B)$-bimodules
corresponding to the counit $\eps:\Res_B^A\Ind_B^A\to \id_B$. On the
other hand, we have $f=\hat{t}$. Summarizing,
we have the following Proposition.

\begin{prop}
Let $B$ be an algebra and $A$ a $B$-algebra. We have inverse
bijections between the set of Frobenius forms and the set of adjunctions
$(\Res_B^A,\Ind_B^A)$:
\begin{align*}
t& \mapsto \text{ adjunction defined by } \hat{t} \\
\text{counit } & \mapsfrom \text{ adjunction}
\end{align*}
\end{prop}

Assume we have a Frobenius form $t:A\to B$.
The unit of adjunction of the pair $(\Res_B^A,\Ind_B^A)$ corresponds
to a morphism of $(A,A)$-bimodules $A\to A\otimes_B A$. The image of $1$
under this morphism is the {\em Casimir element}
$\pi=\pi^A_B\in (A\otimes_B A)^A$.
It satisfies
\begin{equation}
\label{eq:charpi}
(t\otimes 1)(\pi)=(1\otimes t)(\pi)=1\in A.
\end{equation}
Conversely, given an element $\pi\in (A\otimes_B A)^A$, there
exists at most one $t\in\Hom_{B,B}(A,B)$ satisfying (\ref{eq:charpi}),
and such a morphism is a Frobenius form.

\smallskip
Note that right multiplication induces an isomorphism
$A^B\iso \End(\Ind_B^A)$ and the automorphism (\ref{eq:Nakayama2cat})
is the Nakayama automorphism $\gamma_t$.

\begin{rem}
We developed the theory for left modules, but this is the same as the
theory for right modules. Namely, let $t:A\to B$ be a Frobenius form. Since
$A$ is finitely generated and projective as a $B$-module, it follows that
$\Hom_B(A,B)$ is a finitely generated projective right $B$-module, hence
$A$ is a finitely generated projective right $B$-module.
Consider the composition
$$\check{t}:A\xrightarrow{a\mapsto (\zeta\mapsto \zeta(a))}
\Hom_{B^\opp}(\Hom_B(A,B),B)\xrightarrow{\Hom_{B^\opp}(\hat{t},B)}
\Hom_{B^\opp}(A,B),\ a\mapsto (a'\mapsto t(aa')).$$
The first map is an isomorphism since $A$ is finitely generated and projective
as a $B$-module. It follows that $\check{t}$ is an isomorphism.
\end{rem}

\subsubsection{Transitivity}
Let $C$ be an algebra, $B$ a $C$-algebra and $A$ a $B$-algebra. We assume that
$A$ (resp. $B$)
is a finitely generated projective $B$-module (resp. $C$-module).

Given $t\in\Hom_{B,B}(A,B)$, $t'\in\Hom_{C,C}(B,C)$ and
$t''=t'\circ t\in\Hom_{C,C}(A,C)$, we have a commutative diagram
$$\xymatrix{
A\ar[rrr]^-{\hat{t}''}\ar[d]_-{\hat{t}} &&& \Hom_C(A,C) \ar[d]_{\sim}^\can \\
\Hom_B(A,B)\ar[urrr]^{\Hom(A,t')} \ar[rrr]_-{\Hom(A,\hat{t}')} &&&
\Hom_B(A,\Hom_C(B,C))
}$$

The units of adjunction are given by composition:
$$A\xrightarrow{1\mapsto \pi_B^A} A\otimes_B A\xrightarrow{1\otimes 1\mapsto
1 \pi_C^B 1}A\otimes_C A,\ 1\mapsto \pi_C^A.$$

\begin{lemma}
\label{le:trans1}
If $t\in\Hom_{B,B}(A,B)$ and $t'\in\Hom_{C,C}(B,C)$ are Frobenius forms,
then $t'\circ t:A\to C$ is a Frobenius form.
\end{lemma}

\begin{lemma}
\label{le:trans2}
Let $t'\in\Hom_{C,C}(B,C)$ and $t''\in\Hom_{C,C}(A,C)$ be Frobenius forms.
There is a unique $t\in\Hom_B(A,B)$ such that $t''=t'\circ t$. It
is a Frobenius form and it is given by
$t=\Hom_B(A,\hat{t}')^{-1}(\hat{t}''(1))\in \Hom_{B,B}(A,B)$.
\end{lemma}

Let $t''\in\Hom_{C,C}(A,C)$ and $\zeta\in A^C$. Define
$t'\in\Hom_{C,C}(B,C)$ by $t'(b)=t''(b\zeta)$. If $t''$ is a Frobenius
morphism and the pairing
$$B\times B\to C,\ (b,b')\mapsto t''(bb'\zeta)$$
is perfect, then $t'$ is a Frobenius form.

Assume now $t$, $t'$ and $t''$ are given and let $\zeta\in A^C$. Then,
$$t(\zeta)=1 \Leftrightarrow \forall b\in B,\ t'(bt(\zeta))=t'(b)
\Leftrightarrow \forall b\in B,\ t''(b\zeta)=t'(b).$$
Note that $\zeta$ is determined by $t'$ up to adding an element
$\xi\in A^C$ such that $t''(B\xi)=0$. The next lemma shows that under certain
conditions on $A$, the form $t'$ is always obtained from such a $\zeta$.

\begin{lemma}
\label{le:trans3}
Assume $B$ is a quotient of $A$ as a $(B,C)$-bimodule
(this is the case if $A$ is a progenerator for $B$ and $C\subset Z(A)$). 
Let $t\in\Hom_{B,B}(A,B)$ and $t''\in\Hom_{C,C}(A,C)$ be Frobenius forms.
There is a unique $t'\in\Hom_{C,C}(B,C)$ such that $t''=t'\circ t$. It
is a Frobenius form.
\end{lemma}

\begin{proof}
Since $A$ is a progenerator for $B$, the morphism $\hat{t}'$ is determined by
$\Hom_B(A,\hat{t}')$. The unicity of $t'$ follows.

Assume $A$ is a progenerator for $B$ and $C$ is central in $A$.
Since $A$ is a progenerator for $B$, there exists an integer $n$ and
a surjection of $B$-modules $f:A^n\to B$. Let $m\in f^{-1}(1)$ and
consider the morphism $A\to A^n,\ a\mapsto am$. The composition
$g:A\to A^n\to B$ is a morphism of $B$-modules with $g(1)=1$. Since $C$ is
central, $g$ is a morphism of $(B,C)$-bimodules.

Assume now there is a surjective morphism of $(B,C)$-bimodules $h:A\to B$.
Then, $h(1)\in Z(C)^\times$. let $g:A\to B,\ a\mapsto ah(1)^{-1}$. This
is a morphism of $(B,C)$-bimodules with $g(1)=1$.

Let $\zeta=\hat{t}^{-1}(g)$. We have $t(\zeta)=1$ and we define $t'$ by
$t'(b)=t''(b\zeta)$.
We have $t''=t'\circ t$, the morphism $\Hom_B(A,\hat{t}')$ is invertible
and since $A$ is a progenerator for $B$, it follows that $\hat{t}'$ is
an isomorphism.
\end{proof}

\subsubsection{Bases}
Let $B$ be an algebra, $A$ a $B$-algebra and assume $A$ is free of
finite rank as a $B$-module. Let $\CB$ be a basis of $A$ as a left
$B$-module: $A=\bigoplus_{v\in\CB}Bv$. 

\smallskip
Let $t\in\Hom_{B,B}(A,B)$. Then, $t$ is a Frobenius form if and only
if there exists a {\em dual basis} $\CB^\vee=\{v^\vee\}_{v\in\CB}$: \ie,
$\CB^\vee$ 
satisfies $t(v'v^\vee)=\delta_{vv'}$ for $v,v'\in \CB$.

\smallskip
 Assume $t$ is
a Frobenius form. Then, $\CB^\vee$ exists and is unique. It is a basis
of $A$ as a right $B$-module.
We have
$$\hat{t}(v^\vee)=(\CB\ni v'\mapsto \delta_{v,v'}) \text{ for }v\in\CB.$$
Given $a\in A$, we have
$$a=\sum_{v\in\CB}t(av^\vee)v=\sum_{v\in\CB}v^\vee t(va).$$
Given $a\in A^B$, we have
$$\gamma_t(a)=\sum_{v\in\CB}v^\vee t(av).$$
The unit of the adjoint pair $(\Res_B^A,\Ind_B^A)$ is given by the
morphism of $(A,A)$-bimodules
$$A\to A\otimes_B A,\ 1\mapsto \pi_B^A=\sum_{v\in\CB}v^\vee\otimes v.$$

Consider now $C$ an algebra and a $C$-algebra structure on $B$ such that
$B$ is free of finite
rank as a $C$-module. Let $\CB'$ be a basis of $B$ as a $C$-module. Then,
$\CB''=\CB'\CB=\{v'v\}_{v\in\CB,v'\in\CB'}$ is a basis of $A$ as a $C$-module.

\smallskip
Let $t':B\to C$ be a Frobenius form. The dual basis to $\CB''$ for
the Frobenius form $t''=t'\circ t:A\to C$ is
$\CB^{\prime\prime\vee}=\{v^\vee v^{\prime\vee}\}_{v\in\CB,v'\in\CB'}$.
Given $a\in A$, we have
$$t(a)=\sum_{v'\in\CB'} t''(av^{\prime\vee})v'=\sum_{v'\in\CB'}v^{\prime\vee}
t''(v'a).$$
Given $v\in\CB$, we have
$$v^\vee=\sum_{v'\in\CB'} (v'v)^\vee t'(v').$$

\subsubsection{Ramification}
\label{se:ramification}
Let $A$ be a $B$-algebra endowed with a Frobenius form $t$ and assume
$A^B=A$.

The following statements are equivalent:
\begin{itemize}
\item[(a)] $A$ is a projective $(A\otimes_B A^\opp)$-module
\item[(b)] there exists $a\in A$ such that $m((1\otimes a\otimes 1\otimes 1)\pi)=1$
\item[(c)] there exists $a\in A$ such that $m((1\otimes 1\otimes a\otimes 1)\pi)=1$
\end{itemize}
where
$A\otimes_B A$ is viewed as a module over $\left((A\otimes A^\opp)\otimes_B
(A\otimes A^\opp)\right)$.

\smallskip
When $A$ is commutative, the statements (a)-(c)
above are equivalent to the following
two statements
\begin{itemize}
\item[(d)] $A$ is \'etale over $B$
\item[(e)] $m(\pi)\in A^\times$.
\end{itemize}

\section{Hecke algebras}
\subsection{Classical Hecke algebras}
\label{se:classicalHecke}
We recall in this section the various versions of affine Hecke algebras and the isomorphisms between them after suitable localizations. We consider only
the case of $\GL_n$: in this case, the inclusion 
$\BG_m^n\hookrightarrow \BG_a^n$ gives an algebraic $\GS_n$-equivariant
map that makes it possible to avoid completions. In general, one needs
to use the expotential map from the Lie algebra of a torus to the torus.
All constructions and results in this section extend to arbitrary Weyl groups.

\subsubsection{BGG-Demazure operators}
Given $1\le i\le n$,
we put $s_i=(i,i+1)\in\GS_n$. We
define an endomorphism of abelian groups
$\partial_i\in\End_\BZ(\BZ[X_1,\ldots,X_n])$ by
$$\partial_i(P)=\frac{P-s_i(P)}{X_{i+1}-X_i}.$$
The formula defines endomorphisms of various localizations, for
example $\BZ[X_1^{\pm 1},\ldots,X_n^{\pm 1}]$.

Given $w=s_{i_1}\cdots s_{i_r}$ a reduced decomposition of an element of
$\GS_n$, we put
$$\partial_w=\partial_{i_1}\cdots\partial_{i_r}.$$
This is independent of the choice of the reduced decomposition.

\smallskip
The $\BZ[X_1,\ldots,X_n]^{\GS_n}$-linear morphism $\partial_{w[1,n]}$ takes
values in $\BZ[X_1,\ldots,X_n]^{\GS_n}$. It
is a symmetrizing form for the $\BZ[X_1,\ldots,X_n]^{\GS_n}$-algebra
$\BZ[X_1,\ldots,X_n]$. We view $\BZ[X_1,\ldots,X_n]$ as a graded algebra
with $\deg(X_i)=2$. Then, $\partial_{w[1,n]}$ is homogeneous of
degree $-n(n-1)$.

\begin{lemma}
\label{le:Casimir}
Denote by $\pi$ the Casimir element for $\partial_{w[1,n]}$. Then
$m(\pi)=\prod_{1\le j<i\le n}(X_i-X_j)$.
\end{lemma}

\begin{proof}
The algebra $\BZ[X_1,\ldots,X_n]$ is \'etale over
$\BZ[X_1,\ldots,X_n]^{\GS_n}$ outside $m(\pi)=0$. So,
$\prod_{1\le j<i\le n}(X_i-X_j)~|~m(\pi)$ (cf \S \ref{se:ramification}). Since
$m(\pi)$ is homogeneous of
degree $n(n-1)$, it follows that there is $a\in\BZ$ such that
$m(\pi)=a\prod_{1\le j<i\le n}(X_i-X_j)$. On the other hand,
$\partial_{w[1,n]}(m(\pi))=n!=\partial_{w[1,n]}
\left(\prod_{1\le j<i\le n}(X_i-X_j)
\right)$ and the lemma follows.
\end{proof}

\medskip
Let $A=\BZ[X_1,\ldots,X_n]\rtimes\GS_n$. This algebra
has a Frobenius form over $\BZ[X_1,\ldots,X_n]$ given by
$$Pw\mapsto P\delta_{w\cdot w[1,n]} \text{ for }
P\in\BZ[X_1,\ldots,X_n] \text{ and }w\in \GS_n.$$
By composition, we obtain a Frobenius form $t$ for $A$ over
$\BZ[X_1,\ldots,X_n]^{\GS_n}$ given by
$$t(Pw)=\partial_{w[1,n]}(P)\delta_{w\cdot w[1,n]} \text{ for }
P\in\BZ[X_1,\ldots,X_n] \text{ and }w\in \GS_n.$$
The corresponding Nakayama automorphism of $A$ is the involution
$$X_i\mapsto X_{n-i+1},\ s_i\mapsto -s_{n-i}.$$

\subsubsection{Degenerate affine Hecke algebras}
Let $\bar{H}_n$ be the {\em degenerate affine Hecke algebra} of $\GL_n$:
$\bar{H}_n=\BZ[X_1,\ldots,X_n]\otimes \BZ\GS_n$ as an abelian group,
$\BZ[X_1,\ldots,X_n]$ and $\BZ\GS_n$ are subalgebras and
$$T_i X_j=X_j T_i \text{ if } j-i\not=0,1 \text{ and }
T_iX_{i+1}-X_iT_i=1.$$
We denote here by $T_1,\ldots,T_{n-1}$ the Coxeter generators for $\GS_n$
and we write $T_w$ for the element $w$ of $\GS_n$.

Given $P\in \BZ[X_1,\ldots,X_n]$, we have
$T_i P-s_i(P)T_i=\partial_i(P)$.

\smallskip
We have a faithful representation on $\BZ[X_1,\ldots,X_n]=
\bar{H}_n\otimes_{\BZ\GS_n}\BZ$ where
$$T_i(P)=s_i(P)+\partial_i(P).$$
Here, $\BZ$ is the trivial representation of $\GS_n$.

\smallskip
The algebra $\bar{H}_n$ has a Frobenius form over $\BZ[X_1,\ldots,X_n]$
given by 
\begin{equation}
\label{eq:FroboverP}
PT_w\mapsto P\partial_{w\cdot w[1,n]} \text{ for }
P\in\BZ[X_1,\ldots,X_n] \text{ and }w\in \GS_n.
\end{equation}
By composition, we obtain a Frobenius form $t$ for $\bar{H}_n$ over
$\BZ[X_1,\ldots,X_n]^{\GS_n}$ given by
\begin{equation}
\label{eq:FroboverPinv}
t(PT_w)=\partial_{w[1,n]}(P)\delta_{w\cdot w[1,n]} \text{ for }
P\in\BZ[X_1,\ldots,X_n] \text{ and }w\in \GS_n.
\end{equation}
The corresponding Nakayama automorphism of $\bar{H}_n$ is the involution
$$X_i\mapsto X_{n-i+1},\ T_i\mapsto -T_{n-i}.$$

\subsubsection{Finite Hecke algebras}
Let $R=\BZ[q^{\pm 1}]$.
Let $H_n^f$ be the {\em Hecke algebra} of $\GL_n$: this
is the $R$-algebra generated by $T_1,\ldots,T_{n-1}$,
with relations
$$T_iT_{i+1}T_i=T_{i+1}T_iT_{i+1},\ \ T_iT_j=T_jT_i \text { if }
|i-j|>1 \text { and } (T_i-q)(T_i+1)=0.$$

\smallskip
Given $w=s_{i_1}\cdots s_{i_r}$ a reduced decomposition of an element
$w\in\GS_n$, we put $T_w=T_{i_1}\cdots T_{i_r}$.
Let $t_f$ be the $R$-linear form on $H_n^f$
defined by
$t_f(T_w)=\delta_{w\cdot w[1,n]}$.
This is a Frobenius form, with Nakayama automorphism the
involution given by $T_i\mapsto T_{n-i}$.

\begin{rem}
The algebra $H_n^f$ is actually symmetric, via the classical form
given by $T_w\mapsto\delta_{1,w}$. In other terms, the Nakayama automorphism
is inner: it is conjugation by $T_{w[1,n]}$. On the other hand,
the Hecke algebra is not symmetric over $\BZ[q]$ and the classical form
induces a degenerate pairing, while the form $t_f$ above is still
a Frobenius form over $\BZ[q]$ (cf \S \ref{se:nil}).
\end{rem}

\subsubsection{Affine Hecke algebras}
Let $H_n$ be the {\em affine Hecke algebra} of $\GL_n$:
$H_n=R[X_1^{\pm 1},\ldots,X_n^{\pm 1}]\otimes_R H^f$ as an $R$-module,
$R[X_1^{\pm 1},\ldots,X_n^{\pm 1}]$ and $H^f$ are subalgebras and
$$T_i X_j=X_j T_i \text{ if } j-i\not=0,1 \text{ and }
T_iX_{i+1}-X_iT_i=(q-1)X_{i+1}.$$
Given $P\in \BZ[X_1^{\pm 1},\ldots,X_n^{\pm 1}]$, we have
$T_i P-s_i(P)T_i=(q-1)X_{i+1}\partial_i(P)$.

\smallskip
We have a faithful representation on $R[X_1^{\pm 1},\ldots,X_n^{\pm 1}]=
H_n\otimes_{H_n^f} R$, where
$$T_i(P)=qs_i(P)+(q-1)X_{i+1}\partial_i(P).$$
Here $R$ denotes the one-dimensional representation of $H_n^f$ on
which $T_i$ acts by $q$.

\smallskip
\smallskip
The algebra $H_n$ has a Frobenius form over $\BZ[X_1,\ldots,X_n]$
given by (\ref{eq:FroboverP}) and a Frobenius form $t$ over
$\BZ[X_1,\ldots,X_n]^{\GS_n}$ given by (\ref{eq:FroboverPinv}).
The corresponding Nakayama automorphism of $H_n$ is the involution
$$X_i\mapsto X_{n-i+1},\ T_i\mapsto -qT_{n-i}^{-1}.$$

\subsubsection{Nil Hecke algebras}
\label{se:nil}
Let ${^0H}_n^f$ be the {\em nil Hecke algebra} of $\GL_n$:
this is the $\BZ$-algebra generated by $T_1,\ldots,T_{n-1}$,
with relations
$$T_iT_{i+1}T_i=T_{i+1}T_iT_{i+1},\ \ T_iT_j=T_jT_i \text { if }
|i-j|>1 \text { and } T_i^2=0.$$

\smallskip
Given $w=s_{i_1}\cdots s_{i_r}$ a reduced decomposition of an element
$w\in\GS_n$, we put $T_w=T_{i_1}\cdots T_{i_r}$.
Let $t_0$ be the linear form on ${^0H}_n^f$
defined by
$t_0(T_w)=\delta_{w\cdot w[1,n]}$. This is a Frobenius form, with Nakayama
automorphism given by $T_i\mapsto T_{n-i}$.

\smallskip
The nil Hecke algebra ${^0H}_n$ is a graded algebra with
$\deg T_i=-2$ and $t_0$ is homogeneous of degree $n(n-1)$.

\begin{lemma}
\label{le:isoandTw0}
Let $f:M\to N$ be a morphism of relatively $\BZ$-projective
${^0H}_n^f$-modules. If
$T_{w[1,n]}f:T_{w[1,n]}M\to T_{[1,n]}N$ is an isomorphism, then $f$
is an isomorphism.
\end{lemma}

\begin{proof}
The annihilator of $T_{w[1,n]}$ on a relatively
$\BZ$-projective module $L$ is $({^0H}_n^f)_{\le -2}L$. Nakayama's Lemma
shows that under the assumption of the lemma, the morphism $f$ is surjective.
On the other hand, $\ker f$ is a direct summand of $M$, hence
$\ker f$ is relatively $\BZ$-projective. Since $T_{w[1,n]}\ker f=0$, it
follows that $\ker f=0$.
\end{proof}

\smallskip
Let $A$ be an algebra. We denote by $A\wr {^0H}_n^f$ the algebra whose
underlying abelian group is $A^{\otimes n}\otimes {^0H}_n^f$, where
$A^{\otimes n}$ and ${^0H}_n^f$ are subalgebras and where
$(a_1\otimes\cdots \otimes a_n)T_i=T_i(a_1\otimes\cdots\otimes
a_{i-1}\otimes a_{i+1}\otimes a_i\otimes a_{i+2}\otimes\cdots\otimes a_n)$.

\subsubsection{Nil affine Hecke algebras}
\label{se:nilaffineHecke}
Let ${^0H}_n$ be the {\em nil affine Hecke algebra} of $\GL_n$:
${^0H}_n=\BZ[X_1,\ldots,X_n]\otimes {^0H}_n^f$ as an abelian group,
$\BZ[X_1,\ldots,X_n]$ and ${^0H}_n^f$ are subalgebras and
$$T_i X_j=X_j T_i \text{ if } j-i\not=0,1,\ 
T_iX_{i+1}-X_iT_i=1 \text{ and }
T_iX_i-X_{i+1}T_i=-1.$$
Given $P\in \BZ[X_1,\ldots,X_n]$, we have
$T_i P-s_i(P)T_i=PT_i -T_is_i(P)=\partial_i(P)$.

\smallskip
We have a faithful representation on $\BZ[X_1,\ldots,X_n]=
{^0H}_n\otimes_{{^0H}_n^f}\BZ$ where
$$T_i(P)=\partial_i(P).$$
Let $b_n=T_{w[1,n]}X_1^{n-1}X_2^{n-2}\cdots X_{n-1}$. By induction on $n$,
one sees that $\partial_{w[1,n]}(X_1^{n-1}X_2^{n-2}\cdots X_{n-1})=1$, hence
$b_n^2=b_n$. We have an isomorphism of ${^0H}_n$-modules
$$\BZ[X_1,\ldots,X_n]\iso {^0H}_n b_n,\ P\mapsto Pb_n.$$

Since $\{\partial_w(X_1^{n-1}\cdots X_{n-1})\}_{w\in \GS_n}$ is a basis
of $\BZ[X_1,\ldots,X_n]$ over $\BZ[X_1,\ldots,X_n]^{\GS_n}$, it follows
that the multiplication map gives an isomorphism of
$({^0H}_n^f,\BZ[X_1,\ldots,X_n]^{\GS_n})$-bimodules
$${^0H}_n^f\otimes (\BZ[X_1,\ldots,X_n]^{\GS_n}X_1^{n-1}\cdots X_{n-1}b_n)
\iso {^0H}_nb_n.$$

\begin{prop}
\label{pr:MoritanilHecke}
The action of ${^0H}_n$ on $\BZ[X_1,\ldots,X_n]$ induces an isomorphism
$${^0H}_n\iso \End_{\BZ[X_1,\ldots,X_n]^{\GS_n}}(\BZ[X_1,\ldots,X_n]).$$
 Since $\BZ[X_1,\ldots,X_n]$ is a free
$\BZ[X_1,\ldots,X_n]^{\GS_n}$-module
of rank $n!$, the algebra ${^0}H_n$ is isomorphic to a
$(n!\times n!)$-matrix algebra over $\BZ[X_1,\ldots,X_n]^{\GS_n}$.

The restriction to ${^0H}_n^f$ of any ${^0H}_n$-module is
relatively $\BZ$-projective.
\end{prop}

\begin{proof}
Since $\BZ[X_1,\ldots,X_n]$ is a finitely generated projective
${^0H}_n$-module, the canonical map
${^0H}_n\iso \End_{\BZ[X_1,\ldots,X_n]^{\GS_n}}(\BZ[X_1,\ldots,X_n])$
 splits as a morphism of
$\BZ[X_1,\ldots,X_n]^{\GS_n}$-modules. The first two assertions of the
proposition follow from the fact that
${^0H}_n$ is a free $\BZ[X_1,\ldots,X_n]^{\GS_n}$-module of rank $(n!)^2$.

The $({^0H}_n^f,\BZ[X_1,\ldots,X_n]^{\GS_n})$-bimodule 
$\BZ[X_1,\ldots,X_n]$ is a direct summand of ${^0H}_n$. So, given
$M$ an $\BZ[X_1,\ldots,X_n]^{\GS_n}$-module, then 
$\BZ[X_1,\ldots,X_n]\otimes_{\BZ[X_1,\ldots,X_n]^{\GS_n}}M$ is a direct
summand of ${^0H}_n^f\otimes_\BZ 
(\BZ[X_1,\ldots,X_n]\otimes_{\BZ[X_1,\ldots,X_n]^{\GS_n}}M)$ as an
${^0H}_n^f$-module. So, given $N$ an ${^0H}_n$-module, then
$N$ is a direct summand of ${^0H}_n^f\otimes_\BZ N$ as an ${^0H}_n^f$-module
and the proposition is proven.
\end{proof}

Lemma \ref{le:isoandTw0} joined with
Proposition \ref{pr:MoritanilHecke} gives
a useful criterion to check that a morphism of ${^0H}_n$-modules is
invertible. Note also that the proposition shows that
${^0H}_n$ is projective as a $({^0H}_n^f,{^0H}_n)$-bimodule.

\smallskip
The algebra ${^0H}_n$ has a Frobenius form over $\BZ[X_1,\ldots,X_n]$
given by (\ref{eq:FroboverP}) and a Frobenius form $t$ over
$\BZ[X_1,\ldots,X_n]^{\GS_n}$ given by (\ref{eq:FroboverPinv}).
The corresponding Nakayama automorphism of ${^0H}_n$ is the involution
$$X_i\mapsto X_{n-i+1},\ T_i\mapsto -T_{n-i}.$$

\smallskip
A special feature of the nil affine Hecke algebra, compared to the
affine Hecke algebra and the degenerate affine Hecke algebra, is that
the Nakayama automorphism $\gamma$
is inner, hence the nil affine Hecke algebra is actually
symmetric over $\BZ[X_1,\ldots,X_n]^{\GS_n}$. Indeed, when
viewed as a subalgebra of $\End_\BZ(\BZ[X_1,\ldots,X_n])$, then
${^0H}_n$ contains $\GS_n$. The injection of $\GS_n$ in
${^0H}_n$ is given by $s_i\mapsto (X_i-X_{i+1})T_i+1$ (cf also
\S \ref{se:isovariousHecke}). We have
$$w[1,n]\cdot a\cdot w[1,n]=\gamma(a) \text{ for all }a\in {^0H}_n.$$

It follows that the linear form $t'$ given by
$t'(a)=t(aw[1,n])$ is a symmetrizing form for ${^0H}_n$ over
$\BZ[X_1,\ldots,X_n]^{\GS_n}$.

\medskip
The nil affine Hecke algebra ${^0H}_n$ is a graded algebra with
$\deg X_i=2$ and $\deg T_i=-2$ and $t$ is homogeneous of degree $0$.
The nil affine Hecke algebra has also a bifiltration given by
$$F^{\le (i,j)}\left({^0H}_n\right)=\BZ[X_1,\ldots,X_n]_{\le i}\otimes
\left({^0H}_n^f\right)_{\ge -j}.$$
Note that $t(F^{<(n(n-1),n(n-1))})=0$.

\subsubsection{Isomorphisms}
\label{se:isovariousHecke}
The polynomial representations above induce isomorphisms with
the semi-direct product of the algebra of polynomials with $\GS_n$, after
a suitable localization.

\smallskip
Let $R'=\BZ[X_1,\ldots,X_n,(X_i-X_j)^{-1},(X_i-X_j-1)^{-1}]_{i\not=j}$.
We have an isomorphism of $R'$-algebras
$$R'\rtimes \GS_n\iso R'\otimes_{\BZ[X_1,\ldots,X_n]} \bar{H}_n,\ s_i\mapsto 
\frac{X_i-X_{i+1}}{X_i-X_{i+1}+1}(T_i-1)+1=
(T_i+1)\frac{X_i-X_{i+1}}{X_i-X_{i+1}-1}-1$$

\smallskip
Let $R'_q=R[X_1^{\pm 1},\ldots,X_n^{\pm 1},(X_i-X_j)^{-1},
(qX_i-X_j)^{-1}]_{i\not=j}$.
We have an isomorphism of $R'_q$-algebras
$$R'_q\rtimes \GS_n\iso R'_q\otimes_{R[X_1^{\pm 1},\ldots,X_n^{\pm 1}]}
 H_n,\ s_i\mapsto 
\frac{X_i-X_{i+1}}{qX_i-X_{i+1}}(T_i-q)+1=
(T_i+1)\frac{X_i-X_{i+1}}{X_i-qX_{i+1}}-1$$

\smallskip
Let ${^0R}'=\BZ[X_1,\ldots,X_n,(X_i-X_j)^{-1}]_{i\not=j}$.
We have an isomorphism of ${^0R}'$-algebras
$${^0R}'\rtimes \GS_n\iso {^0R}'\otimes_{\BZ[X_1,\ldots,X_n]} {^0H}_n,
\ s_i\mapsto 
(X_i-X_{i+1})T_i+1= T_i(X_{i+1}-X_i)-1$$

\medskip
Let us finally note that the functor 
$$M\mapsto M^{\GS_n}:({^0R}'\rtimes\GS_n)\mMod\to ({^0R}')^{\GS_n}\mMod$$
is an equivalence of categories. 

\subsection{Nil Hecke algebras associated with hermitian matrices}
\label{se:nilmatrix}
In this section, we introduce a flat family of algebras presented by
quiver and relations. To a symmetrizable Cartan datum afforded by a
quiver with automorphism, we associate a member of that family.

\subsubsection{Definition}
Let $I$ be a set, $k$ a commutative ring and $Q=(Q_{i,j})_{i,j\in I}$
a matrix in $k[u,v]$ with $Q_{ii}=0$ for all $i\in I$.

Let $n$ be a positive integer and $L=I^n$.
We define a (possibly non-unitary) $k$-algebra $H_n(Q)$ by generators
and relations.
It is generated by elements $1_\nu$, $x_{i,\nu}$ for $i\in\{1,\ldots,n\}$ and
$\tau_{i,\nu}$ for $i\in\{1,\ldots,n-1\}$ and $\nu\in L$ and the relations are
\begin{itemize}
\item $1_\nu 1_{\nu'}=\delta_{\nu,\nu'}1_\nu$
\item $\tau_{i,\nu}= 1_{s_i(\nu)}\tau_{i,\nu}1_\nu$
\item $x_{a,\nu}= 1_{\nu} x_{a,\nu} 1_\nu$
\item $x_{a,\nu}x_{b,\nu}=x_{b,\nu}x_{a,\nu}$
\item $\tau_{i,s_i(\nu)}\tau_{i,\nu}=
 Q_{\nu_i,\nu_{i+1}}(x_{i,\nu},x_{i+1,\nu})$
\item $\tau_{i,s_{j}(\nu)}\tau_{j,\nu}=\tau_{j,s_i(\nu)}\tau_{i,\nu}$ if $|i-j|>1$
\item
$\tau_{i+1,s_is_{i+1}(\nu)}\tau_{i,s_{i+1}(\nu)}\tau_{i+1,\nu}-
\tau_{i,s_{i+1}s_i(\nu)}\tau_{i+1,s_i(\nu)}\tau_{i,\nu}=$\\

\noindent
$\begin{cases}
(x_{i+2,\nu}-x_{i,\nu})^{-1}\left(Q_{\nu_i,\nu_{i+1}}(x_{i+2,\nu},x_{i+1,\nu})-
Q_{\nu_i,\nu_{i+1}}(x_{i,\nu},x_{i+1,\nu})\right)
& \text{if }\nu_i=\nu_{i+2}\\
0 & \text{otherwise}
\end{cases}$
\item
$\tau_{i,\nu}x_{a,\nu}-x_{s_i(a),s_i(\nu)}\tau_{i,\nu}=\begin{cases}
-1_\nu & \text{ if }a=i \text{ and } \nu_i=\nu_{i+1} \\
1_\nu & \text{ if }a=i+1 \text{ and } \nu_i=\nu_{i+1} \\
0 & \text{ otherwise.}\end{cases}$
\end{itemize}
for $\nu,\nu'\in I$, $1\le i,j\le n-1$ and $1\le a,b\le n$.

\begin{rem}
Note that when $I$ is finite, then
$H_n(Q)$ has a unit $1=\sum_{\nu\in L}1_{\nu}$.
\end{rem}

\begin{rem}
It is actually
more natural to view $H_n(Q)$ as a category $\CH_n(Q)$ with set of
objects $L$ and with $\Hom$-spaces generated by
$$x_{a,\nu}\in\End(\nu) \text{ for }1\le a\le n$$
$$\tau_{i,\nu}\in\Hom(\nu,s_i(\nu))\text{ for }1\le i\le n-1$$
with the relations above.
\end{rem}

Given $a\in 1_\nu H_n(Q)1_{\nu'}$, we will sometimes write
$x_ia$ for $x_{i,\nu}a$ and $ax_i$ for $ax_{i,\nu'}$ and proceed similarly for
$\tau_i$.

\medskip
Consider the (possibly non-unitary) algebra
$R_n=\left(k^{(I)}[x]\right)^{\otimes n}=
k[x_1,\ldots,x_n]\otimes (k^{(I)})^{\otimes n}$.
We denote by $1_s$ the idempotent corresponding to the $s$-th factor
of $k^{(I)}$ and we put $1_\nu=1_{\nu_1}\otimes\cdots\otimes 1_{\nu_n}$ for
$\nu\in L$.

There is a morphism of algebras
$R_n\to H_n(Q),\ x_i 1_\nu\mapsto x_{i,\nu}$. It restricts to
a morphism $R_n^{\GS_n}\to Z(H_n(Q))$. Note that
$R_1=H_1(Q)$ and we put $H_0(Q)=k$.

\smallskip
Let $J$ be a set of finite sequences of elements of
$\{1,\ldots,n-1\}$ such that
$\{s_{i_1}\cdots s_{i_r}\}_{(i_1,\ldots,i_r)\in J}$ is a set of minimal
length representatives of elements of $\GS_n$. Then,
$$S=\{\tau_{i_1,s_{i_2}\cdots s_{i_r}(\nu)}\cdots
 \tau_{i_r,\nu}x_{1,\nu}^{a_1}\cdots x_{n,\nu}^{a_n}
\}_{(i_1,\ldots,i_r)\in J,(a_1,\ldots,a_n)\in\BZ_{\ge 0}^n,\nu\in L}$$
generates $H_n(Q)$ as a $k$-module.

\smallskip
The algebra $H_n(Q)$ is filtered with $1_\nu$ and $x_{i,\nu}$ in degree $0$
and $\tau_{i,\nu}$ in degree $1$. The morphism $R_n\to H_n(Q)$ extends
to a surjective algebra morphism
$$k^{(I)}[x]\wr {^0H}_n^f\to \gr H_n(Q),\ 
T_i 1_\nu\mapsto \tau_{i,\nu}.$$
The algebra is said to satisfy the PBW (Poincar\'e-Birkhoff-Witt)
property if that morphism is an isomorphism.

\begin{thm}
\label{th:PBWHecke}
Assume $n\ge 2$.
The following assertions are equivalent
\begin{itemize}
\item $H_n(Q)$ satisfies PBW
\item $H_n(Q)$ is a free $k$-module with basis $S$
\item $Q_{ij}(u,v)=Q_{ji}(v,u)$ for all $i,j\in I$.
\end{itemize}
\end{thm}

\begin{proof}
The first two assertions are equivalent, thanks to the generating
family $S$ described above.

\smallskip
Let $\nu\in L$ with $\nu_i\not=\nu_{i+1}$.
We have
$$Q_{\nu_{i+1},\nu_i}(x_{i,s_i(\nu)},x_{i+1,s_i(\nu)})\tau_{i,\nu}=
\tau_{i,\nu}\tau_{i,s_i(\nu)}\tau_{i,\nu}=
\tau_{i,\nu}Q_{\nu_i,\nu_{i+1}}(x_{i,\nu},x_{i+1,\nu})=
Q_{\nu_i,\nu_{i+1}}(x_{i+1,s_i(\nu)},x_{i,s_i(\nu)})\tau_{i,\nu}.$$
It follows that
$$(Q_{\nu_{i+1},\nu_i}(x_{i,s_i(\nu)},x_{i+1,s_i(\nu)})-
Q_{\nu_i,\nu_{i+1}}(x_{i+1,s_i(\nu)},x_{i,s_i(\nu)})) \tau_{i,\nu}=0.$$
Assume $S$ is a basis of $H_n(Q)$. We have
$Q_{\nu_{i+1},\nu_i}(x_{i,s_i(\nu)},x_{i+1,s_i(\nu)})-
Q_{\nu_i,\nu_{i+1}}(x_{i+1,s_i(\nu)},x_{i,s_i(\nu)})=0$.
Consequently,  $Q_{ij}(u,v)=Q_{ji}(v,u)$ for all $i,j\in I$.

\smallskip
Assume $Q_{ij}(u,v)=Q_{ji}(v,u)$ for all $i,j\in I$. Choose an
ordering of pairs of distinct elements of $I$. Given
$i<j$, put $P_{ij}=Q_{ij}$ and $P_{ji}=1$. The theorem follows now
from Proposition \ref{trivialHnA} below.
\end{proof}

 Denote by $Q\mapsto\bar{Q}$ the automorphism given by
$\bar{Q}_{ij}(u,v)=Q_{ji}(v,u)$.
The algebras $H_n(Q)$ form a flat family of algebras over the space
of matrices $Q$ with
vanishing diagonal and hermitian with respect to the automorphism of
$k[u,v]$ swapping $u$ and $v$ (\ie, such that $\bar{Q}=Q$).

\begin{cor}
Assume $Q$ is hermitian.
Let $I'$ be a subset of $I$ and $Q'=(Q_{i,j})_{i,j\in I'}$. Then,
the canonical map $H_n(Q')\to H_n(Q)$ is injective and
induces isomorphisms
$1_\nu H_n(Q')1_{\nu'}\iso 1_\nu H_n(Q)1_{\nu'}$ for $\nu,\nu'\in (I')^n$.
\end{cor}

From Proposition \ref{trivialHnA} below, we obtain a description of
the center of $H_n(Q)$.

\begin{prop}
Assume $Q$ is hermitian. Then, we have
$Z(H_n(Q))=R_n^{\GS_n}$.
\end{prop}

When $|I|=1$, then $H_n(Q)$
is the nil affine Hecke algebra ${^0 H}_n$ associated with $\GL_n$.

\smallskip
Given $0\le i\le n$, we have an injective morphism of $R_n$-algebras
$$H_i(Q)\otimes H_{n-i}(Q)\to H_n(Q)$$
given by $1_\nu\otimes 1_{\nu'}\mapsto 1_{\nu\cup \nu'}$,
$x_{j,\nu}\otimes 1_{\nu'}\mapsto x_{j,\nu\cup \nu'}$,
$1_\nu\otimes x_{j,\nu'}\mapsto x_{i+j,\nu\cup \nu'}$, etc.

\smallskip
Assume $Q$ is hermitian. Let $i_1,\ldots,i_m$ be distinct elements of $I$ and
let $d_1,\ldots,d_m\in\BZ_{\ge 0}$ with
$n=\sum_r d_r$. Let $\nu=(\underbrace{i_1,\ldots,i_1}_{d_1\text{ terms}},
\ldots,\underbrace{i_m,\ldots,i_m}_{d_m\text{ terms}})$.
The construction above induces
an isomorphism of algebras
$${^0H}_{d_1}\otimes\cdots\otimes {^0H}_{d_m}\iso
1_\nu H_n(Q)1_\nu.$$

\begin{rem}
The algebra Khovanov and Lauda \cite{KhoLau2} associate to a
symmetrizable Cartan matrix $(a_{ij})$ corresponds to
$Q_{ij}(u,v)=u^{-a_{ij}}+v^{-a_{ji}}$ for
$i\not=j$.
\end{rem}

\medskip
Let us describe some isomorphisms between $H_n(Q)$'s.

\smallskip
Let $\{a_i\}_{i\in I}$ in $k$ and $\{\beta_{ij}\}_{i,j\in I}$ in
$k^\times$.
Let $Q'_{ij}(u,v)=\beta_{ij}\beta_{ji} Q_{ij}(\beta_{jj} u+a_j,
\beta_{ii} v+a_i)$.
We have an isomorphism
$$H_n(Q')\iso H_n(Q),\ 1_\nu\mapsto 1_\nu,\ 
x_{i,\nu}\mapsto \beta_{\nu_i,\nu_i}^{-1}(x_{i,\nu}-a_{\nu_i}),\
\tau_{i,\nu}\mapsto \beta_{\nu_i,\nu_{i+1}}\tau_{i,\nu}.$$

The construction above provides an action of the subgroup
$\{(\beta_{ij})_{i,j}|\beta_{ij}\beta_{ji}=1\text{ and }\beta_{ii}=1\}$
of $(\BG_m)^{I\times I}$ on $H_n(Q)$.

\smallskip
Assume $Q$ is hermitian.
Given $\nu\in I^n$, we define $\bar{\nu}\in I^n$ by $\bar{\nu}_i=\nu_{n-i+1}$.
There is an involution of $H_n(Q)$
$$H_n(Q)\iso H_n(Q),\ 1_\nu\mapsto 1_{\bar{\nu}},\
x_{i,\nu}\mapsto x_{n-i+1,\bar{\nu}},\ \tau_{i,\nu}\mapsto
-\tau_{n-i,\bar{\nu}}.$$

\smallskip
Let us finally construct a duality. There is an isomorphism
$$H_n(Q)\iso H_n(Q)^\opp,\ 1_\nu\mapsto 1_\nu,\ x_{i,\nu}\mapsto x_{i,\nu},\
\tau_{i,\nu}\mapsto \tau_{i,s_i(\nu)}.$$

\begin{rem}
One can also work with a matrix $Q$ with values in $k(u,v)$ and define
$H_n(Q)$ by adding inverses of the relevant polynomials in $x_{i,\nu}$'s.
\end{rem}

\subsubsection{Polynomial realization}
\label{se:isographs}
Let $P=(P_{ij})_{i,j\in I}$ be a matrix in
$k[u,v]$ with $P_{ii}=0$ for all $i\in I$
and let $Q_{i,j}(u,v)=P_{i,j}(u,v)P_{j,i}(v,u)$.

Consider the (possibly non-unitary)
$k$-algebra $A_n(I)=
k^{(I)}[x]\wr \GS_n$.

\smallskip
The following Proposition provides a faithful
representation of $H_n(Q)$ on the space $R_n$. It also shows that, after
localization, the algebra $H_n(Q)$ depends only on the cardinality of $I$
(assuming non-vanishing of $Q_{ij}$ for $i\not=j$).

\begin{prop}
\label{trivialHnA}
Let $\CO'=\bigoplus_{\nu\in L}
k[x_1,\ldots,x_n][\{(x_i-x_j)^{-1}\}_{i\not=j,\nu_i=\nu_j}] 1_\nu$.
We have an injective morphism of $k$-algebras
$$H_n(Q)\to \CO'\otimes_{\BZ^{(I)}[x]^{\otimes n}}A_n(I)$$
$$1_\nu\mapsto 1_\nu,\ x_{a,\nu}\mapsto x_a 1_\nu,$$
$$\tau_{i,\nu}\mapsto\begin{cases}
(x_i-x_{i+1})^{-1}(s_i 1_\nu -1_\nu) & \text{ if }\nu_i=\nu_{i+1} \\
P_{\nu_i,\nu_{i+1}}(x_{i+1},x_i)s_i 1_\nu &
\text{ otherwise}
\end{cases}$$
for $1\le a\le n$, $1\le i\le n-1$ and $\nu\in L$.
It defines a faithful representation of $H_n(Q)$ on
$R_n=\bigoplus_{\nu\in L}k[x_1,\ldots,x_n]1_\nu$.

\smallskip
Assume $P_{i,j}\not=0$ for all $i\not=j$.
Let 
$$\CO=\bigoplus_{\nu\in L}
k[x_1,\ldots,x_n][\{P_{\nu_i,\nu_j}(x_i,x_j)^{-1}\}_{\nu_i\not=\nu_j},
\{(x_i-x_j)^{-1}\}_{i\not=j,\nu_i=\nu_j}] 1_\nu.$$
The morphism above induces an isomorphism
$\CO\otimes_{k^{(I)}[x]^{\otimes n}}H_n(Q)\iso
\CO\otimes_{k^{(I)}[x]^{\otimes n}}A_n(I)$.
\end{prop}

\begin{proof}
Let $\tau'_{i,\nu}=\begin{cases}
(x_i-x_{i+1})^{-1}(s_i 1_\nu -1_\nu) & \text{ if }\nu_i=\nu_{i+1} \\
P_{\nu_i,\nu_{i+1}}(x_{i+1},x_i)s_i 1_\nu &
\text{ otherwise}.
\end{cases}$

Let us check that the defining relations of $H_n(Q)$ hold with $\tau_{i,\nu}$
replaced by $\tau'_{i,\nu}$.
We will not write the idempotents $1_\nu$ to make the calculations more
easily readable.

We have
\begin{multline*}\tau_{i,s_{i+1}(\nu)}'\tau_{i+1,\nu}'=\\
\begin{cases}
(x_i-x_{i+1})^{-1}\left((x_i-x_{i+2})^{-1}(s_is_{i+1}-s_i)-
(x_{i+1}-x_{i+2})^{-1}(s_{i+1}-1)\right) &
 \text{ if }\nu_i=\nu_{i+1}=\nu_{i+2} \\
P_{\nu_,\nu_{i+1}}(x_{i+1},x_i)(x_i-x_{i+2})^{-1}(s_is_{i+1}-s_i) &
 \text{ if }\nu_{i+1}=\nu_{i+2}\not=\nu_i \\
(x_i-x_{i+1})^{-1}\left(P_{\nu_{i+1},\nu_{i+2}}(x_{i+2},x_i)s_is_{i+1}-
P_{\nu_{i+1},\nu_{i+2}}(x_{i+2},x_{i+1})s_{i+1} \right)  &
 \text{ if }\nu_i=\nu_{i+2}\not=\nu_{i+1} \\
P_{\nu_i,\nu_{i+2}}(x_{i+1},x_i)P_{\nu_{i+1},\nu_{i+2}}(x_{i+2},x_i)s_is_{i+1}&
 \text{ if }\nu_{i+2}\not\in\{\nu_i,\nu_{i+1}\}.
\end{cases}
\end{multline*}

\smallskip
Assume $\nu_i=\nu_{i+1}=\nu_{i+2}$.
We have
\begin{multline*}\tau_{i,s_{i+1}s_i(\nu)}'\tau_{i+1,s_i(\nu)}'\tau_{i,\nu}'=\\
=(x_{i+1}-x_{i+2})^{-1}
(x_i-x_{i+2})^{-1}(x_i-x_{i+1})^{-1}(s_{i+1}s_is_{i+1}-s_{i+1}s_i-s_is_{i+1}
+s_i+s_{i+1}-1)\\
=\tau_{i+1,s_is_{i+1}(\nu)}'\tau_{i,s_{i+1}(\nu)}'\tau_{i+1,\nu}'
\end{multline*}
Assume $\nu_i=\nu_{i+1}\not=\nu_{i+2}$. We have
\begin{multline*}\tau_{i,s_{i+1}s_i(\nu)}'\tau_{i+1,s_i(\nu)}'\tau_{i,\nu}'=\\
=(x_{i+1}-x_{i+2})^{-1}P_{\nu_i,\nu_{i+2}}(x_{i+1},x_i)
P_{\nu_i,\nu_{i+2}}(x_{i+2},x_i)(s_{i+1}s_is_{i+1}-s_is_{i+1})\\
=\tau_{i+1,s_is_{i+1}(\nu)}'\tau_{i,s_{i+1}(\nu)}'\tau_{i+1,\nu}'
\end{multline*}
Assume $\nu_{i+1}=\nu_{i+2}\not=\nu_i$. We have
\begin{multline*}\tau_{i,s_{i+1}s_i(\nu)}'\tau_{i+1,s_i(\nu)}'\tau_{i,\nu}'=\\
=(x_i,x_{i+1})^{-1}P_{\nu_i,\nu_{i+1}}(x_{i+2},x_i)
P_{\nu_i,\nu_{i+1}}(x_{i+2},x_{i+1})(s_{i+1}s_is_{i+1}-s_{i+1}s_i)\\
=\tau_{i+1,s_is_{i+1}(\nu)}'\tau_{i,s_{i+1}(\nu)}'\tau_{i+1,\nu}'
\end{multline*}

Assume $\nu_i$, $\nu_{i+1}$ and $\nu_{i+2}$ are distinct. We have
\begin{multline*}\tau_{i,s_{i+1}s_i(\nu)}'\tau_{i+1,s_i(\nu)}'\tau_{i,\nu}'=\\
=P_{\nu_i,\nu_{i+1}}(x_{i+2},x_{i+1})P_{\nu_i,\nu_{i+2}}(x_{i+2},x_i)
P_{\nu_{i+1},\nu_{i+2}}(x_{i+1},x_i)s_{i+1}s_is_{i+1}\\
=\tau_{i+1,s_is_{i+1}(\nu)}'\tau_{i,s_{i+1}(\nu)}'\tau_{i+1,\nu}'
\end{multline*}

Assume finally $\nu_i=\nu_{i+2}\not=\nu_{i+1}$.
We have
\begin{multline*}\tau_{i,s_{i+1}s_i(\nu)}'\tau_{i+1,s_i(\nu)}'\tau_{i,\nu}'=\\
=(x_i-x_{i+2})^{-1}P_{\nu_{i+1},\nu_i}(x_{i+1},x_i)\left(
P_{\nu_i,\nu_{i+1}}(x_{i+2},x_{i+1})s_is_{i+1}s_i-P_{\nu_i,\nu_{i+1}}(x_i,x_{i+1})\right)
\end{multline*}
and
\begin{multline*}
\tau_{i+1,s_is_{i+1}(\nu)}'\tau_{i,s_{i+1}(\nu)}'\tau_{i+1,\nu}'=\\
(x_i-x_{i+2})^{-1}P_{\nu_i,\nu_{i+1}}(x_{i+2},x_{i+1})\left(
P_{\nu_{i+1},\nu_i}(x_{i+1},x_i)s_{i+1}s_is_{i+1}-
P_{\nu_{i+1},\nu_i}(x_{i+1},x_{i+2})\right)
\end{multline*}
hence
\begin{multline*}
\tau_{i+1,s_is_{i+1}(\nu)}'\tau_{i,s_{i+1}(\nu)}'\tau_{i+1,\nu}'-
\tau_{i,s_{i+1}s_i(\nu)}'\tau_{i+1,s_i(\nu)}'\tau_{i,\nu}'=\\
(x_i-x_{i+2})^{-1}\left(
P_{\nu_{i+1},\nu_i}(x_{i+1},x_i)P_{\nu_i,\nu_{i+1}}(x_i,x_{i+1})-
P_{\nu_{i+1},\nu_i}(x_{i+1},x_{i+2})P_{\nu_i,\nu_{i+1}}(x_{i+2},x_{i+1})\right).
\end{multline*}
The other relations are immediate to check.

\smallskip
Let $B$ be the $k$-subalgebra of $\CO\otimes_{k^{(I)}[x]^{\otimes n}}A_n(I)$
image of the morphism.
We have $\CO\otimes_{k^{(I)}[x]^{\otimes n}}B=
\CO\otimes_{k^{(I)}[x]^{\otimes n}}A_n(I)$.
The image of $S$ in $\CO\otimes_{k^{(I)}[x]^{\otimes n}}A_n(I)$ is
linearly independent
over $k$. It follows that 
the canonical map $H_n(Q)\to B$ is an isomorphism and that $S$ is a basis
of $H_n(Q)$ over $k$.
\end{proof}

\subsubsection{Cartan matrices}
\label{se:deformedHeckeCartan}
Let $C=(a_{ij})$ be a Cartan matrix, \ie, 
\begin{itemize}
\item $a_{ii}=2$,
\item
$a_{ij}\in\BZ_{\le 0}$ for $i\not=j$ and
\item $a_{ij}=0$ if and only if $a_{ji}=0$.
\end{itemize}

 We put $m_{ij}=-a_{ij}$.
Let $\{t_{i,j,r,s}\}$ be a family of indeterminates with
$i\not=j\in I$, $0\le r<m_{ij}$ and $0\le s<m_{ji}$ and such that
$t_{j,i,s,r}=t_{i,j,r,s}$. Let $\{t_{ij}\}_{i\not=j}$ be a family of
indeterminates with $t_{ij}=t_{ji}$ if $a_{ij}=0$.

Let $\Bk=\Bk_C=\BZ[\{t_{i,j,r,s}\}\cup\{t_{ij}^{\pm 1}\}]$. Let $Q_{ii}=0$,
$Q_{ij}=t_{ij}$ if $a_{ij}=0$ and
$$Q_{ij}=t_{ij}u^{m_{ij}}+\sum_{\substack{0\le r<m_{ij}\\ 0\le s<m_{ji}}}
t_{i,j,r,s} u^r v^s + t_{ji}v^{m_{ji}} \text{ for }i\not=j
\text{ and }a_{ij}\not=0.$$

We put $H_n(C)=H_n(Q)$. This is a $\Bk$-algebra, free as a $\Bk$-module.

\medskip
Consider $s\not=t\in I$ and assume $n=m_{st}+2$.
Let $\nu=(t,s,\ldots,s)\in I^n$. Given $0\le i\le n-1$, let
$c_i=s_i\cdots s_1$: we have $c_i(\nu)=(s,\ldots,s,t,s,\ldots,s)$, where
$t$ is in the $(i+1)$-th position.
The canonical isomorphisms ${^0H}_i\iso 1_{(s,\ldots,s)}H_i(Q)1_{(s,\ldots,s)}$
and
${^0H}_{n-i-1}\iso 1_{(s,\ldots,s)}H_{n-i-1}(Q)1_{(s,\ldots,s)}$ give rise to
a morphism of unitary algebras
$${^0H}_i\otimes {^0 H}_{n-1-i} \to
1_{c_i(\nu)}H_n(Q)1_{c_i(\nu)}.$$
We denote by $e_{i+1}$ the image of $b_i\otimes b_{n-1-i}$
(cf \S \ref{se:nilaffineHecke}).

\smallskip
The following Lemma generalizes a result of Khovanov and Lauda
\cite[Corollary 7]{KhoLau2}.
\begin{lemma}
\label{le:SerreHecke}
Let $P^i=H_n(Q)e_{i+1}$.
Define $\alpha_{i,i+1}=e_{i+1}\tau_{n-1}\cdots \tau_{i+2}\tau_{i+1}e_{i+2}$ and
$\alpha_{i+1,i}=e_{i+2}\tau_1\tau_2\cdots \tau_{i+1}e_{i+1}$.
We have a complex $P$ of projective $H_n(Q)$-modules
$$\xymatrix{
0\ar[r] &
P^0\ar[r]^-{\alpha_{0,1}} & P^1 \ar@/^1pc/@{.>}[l]^-{\alpha'_{1,0}}\ar[r]
&\cdots \ar[r]& P^{i-1}\ar[r]^-{\alpha_{i-1,i}} &
P^i \ar[r]^-{\alpha_{i,i+1}} \ar@/^1pc/@{.>}[l]^-{\alpha'_{i,i-1}}  &
P^{i+1} \ar@/^1pc/@{.>}[l]^-{\alpha'_{i+1,i}}\ar[r]& \cdots \ar[r]&
P^{n-1}\ar[r] & 0
}$$
which is homotopy equivalent to $0$, with splittings given by the maps
$\alpha'_{i+1,i}=(-1)^{i+n} t_{st}^{-1}\alpha_{i+1,i}$.
\end{lemma}

\begin{proof}
Note that $b_rb_{r+1}=b_{r+1}$ and
$b_{r+1}T_1\cdots T_r b_r=T_1\cdots T_r b_r$, hence
$\alpha_{i,i+1}=\tau_{n-1}\cdots \tau_{i+2}\tau_{i+1}e_{i+2}$ and
$\alpha_{i+1,i}=e_{i+2}\tau_1\tau_2\cdots \tau_{i+1}$.

We have 
$$\alpha_{i-1,i}\alpha_{i,i+1}=
e_i\tau_{n-1}\cdots\tau_i\tau_{n-1}\cdots\tau_{i+1}e_{i+2}=
e_i\tau_{n-2}\cdots\tau_i\tau_{n-1}\cdots\tau_ie_{i+2}=0.$$
It follows that the maps $\alpha_{i-1,i}$ provide a differential.

\smallskip
We have
$$\alpha_{i,i+1}\alpha_{i+1,i}=
\tau_{n-1}\cdots\tau_{i+1}\tau_1\cdots \tau_{i+1}e_{i+1}=
\tau_1\cdots\tau_{i-1}\tau_{n-1}\cdots\tau_{i+2}\tau_{i+1}\tau_i
\tau_{i+1}e_{i+1}$$
and
$$\alpha_{i,i-1}\alpha_{i-1,i}=
\tau_1\cdots\tau_i\tau_{n-1}\cdots\tau_ie_{i+1}=
\tau_1\cdots\tau_{i-1}\tau_{n-1}\cdots\tau_{i+2}\tau_i\tau_{i+1}\tau_i
e_{i+1}.$$
It follows that
$$\alpha_{i,i+1}\alpha_{i+1,i}-\alpha_{i,i-1}\alpha_{i-1,i}=
\partial_{s_1\cdots s_{i-1}s_{n-1}\cdots s_{i+2}}\Bigl(
(x_{i+2}-x_i)^{-1}\bigl(Q_{st}(x_{i+2},x_{i+1})-Q_{st}(x_i,x_{i+1})\bigr)\Bigr).$$
Write $Q_{st}(u,v)=\sum_{a,b}q_{ab}u^av^b$ with $q_{a,b}\in\BZ$. We have
$$(x_{i+2}-x_i)^{-1}\bigl(Q_{st}(x_{i+2},x_{i+1})-Q_{st}(x_i,x_{i+1})\bigr)=
\sum_{a\ge 1,b\ge 0}q_{ab}x_{i+1}^b(x_{i+2}^{a-1}+x_{i+2}^{a-2}x_i+\cdots+
x_i^{a-1}),$$
hence
$$\alpha_{i,i+1}\alpha_{i+1,i}-\alpha_{i,i-1}\alpha_{i-1,i}=
\sum_{a\ge 1,b\ge 0}q_{ab}x_{i+1}^b\sum_{c=i-1}^{a-n+i+1}
\partial_{s_1\cdots s_{i-1}}(x_i^c)\partial_{s_{n-1}\cdots s_{i+2}}
(x_{i+2}^{a-c-1})=(-1)^{n+i}q_{n-2,0}$$
and finally $\alpha_{i,i+1}\alpha'_{i+1,i}+\alpha'_{i,i-1}\alpha_{i-1,i}=1$.
\end{proof}

Assume $C$ is a symmetrizable Cartan matrix, \ie, there is a family
$(d_i)_{i\in I}$ of positive integers with $\lcm(\{d_i\})=1$ and such that
$(b_{ij})$ is symmetric, for $b_{ij}=d_ia_{ij}$.

Let $\Bk^\bullet$ be the quotient of $\Bk$ by the ideal generated by
those $t_{i,j,r,s}$ such that $d_ir+d_js\not=-b_{ij}$. Let
$H_n^\bullet(C)=\Bk^\bullet\otimes_\Bk H_n(C)$. The algebra
$H_n^\bullet(C)$ is graded with
$\deg 1_\nu=0$,
$\deg x_{i,\nu}=2d_{\nu_i}$ and $\deg\tau_{i,\nu}=-b_{\nu_i,\nu_{i+1}}$.

\begin{rem}
\label{re:finiterank}
The description of the basis $S$ for $H_n^\bullet(C)$ (cf Theorem
\ref{th:PBWHecke}) shows that the rank of the sum of the homogeneous
components of $1_{\nu'} H_n^\bullet(C)1_\nu$ with degree less than a given integer
is finite.
\end{rem}

\subsubsection{Quivers with automorphism}
\label{se:quivers}
Let $\Gamma$ be a quiver with a compatible automorphism
\cite[\S 12.1.1]{Lu}: this is the data of
\begin{itemize}
\item a set $\tI$ (vertices)
\item a set $H$ (edges) and a map with finite fibers $h\mapsto [h]$
from $H$ to the set of two-element subsets of $\tI$
\item maps $s:H\to \tI$ (source) and $t:H\to \tI$ (target) such that
$\{s(h),t(h)\}=[h]$ for any $h\in H$
\item automorphisms $a:\tI\to \tI$ and $a:H\to H$ such that
$s(a(h))=a(s(h))$ and $t(a(h))=a(t(h))$  and such that
$s(h)$ and $t(h)$ are not in the same $a$-orbit for $h\in H$.
\end{itemize}

We put $I=\tI/a$. We define 
$i\cdot i=2\#(i)$ and
$i\cdot j=-\#\{h\in H| [h]\in i\cup j\}$ for $i\not=j$ in $I$
(note that this uses only the graph structure, not the orientation).
This defines a Cartan datum and
$\left(2\frac{i\cdot j}{i\cdot i}\right)_{i,j}$
is a symmetrizable Cartan matrix.

Given $i,j\in I$, let $d_{ij}$ be the number of orbits of $a$ in
$\{h\in H| s(h)\in i \text{ and } t(h)\in j\}$. We have
$d_{ij}+d_{ji}=-2(i\cdot j)/\lcm(i\cdot i,j\cdot j)$ for $i\not=j$.

Define
$$P_{ij}(u,v)=\left(v^{l/(j\cdot j)}-u^{l/(i\cdot i)}\right)^{d_{ij}}
\text{ where }
l=\lcm(i\cdot i,j\cdot j), \text{ for }i\not=j \text{ and }
P_{ii}=0.$$

We have
$$Q_{ij}=(-1)^{d_{ij}}\left(u^{l/(i\cdot i)}-v^{l/(j\cdot j)}
\right)^{-2(i\cdot j)/l}\text{ for }i\not= j.$$

We put $k=\BZ$ and $H_n(\Gamma)=H_n(Q)$. This is a specialization of
the algebra $H_n(C)$ introduced in \S \ref{se:deformedHeckeCartan}.

\smallskip
The algebra $H_n(\Gamma)$ is graded with
$\deg 1_\nu=0$,
$\deg x_{i,\nu}=\nu_i\cdot \nu_i$ and $\deg\tau_{i,\nu}=-\nu_i\cdot \nu_{i+1}$. As
a graded algebra, it is a specialization of $H^\bullet_n(C)$ (here,
$d_i=(i\cdot i)/2$).

\smallskip
Consider another choice of orientation $s',t'$ of the graph
$(\tI,H,h\mapsto[h])$, compatible  with the automorphism $a$. Given
$i\not=j$, define
$$\beta_{ij}=\begin{cases}
(-1)^{d_{ij}+d'_{ij}} & \text{ if } d_{ij}\ge d'_{ij} \\
1 & \text{ otherwise.}
\end{cases}$$
We have an isomorphism
$$H_n(\Gamma)\iso H_n(\Gamma'),\
1_\nu\mapsto 1_\nu,\ x_{i,\nu}\mapsto x_{i,\nu},\ 
\tau_{i,\nu}\mapsto \beta_{\nu_i,\nu_{i+1}}\tau_{i,\nu}.$$

It follows that, up to isomorphism, the graded algebra $H_n(\Gamma)$ depends
only on the Cartan datum. Note
nevertheless that the system of isomorphisms constructed above between the
algebras corresponding to different orientations is not a transitive
system. Consequently, we do not define ``the'' algebra associated to a
Cartan datum (or a graph with automorphism). Note finally
that, up to isomorphism, $H_n(\Gamma)$ depends only on the Cartan matrix
and a change of quiver with automorphism
corresponds to a rescaling of the grading.

\smallskip
Note that if $\Gamma$ is the disjoint union of full subquivers $\Gamma_1$
and $\Gamma_2$, then $H_n(\Gamma)=H_n(\Gamma_1)\otimes H_n(\Gamma_2)$.

\subsubsection{Type $A$ graphs}
\label{se:typeAgraphs}
Let $k$ be a field and $q\in k^\times$.

Assume first $q=1$. Given $I$ a subset of $k$, we denote by
$I_1$ the quiver with set of vertices $I$ and with an arrow $i\to i+1$,
whenever $i,i+1\in I$.

Assume now $q\not=1$. Given $I$ a subset of $k^\times$, we denote by
$I_q$ the quiver with set of vertices $I$ and with an arrow $q\to qi$,
whenever $i,qi\in I$.

\medskip
Note that $I_q$ has type $A$ and we put
$\Gsl_{I_q}=\Gg_{I_q}$. Let us assume $I_q$ is connected. Let us 
describe the possible types for the underlying graph.

Assume $q=1$. Type:
\begin{itemize}
\item $A_n$ if $|I|=n$ and $k$ has characteristic $0$ or $p>n$.
\item $\tilde{A}_{p-1}$ if $|I|=p$ is the characteristic of $k$.
\item $A_{\infty}$ if $I$ is bounded in one direction but not finite.
\item $A_{\infty,\infty}$ if $I$ is unbounded in both directions.
\end{itemize}

Assume $q\not=1$. Denote by $e$ the multiplicative order of $q$. Type:
\begin{itemize}
\item $A_n$ if $|I|=n<e$.
\item $\tilde{A}_{e-1}$ if $|I|=e$.
\item $A_{\infty}$ if $I$ is bounded in one direction but not finite.
\item $A_{\infty,\infty}$ if $I$ is unbounded in both directions.
\end{itemize}

\subsubsection{Idempotents and representations}
\label{se:idempotents}
Let $k$ be a field and
let $\Gamma$ be a quiver.
We denote by $kH_n(\Gamma)\mMOD_0$ the category of 
$H_n(\Gamma)$-modules $M$ such that $M=\bigoplus_{\nu}1_\nu M$ and
for every $\nu$, the elements $x_{i,\nu}$ act locally
nilpotently on $1_\nu M$ for $1\le i\le n$.

\smallskip
Let $I$ be a subset of $k$ and let $\Gamma=I_1$.

 Let
$$\bar{\CO}'=\bigoplus_{\nu\in I^n}k[X_1,\ldots,X_n]
[\{(X_i-X_j)^{-1}\}_{i\not=j,\nu_i\not=\nu_j},
\{(X_i-X_j+1)^{-1}\}_{i\not=j,\nu_i+1\not=\nu_j}],$$
a non-unitary ring. Note that this is a subring of
$$\bigoplus_{\nu\in I^n}k[X_1,\ldots,X_n]
[(X_i-X_j-a)^{-1}]_{i\not=j,a\not=\nu_i-\nu_j}.$$

We denote by $1_\nu$ the unit of the summand of $\bar{\CO}'$ corresponding
to $\nu$. We put a structure of non-unitary algebra on
$\bar{\CO}' \bar{H}_n=\bar{\CO}'
\otimes_{\BZ[X_1,\ldots,X_n]}\bar{H}_n$
by setting
$$T_i 1_\nu-1_{s_i(\nu)}T_i=
(X_{i+1}-X_i)^{-1}(1_\nu-1_{s_i(\nu)}).
$$

Let
$$\tilde{\bar{\CO}}'=\bigoplus_{\nu\in I^n}k[x_1,\ldots,x_n]
[\{(\nu_i-\nu_j+x_i-x_j)^{-1}\}_{i\not=j,\nu_i\not=\nu_j},
\{(\nu_i-\nu_j+1+x_i-x_j)^{-1}\}_{i\not=j,\nu_i+1\not=\nu_j}],$$
a subring of
$$\bigoplus_{\nu\in I^n}k[x_1,\ldots,x_n]
[(x_i-x_j-a)^{-1}]_{i\not=j,a\not=0}.$$

From Proposition \ref{trivialHnA} and \S \ref{se:isovariousHecke},
we obtain the following proposition.

\begin{prop}
\label{pr:isowithgraph1}
We have an isomorphism of non-unitary algebras
$$\tilde{\bar{\CO}}' H_n(\Gamma)\iso\bar{\CO}' \bar{H}_n,\ 
x_i 1_\nu\mapsto (X_i-\nu_i) 1_\nu,$$
$$\tau_i 1_\nu\mapsto\begin{cases}
(X_i-X_{i+1}+1)^{-1}(T_i-1) 1_\nu & \text{ if }\nu_i=\nu_{i+1} \\
((X_i-X_{i+1})T_i+1) 1_\nu & \text{ if }\nu_{i+1}=\nu_i+1 \\
\frac{X_i-X_{i+1}}{X_i-X_{i+1}+1}(T_i-1)1_\nu+1_\nu & \text{ otherwise.}
\end{cases}$$
\end{prop}

\smallskip
Let $M$ be a $k\bar{H}_n$-module. Given $a\in k^n$, we denote by
$M_a$ the $k[X_1,\ldots,X_n]$-submodule of $M$ of elements with support
contained in the closed point of $\BA_k^n$ given by $a$.

We denote by $\bar{\CC}_\Gamma$
the category of $k\bar{H}_n$-modules $M$ such that
$$M=\bigoplus_{a\in \Gamma^n} M_a.$$

\begin{thm}
\label{th:equivdegnil}
We have an equivalence of categories
$$kH_n(\Gamma)\mMOD_0\iso \bar{\CC}_\Gamma,\ M\mapsto M$$
where $X_i$ acts on $1_\nu M$ by $(x_i+\nu_i)$ and
$T_i$ acts on $1_\nu M$ by

\begin{itemize}
\item $(x_i-x_{i+1}+1)\tau_i+1$ if $\nu_i=\nu_{i+1}$ \\
\item $(x_i-x_{i+1}-1)^{-1}(\tau_i-1)$
if $\nu_{i+1}=\nu_i+1$ \\
\item $(x_i-x_{i+1}+\nu_{i+1}-\nu_i+1)
(x_i-x_{i+1}+\nu_{i+1}-\nu_i)^{-1}(\tau_i-1)+1$
otherwise.
\end{itemize}
\end{thm}

Assume $I$ is finite.
Let $d:I\to\BZ_{>0}$ be a function and
$\bar{H}_n(I,d)$ be the quotient of $k\bar{H}_n$ by the two-sided ideal
generated by
$\prod_{i\in I}(X_1-i)^{d(i)}$, a degenerate cyclotomic Hecke algebra.
Let $H_n(\Gamma,d)$ be the quotient of $H_n(\Gamma)$ by the ideal
generated by $x_i^{d(i)}$ for $i\in I$.

\begin{cor}
\label{co:isodeg}
The construction of Theorem \ref{th:equivdegnil} induces an
isomorphism of $k$-algebras
$H_n(\Gamma,d)\iso \bar{H}_n(I,d)$.
\end{cor}

\medskip
Let $k$ be a field and $q\in k-\{0,1\}$.
Let $I$ be a subset of $k^\times$ and let $\Gamma=I_q$.

Let
$$\CO'=\bigoplus_{\nu\in I^n}k[X_1^{\pm 1},\ldots,X_n^{\pm 1}]
[\{(X_i-X_j)^{-1}\}_{i\not=j,\nu_i\not=\nu_j},
\{(qX_i-X_j)^{-1}\}_{i\not=j,q\nu_i\not=\nu_j}],$$
a non-unitary $k[X_1^{\pm 1},\ldots,X_n^{\pm 1}]$-algebra.
Note that this is a subring of
$$\bigoplus_{\nu\in I^n}k[X_1^{\pm 1},\ldots,X_n^{\pm 1}]
[(X_i-aX_j)^{-1}]_{i\not=j,a\in k-\{0,\nu_i\nu_j^{-1}\}}.$$

We denote by $1_\nu$ the unit of the summand of $\CO'$ corresponding
to $\nu$. We put a structure of non-unitary algebra on
$\CO' H_n=\CO'\otimes_{\BZ[q^{\pm 1},X_1^{\pm 1},\ldots,X_n^{\pm 1}]}H_n$
by setting
$$T_i 1_\nu-1_{s_i(\nu)}T_i=
(1-q)X_{i+1}(X_i-X_{i+1})^{-1}(1_\nu-1_{s_i(\nu)}).
$$

Let
$$\tilde{\CO}'=\bigoplus_{\nu\in I^n}k[x_1^{\pm 1},\ldots,x_n^{\pm 1}]
[\{(\nu_i\nu_j^{-1}x_i-x_j)^{-1}\}_{i\not=j,\nu_i\not=\nu_j},
\{(q\nu_i\nu_j^{-1}x_i-x_j)^{-1}\}_{i\not=j,q\nu_i\not=\nu_j}],$$
a subring of
$$\bigoplus_{\nu\in \BZ^n}k[q^{\pm 1},x_1^{\pm 1},\ldots,x_n^{\pm 1}]
[(x_i-ax_j)^{-1}]_{i\not=j,a\in k-\{0,1\}}.$$

From Proposition \ref{trivialHnA} and \S \ref{se:isovariousHecke},
we obtain the following proposition.

\begin{prop}
\label{pr:isowithgraphq}
We have an isomorphism of non-unitary algebras
$$\tilde{\CO}' H_n(\Gamma)\iso\CO' H_n,\ 
x_i 1_\nu\mapsto \nu_i^{-1}X_i 1_\nu,$$
$$\tau_i 1_\nu\mapsto\begin{cases}
\nu_i(qX_i-X_{i+1})^{-1}(T_i-q) 1_\nu & \text{ if }\nu_i=\nu_{i+1} \\
q^{-1}\nu_i^{-1}((X_i-X_{i+1})T_i + (q-1)X_{i+1}))1_\nu
& \text{ if }\nu_{i+1}=q\nu_i \\
(\frac{X_i-X_{i+1}}{qX_i-X_{i+1}}(T_i-q)+1)1_\nu & \text{ otherwise.}
\end{cases}$$
\end{prop}

\smallskip
Let $M$ be a $kH_n$-module. Given $a\in \left(k^\times\right)^n$,
we denote by
$M_a$ the $k[X_1^{\pm 1},\ldots,X_n^{\pm 1}]$-submodule of $M$ of elements
with support
contained in the closed point of $\BA_k^n$ given by $a$.

We denote by $\CC_\Gamma$ the category of $kH_n$-modules $M$ such that
$$M=\bigoplus_{a\in \Gamma^n} M_a.$$

\begin{thm}
\label{th:equivaffnil}
We have an equivalence of categories
$$kH_n(\Gamma)\mMOD_0\iso \CC_\Gamma,\ M\mapsto M$$
where $X_i$ acts on $1_\nu M$ by $\nu_i(x_i+1)$ and
$T_i$ acts on $1_\nu M$ by

\begin{itemize}
\item $(qx_i-x_{i+1})\tau_i+q$ if $\nu_i=\nu_{i+1}$ \\
\item $(q^{-1}x_i-x_{i+1})^{-1}(\tau_i+(1-q)x_{i+1})$ if $\nu_{i+1}=q\nu_i$ \\
\item $(\nu_ix_i-\nu_{i+1}x_{i+1})^{-1}\left(
(q\nu_ix_i-\nu_{i+1}x_{i+1})\tau_i+
(1-q)\nu_{i+1}x_{i+1}\right)$
otherwise.
\end{itemize}
\end{thm}

Assume $I$ is finite.
Let $d:I\to\BZ_{>0}$ be a function and
$H_n(I,d)$ be the quotient of $kH_n$ by the two-sided ideal generated by
$\prod_{i\in I}(X_1-i)^{d(i)}$, a cyclotomic Hecke algebra.

\begin{cor}
\label{co:isoaff}
The construction of Theorem \ref{th:equivaffnil} induces an
isomorphism of $k$-algebras
$H_n(\Gamma,d)\iso H_n(I,d)$.
\end{cor}

\begin{rem}
The isomorphisms of Corollaries \ref{co:isodeg} and \ref{co:isoaff} 
have been constructed and studied independently by
Brundan and Kleshchev \cite{BrKl}. They provide gradings on (degenerate)
cyclotomic Hecke algebras.
\end{rem}

\section{$2$-categories}
\subsection{Construction}
\subsubsection{Half Kac-Moody algebras}
\label{se:halfKM}
Let $I$ be a set and $C=(a_{ij})_{i,j\in I}$ a Cartan matrix.
We consider the ring $\Bk$ and the matrix $Q$ of
\S \ref{se:deformedHeckeCartan}.

Define $\CB=\CB(C)$ as the free strict
monoidal $\Bk$-linear category generated by objects
$E_s \text{ for } s\in I$
and by arrows
$$x_s:E_s\to E_s \text{ and }
\tau_{st}:E_sE_t\to E_tE_s \text{ for }s,t\in I$$
with relations

\begin{enumerate}
\item 
\label{en:half1}
$\tau_{st}\circ \tau_{ts}=Q_{st}(E_tx_s,x_tE_s)$
\item
\label{en:half2}
$\tau_{tu}E_s\circ E_t \tau_{su}\circ \tau_{st}E_u-
E_u \tau_{st}\circ \tau_{su}E_t\circ E_s \tau_{tu}=
\begin{cases}
\frac{Q_{st}(x_sE_t,E_sx_t)E_s-E_sQ_{st}(E_tx_s,x_tE_s)}{x_sE_tE_s-E_sE_tx_s}E_s
 \text{ if } s=u\vspace{0.2cm}\\
0  \text{ otherwise.}
\end{cases}$
\item
\label{en:half3}
$\tau_{st}\circ x_s E_t-E_s x_t\circ \tau_{st}=\delta_{st}$
\item
\label{en:half4}
$\tau_{st}\circ E_sx_t-x_sE_t\circ \tau_{st}=-\delta_{st}$
\end{enumerate}

These relations state that the maps $x_s$ and $\tau_{st}$ give an action of
the nil affine Hecke algebra associated with $C$ on powers of $E$.
More precisely, we have an isomorphism of (non-unitary) algebras
\begin{align*}
H_n(C)&\iso \bigoplus_{\nu,\nu'\in I^n} \Hom_{\CB}(E_{\nu_n}\cdots
E_{\nu_1},E_{\nu'_n}\cdots E_{\nu'_1})\\
1_\nu&\mapsto \id_{E_{\nu_n}\cdots E_{\nu_1}} \\
x_{i,\nu}&\mapsto E_{\nu_n}\cdots E_{\nu_{i+1}}x_{\nu_i}E_{\nu_{i-1}}\cdots E_{\nu_1}\\
\tau_{i,\nu}&\mapsto E_{\nu_n}\cdots E_{\nu_{i+2}}\tau_{\nu_{i+1},\nu_i}
E_{\nu_{i-1}}\cdots E_{\nu_1}
\end{align*}

\smallskip
Let $s\in I$ and $n\ge 0$.
We have an isomorphism of algebras
$\Bk({^0H}_n)\iso \End_{\CB}(E_s^n)$
and we denote by $E^{(n)}_s=b_nE^n_s\in\CB^i$ the image of the idempotent
$b_n=T_{w[1,n]}X_1^{n-1}X_2^{n-2}\cdots X_{n-1}$
of ${^0 H}_n$ (cf \S \ref{se:nilaffineHecke}).
We denote also by $F^{(n)}_s$ the image of
$T_{w[1,n]}X_1^{n-1}X_2^{n-2}\cdots X_{n-1}\in {^0H}_n^\opp$. Note that
this idempotent
corresponds to the idempotent $b'_n=X_1^{n-1}X_2^{n-2}\cdots X_{n-1}
T_{w[1,n]}$ of ${^0H}_n$.
Thanks to Lemma \ref{pr:MoritanilHecke}, we
have the following result (as in \cite[Lemma 5.15]{ChRou}).

\begin{lemma}
The action map is an isomorphism
${^0H}_nb_n\otimes_{P_n^{\GS_n}}E^{(n)}_s\iso E^n_s$.
In particular, we have
$E^n_s\simeq n!\cdot E^{(n)}_s$.
Similarly, we have isomorphisms
$b'_n\cdot{^0H}_n\otimes_{P_n^{\GS_n}}F^{(n)}_s\iso F^n_s$.
In particular, we have
$F^n_s\simeq n!\cdot F^{(n)}_s$.
\end{lemma}

\smallskip
The following Proposition is a consequence of
Lemma \ref{le:SerreHecke} (apply $\Hom_{H_n(C)}(P,-)$).
It gives a categorical version of the Serre relations.

\begin{prop}
\label{pr:Serre}
Consider $s\not=t\in I$ and let $m=m_{st}$.
Let $\alpha_{i,i+1}=\tau_{m+1}\cdots \tau_{i+2}\tau_{i+1}$ and
$\alpha'_{i+1,i}=(-1)^{i+m} t_{st}^{-1}\tau_1\tau_2\cdots \tau_{i+1}$.
We have a complex
$$\xymatrix{
\cdots \ar[r]&
E_s^{(m-i)}E_tE_s^{(i+1)}\ar[r]^-{\alpha_{i,i+1}} &
E_s^{(m-i+1)}E_tE_s^{(i)} \ar[r]^-{\alpha_{i-1,i}} 
\ar@/^1pc/@{.>}[l]^-{\alpha'_{i+1,i}}  &
E_s^{(m-i+2)}E_tE_s^{(i-1)} \ar@/^1pc/@{.>}[l]^-{\alpha'_{i,i-1}}\ar[r]& \cdots 
}$$
which is homotopy equivalent to $0$, with splittings given by the maps
$\alpha'_{i+1,i}$.
In particular,
$$\bigoplus_{i\text{ even}}E_s^{(m-i+1)}E_tE_s^{(i)}\simeq
\bigoplus_{i\text{ odd}}E_s^{(m-i+1)}E_tE_s^{(i)}.$$
\end{prop}

\begin{rem}
The first part of Proposition \ref{pr:Serre} generalizes \cite{KhoLau2}.
We will give a different proof of the existence of an isomorphism (second
part of the Proposition) in \cite{Rou3} in the
case of integrable $2$-representations.
\end{rem}

\medskip
Assume now $C$ is symmetrizable and consider $(d_i)$, $(b_{ij})$ and
$\Bk^\bullet$ as in \S \ref{se:deformedHeckeCartan}. We put
$\CB_0^{\bullet}=\CB\otimes_\Bk\Bk^\bullet$.

The category
$\CB_0^{\bullet}$ can be enriched in graded abelian groups by
setting $\deg x_s=2d_s$ and $\deg\tau_{st}=-b_{st}$. We denote by
$\CB^{\bullet}$ the corresponding graded category. 
It follows from
Theorem \ref{th:PBWHecke} and
Remark \ref{re:finiterank} that $\Hom$-spaces
in $\CB^{\bullet}$ are free $\Bk^\bullet$-modules of
finite rank.

\smallskip
We put $E_s^{(n)}=b_nE^n_s(\frac{n(n-1)}{2}d_s)$. Note that
$P_n(\frac{n(n-1)}{2}d_s)$ is self-dual as a graded
$P_n^{\GS_n}$-module and we have
$$E_s^n\simeq v^{n(n-1)d_s}[n]_s! E_s^{(n)}$$
where $[n]_s!=[n]!(v^{d_s})$.

The maps $\alpha_{ij}$ and $\alpha'_{ij}$
of Proposition \ref{pr:Serre} are graded and
the proposition remains true in $\CB^\bullet$.

\medskip
Consider finally $\Gamma$ a quiver with a compatible automorphism
and consider the specialization $\Bk^\bullet\to\BZ$ of \S \ref{se:quivers}.
We put $\CB_\BZ^\bullet(\Gamma)
=\CB^{\bullet}(C)\otimes_{\Bk^\bullet} \BZ$.

\subsubsection{Kac-Moody algebras}
\label{se:defcat}
Let $C=(a_{ij})_{i,j\in I}$ be a Cartan matrix.
Let $(X,Y,\langle-,-\rangle,\{\alpha_i\}_{i\in I},\{\alpha^\vee_i\}_{i\in I})$
be a root datum of type $C$, \ie,
\begin{itemize}
\item $X$ and $Y$ are finitely generated free abelian groups and
$\langle-,-\rangle:Y\times X\to\BZ$ is a perfect pairing
\item $I\to X,\ i\mapsto \alpha_i$ and
$I\to Y,\ i\mapsto \alpha_i^\vee$ are injective and
$\langle\alpha_i^\vee,\alpha_j\rangle=a_{ij}$.
\end{itemize}

Associated with this data, there is a Kac-Moody algebra $\Gg$,
a quantum group $U_v(\Gg)$ when $C$ is symmetrizable,
as well as completed versions \cite{Lu}. Let us
recall those we will need.

\smallskip
Assume $C$ is symmetrizable.
Consider the $\BQ(v)$-algebra ${'U}_v^+(\Gg)$ generated by elements
$e_i$ for $i\in I$ with relations
\begin{equation}
\label{eq:Serre}
\sum_{a+b=1-a_{ij}} (-1)^a e_i^{(a)}e_je_i^{(b)}=0
\end{equation}
for any $i\not=j\in I$, where $e_i^{(a)}=\frac{e_i^a}{[a]_i!}$.
We denote by $U_v^+(\Gg)$ the $\BZ[v^{\pm 1}]$-subalgebra generated by
the $e_i^{(a)}$ for $i\in I$ and $a\ge 0$.
We define an algebra $U_v^-(\Gg)$ isomorphic to $U_v^+(\Gg)$ with
$e_i$ replaced by $f_i$.

Let ${'U}_v(\Gg)$ be the category enriched in $\BQ(v)$-vector spaces with
set of objects $X$ and morphisms generated by
$e_i:\lambda\to \lambda+\alpha_i$ and
$f_i:\lambda\to \lambda-\alpha_i$ subject to the following relations:

\begin{itemize}
\item the relation (\ref{eq:Serre}) and its version with $e_r$ replaced by
$f_r$

\item $[e_i,f_j]1_\lambda=\delta_{ij}
\langle \alpha_i^\vee,\lambda\rangle 1_\lambda$.
\end{itemize}

Let $U_v(\Gg)$ be the subcategory enriched in $\BZ[v^{\pm 1}]$-modules
of ${'U}_v(\Gg)$ with same objects as ${'U}_v(\Gg)$ and with morphisms
generated by $e_i^{(r)}$ and $f_i^{(r)}$ for $i\in I$ and $r\ge 0$.

We put $U_1(\Gg)=U_v(\Gg)\otimes_{\BZ[v^{\pm 1}]}
\BZ[v^{\pm 1}]/(v-1)$, etc.

Note that $\bigoplus_{\lambda,\mu\in X}\Hom_{U_v(\Gg)}(\lambda,\mu)$
is the non-unitary ring ${_\CA \dot{\BU}}$ of \cite[\S 23.2]{Lu}. 

The category of functors (compatible with the $\BZ[v^{\pm 1}]$-structure)
$U_v(\Gg)\to \BZ[v^{\pm 1}]\mMOD$ is equivalent
to the category of unital ${_\CA \dot{\BU}}$-modules via
$V\mapsto \bigoplus_\lambda V(\lambda)$ and we will identify the two categories.
A representation $V$ of $U_v(\Gg)$ is 
{\em integrable} if for every $v\in V$ and $i\in I$, there is $n_0$ such that
$e_i^{(n)}v=f_i^{(n)}v=0$ for all $n\ge n_0$.

\smallskip
Assume $C$ is a general Cartan matrix. The constructions above still
make sense in the non-quantum case and lead to a category
$U_1(\Gg)$.

\smallskip
We denote by $W=\langle\sigma_i\rangle_{i\in I}$ the Weyl group of
$\Gg$.

\subsubsection{$2$-Kac Moody algebras}
\label{se:def2km}
Let $\CB_1$ be the strict monoidal $\Bk$-linear category obtained from
$\CB$ by adding $F_s$
right dual to $E_s$ for every $s\in I$.
Define
$$\eps_s=\eps_{E_s}:E_s F_s\to \idun\text{ and } \eta_s=\eta_{E_s}:
\idun\to F_s E_s.$$
The dual pairs $(E_s,F_s)$ provides dual pairs $(E_s^n,F_s^n)$ and the
action of ${^0H}_n$ on $E_s^n$ induces
an action of $({^0H}_n)^\opp$ on $F_s^n$. We denote by
$x_s$ the endomorphism of $F_s$ induced by $x_s\in\End(E_s)$ and
denote also by $\tau_{st}:F_sF_t\to F_tF_s$ the morphism induced by
$\tau_{st}\in\Hom(E_sE_t,E_tE_s)$.

\smallskip
We define a morphism of monoids
$$h:\Ob(\CB_1)\to X,\ E_s\mapsto \alpha_s,\ F_s\mapsto -\alpha_s.$$

\smallskip
Consider the strict $2$-category $\FA_1$ with set of objects $X$ and
$\CHom(\lambda,\lambda')=h^{-1}(\lambda'-\lambda)$, a full subcategory
of $\CB_1$. We write $E_{s,\lambda}$ for $E_s\idun_\lambda$,
$\eps_{s,\lambda}$ for $\eps_s\idun_\lambda$, etc.

\smallskip
Let
$\FA=\FA(\Gg)$ be the $\Bk$-linear strict
$2$-category deduced from $\FA_1$ by inverting the following
$2$-arrows:
\begin{itemize}
\item  when $\langle\alpha_s^\vee,\lambda\rangle\ge 0$,
$$\rho_{s,\lambda}=
\sigma_{ss}+\sum_{i=0}^{\langle\alpha_s^\vee,\lambda\rangle-1}\eps_s\circ
(x_s^i F_s):
E_s F_s\idun_\lambda\to F_s E_s\idun_\lambda \oplus
\idun_\lambda^{\langle\alpha_s^\vee,\lambda\rangle}$$
\item  when $\langle\alpha_s^\vee,\lambda\rangle\le 0$,
$$\rho_{s,\lambda}=
\sigma_{ss}+\sum_{i=0}^{-1-\langle\alpha_s^\vee,\lambda\rangle}
(F_s x_s^i)\circ \eta_s:
E_s F_s\idun_\lambda\oplus
 \idun_\lambda^{-\langle\alpha_s^\vee,\lambda\rangle}
\to F_s E_s\idun_\lambda$$
\item $\sigma_{st}:E_sF_t\idun_\lambda\to F_tE_s\idun_\lambda$
 for all $s\not=t$ and all $\lambda$
\end{itemize}

where we define
$$\sigma_{st}=(F_t E_s\eps_t)\circ (F_t \tau_{ts}F_s)\circ
(\eta_t E_sF_t):E_s F_t\to F_t E_s.$$

\begin{rem}
The inversion of maps in the definition of $\FA$
accounts for the Lie algebra relations
$[e_s,f_s]=h_s$ and $[e_s,f_t]=0$ for $s\not=t$. The elements $h_\zeta$ for
$\zeta\in Y$ appear
only through their action as multiplication by $\langle\zeta,\lambda\rangle$
on the $\lambda$-weight space.
\end{rem}

Assume $C$ is symmetrizable.
We proceed as in \S \ref{se:halfKM} to define graded versions.
Let $\FA_0^{\bullet}=\FA\otimes_\Bk \Bk^\bullet$.
The category
$\FA_0^\bullet$ can be
enriched in graded abelian groups by setting 
$$\deg \eps_{s,\lambda}=d_s(1-\langle\alpha_s^\vee,\lambda\rangle)
\text{ and }
\deg \eta_{s,\lambda}=d_s(1+\langle\alpha_s^\vee,\lambda\rangle).$$
We denote by $\FA^{\bullet}$ the corresponding graded $2$-category.

Note that $\sigma_{st}$ is a graded map (for all $s,t\in I$), while
$\rho_{s,\lambda}$ carries shifts:
$$\rho_{s,\lambda}:
E_s F_s\idun_\lambda\iso F_s E_s\idun_\lambda \oplus
\bigoplus_{i=0}^{\langle\alpha_s^\vee,\lambda\rangle-1}
\idun_\lambda\bigl(d_s(2i+1-\langle\alpha_s^\vee,\lambda\rangle)\bigr)
\text{ when }
\langle\alpha_s^\vee,\lambda\rangle\ge 0,$$
$$\rho_{s,\lambda}:
E_s F_s\idun_\lambda\oplus
\bigoplus_{i=0}^{-1-\langle\alpha_s^\vee,\lambda\rangle} 
\idun_\lambda\bigl(-d_s(2i+1+\langle\alpha_s^\vee,\lambda\rangle)\bigr)
\iso F_s E_s\idun_\lambda
\text{ when }\langle\alpha_s^\vee,\lambda\rangle\le 0.$$

We have a dual pair in $\FA^\bullet$
$$\biggl(E_s\idun_\lambda,
\idun_\lambda F_s\Bigl(d_s(1+\langle\alpha_s^\vee,\lambda\rangle)\Bigr)\biggr)
.$$

Finally, given a quiver $\Gamma$ with a compatible automorphism and
associated Cartan matrix $C$, we put
$\FA^\bullet_\BZ(\Gamma)=\FA^{\bullet}\otimes_{\Bk^\bullet}\BZ$
(cf \S \ref{se:quivers}).
We put also $\FA_\BZ=\FA\otimes_\Bk\BZ$.

\medskip
Let us summarize: we have constructed several $2$-categories with set of objects
$X$ and with $\Hom(\lambda,\lambda')=h^{-1}(\lambda'-\lambda)$. Given a root
datum, we have a $\Bk$-linear $2$-category
$\CA$ and, when $C$ is symmetrizable, we have
a specialization $\CA^{\bullet}$ that
is $\Bk^\bullet$-linear and graded. Given in addition a quiver with
compatible automorphism affording the Cartan matrix, we have a further
specialization $\CA^{\bullet}_\BZ$ that is graded and $\BZ$-linear.

\begin{rem}
The action of ${^0H}_n$ on $E_s^n$ is given by
$$X_i\mapsto E_s^{n-i}x_s E_s^{i-1} \text{ and }
T_i\mapsto E_s^{n-i-1}\tau_{ss}E_s^{i-1}$$
while the action of ${^0H}_n^\opp$ on $F_s^n$ is given by
$$X_i\mapsto F_s^{i-1}x_s F_s^{n-i} \text{ and }
T_i\mapsto F_s^{i-1}\tau_{ss}F_s^{n-i-1}.$$
\end{rem}

\subsubsection{Second adjunctions}
We define ``candidates'' units and counits for an adjuntion
$(F_s,E_s)$.

\smallskip
Let $s\in I$ and $\lambda\in X$.
Assume $\langle \alpha_s^\vee,\lambda\rangle\ge 0$.
Let $\tilde{\eps}_{s,\lambda}:
F_sE_s\idun_\lambda\to \idun_\lambda$ be the map whose image
under
$$\Hom(\sigma_{ss},\idun_\lambda):
\Hom(F_sE_s\idun_\lambda,\idun_\lambda)\iso
\Hom(E_sF_s\idun_\lambda,\idun_\lambda)/
\left(\bigoplus_{i=0}^{\langle\alpha_s^\vee,\lambda\rangle-1}
\End(\idun_\lambda)\cdot\eps_s\circ (x^i_sF_s)\right)$$
coincides with $(-1)^{\langle\alpha_s^\vee,\lambda\rangle+1}\eps_s
\circ (x_s^{\langle\alpha_s^\vee,\lambda\rangle}F_s)$.

\smallskip
Assume $\langle\alpha_s^\vee,\lambda\rangle<0$.
Let $\tilde{\eps}_{s,\lambda}:F_sE_s\idun_\lambda\to \idun_\lambda$
be the unique morphism such that
$$\tilde{\eps}_{s,\lambda}\circ\rho_{s,\lambda}=
(0,0,\ldots,0,(-1)^{\langle\alpha_s^\vee,\lambda\rangle+1}).$$

\smallskip
Assume $\langle\alpha_s^\vee,\lambda\rangle>0$.
Let $\hat{\eta}_{s,\lambda}:\idun_{\lambda}\to
E_sF_s \idun_{\lambda}$ be the unique morphism such that
$$\rho_{s,\lambda}\circ \hat{\eta}_{s,\lambda}=
(0,0,\ldots,0,(-1)^{\langle\alpha_s^\vee,\lambda\rangle+1}).$$

\smallskip

Assume $\langle\alpha_s^\vee,\lambda\rangle\le 0$.
Let $\hat{\eta}_{s,\lambda}:\idun_{\lambda}\to
E_sF_s \idun_{\lambda}$ be the map whose image under
$$\Hom(\idun_{\lambda},\sigma_{ss}):
\Hom(\idun_{\lambda},E_sF_s \idun_{\lambda})\iso
\Hom(\idun_{\lambda},F_sE_s \idun_{\lambda})/
\left(\bigoplus_{i=0}^{-\langle\alpha_s^\vee,\lambda\rangle-1}(F_sx_s^i)
\circ\eta_s\cdot \End(\idun_{\lambda})\right)$$
coincides with $(-1)^{\langle\alpha_s^\vee,\lambda\rangle}
(F_sx_s^{-\langle\alpha_s^\vee,\lambda\rangle})\circ\eta_s$.

\subsubsection{Other versions}
\label{se:otherversions}
We define here $2$-categories related to the ones defined in the previous
section
by adding generators and imposing extra symmetry conditions
and relations.

\smallskip
We define $\CB_1^l$ as the strict monoidal $\Bk$-linear category
obtained from $\CB$ by adding $F_s$ left and
right adjoint to $E_s$ for
every $s\in I$.
Define
$$\eps_s^l=\eps_{F_s}:F_s\cdot E_s\to \idun,\text{ and }
\eta_s^l=\eta_{F_s}: \idun\to E_s\cdot F_s.$$

Define also $\FA_1^l$ and $\FA^l$ as $\FA_1$ and $\FA$ were defined
from $\CB_1$. Now, we define $\FA'$ 
 as the $\Bk$-linear
strict $2$-category obtained from $\FA^l$ by adding the relation
$\eps_s^l=\tilde{\eps}_s$. Note that given $X\in\FA'$, we have
$X^\vee={^\vee X}$.

\begin{rem}
The relation $\eps_s^l=\tilde{\eps}_s$ shows that $\eps_s^l$
can be expressed in terms of the maps $x_s$, $\tau_{s,s}$,
$\eta_s$ and $\eps_s$. As a consequence,
the adjunction $(F_s,E_s)$ is
determined by the adjunction $(E_s,F_s)$ and the maps $x_s$ and
$\tau_{ss}$.
\end{rem}

We define specializations of $\FA'$ in the same way as those defined
for $\FA$. Note that 
$$\deg \eta_{s,\lambda}^l=
d_s(1-\langle\alpha_s^\vee,\lambda\rangle) \text{ and }
\deg \eps_{s,\lambda}^l=
d_s(1+\langle\alpha_s^\vee,\lambda\rangle)$$
and we have a dual pair in $\FA^{\prime\bullet}$
$$\biggl(
\idun_\lambda F_s\Bigl(-d_s(1+\langle\alpha_s^\vee,\lambda\rangle)\Bigr),
E_s\idun_\lambda\biggr).$$

\medskip
We define $\bar{\FA}'$ to
be the quotient of $\FA'$ given by the relations
$\eta_s^l=\hat{\eta}_s$ and
$(f^\vee)^\vee=f$ for every $2$-arrow $f$ of $\FA'$.

There are canonical strict $2$-functors
$$\FA\to \FA'\to \bar{\FA}'.$$

\subsection{Properties}
\subsubsection{Symmetries}
\label{se:symmetries}
In \S \ref{se:symmetries}, we work in $\FA$.
The map $\sigma_{st}$ can be defined using the Hecke action on
$F^2$ instead of $E^2$:
\begin{lemma}
\label{le:sigmaF}
Given $s,t\in I$, we have
$\sigma_{st}=(E_sF_t\xrightarrow{E_sF_t\eta_s}E_sF_tF_sE_s
\xrightarrow{E_s\tau_{ts}E_s}E_sF_sF_tE_s\xrightarrow{\eps_sF_tE_s}F_tE_s)$.
\end{lemma}

\begin{proof}
The lemma follows from the commutativity of the following diagram
$$\xymatrix{
E_sF_t\ar[r]^-{\bullet\eta}\ar[d]_-{\eta\bullet} 
\ar@/_3pc/[dd]_-{\eta\bullet} &
 E_sF_tF_sE_s\ar[r]^-{E_s\tau_{ts}E_s} & 
 E_sF_sF_tE_s\ar[rr]^-{\eps\bullet} && F_tE_s && \\
F_tE_tE_sF_t\ar[r]^-{\bullet\eta}&
 F_tE_tE_sF_tF_sE_s\ar[r]^-{\bullet\tau_{ts}E_s} & 
 F_tE_tE_sF_sF_tE_s\ar[rr]^-{F_tE_t\eps F_tE_s} &&
 F_tE_tF_tE_s \ar[u]^-{F_t\eps E_s} && \\
F_tE_tE_sF_t\ar[r]^-{\bullet\eta} \ar[drr]_-{F_t\tau_{ts}F_t}&
 F_tE_tE_sF_tF_sE_s\ar[r]^-{F_t\tau_{ts}\bullet} & 
 F_tE_sE_tF_tF_sE_s\ar[rr]^-{F_tE_s\eps F_sE_s} &&
 F_tE_sF_sE_s \ar@/_3pc/[uu]_-{F_t\eps E_s} && \\
&& F_tE_sE_tF_t \ar `r[rrrru]_-{\bullet\eps} `[uuu] [rruuu]
}$$
\end{proof}

We define the {\em Chevalley involution}, a strict $2$-equivalence of
$2$-categories
$I:\FA^{\opp}\iso\FA$ satisfying $I^2=\Id$ by
$$I(\idun_\lambda)=\idun_{-\lambda},\ I(E_s)=F_s,\ 
I(\eps_s)=-\eta_s,\ 
I(\tau_{st})=-\tau_{ts} \text{ and } I(x_s)=x_s.$$
Note that $I(\sigma_{st})=-\sigma_{ts}$ (Lemma \ref{le:sigmaF}),
$I(\rho_{s,\lambda})=-\rho_{s,-\lambda}$ and
$I(\tilde{\eps}_s)=-\hat{\eta}_s$.

\smallskip
We define the {\em Chevalley duality}, a strict $2$-equivalence of
$2$-categories
$D:\FA^{\mathrm{rev}}\iso \FA$
satisfying $D^2=\Id$ by
$$\idun_\lambda\mapsto\idun_\lambda,\ E_s\mapsto F_s,\ 
x_s\mapsto x_s, \tau_{st}\mapsto\tau_{ts}, \eps_s\mapsto \eps_s, \
\eta_s\mapsto \eta_s.$$
Note that $D(\sigma_{st})=\sigma_{ts}$ (Lemma \ref{le:sigmaF}) and
$D$ fixes $\rho_{s,\lambda}$, $\tilde{\eps}_s$ and $\hat{\eta}_s$.

\smallskip
There is also a strict equivalence of monoidal categories
$$\CB^{\mathrm{rev}}\iso
\CB,\ E_s\mapsto E_s,\ x_s\mapsto x_s,\ \tau_{st}\mapsto -\tau_{ts}.$$

\subsubsection{Relations in $\Gsl_2$}
\label{se:relationsl2}
We provide isomorphisms between sums of objects of type $E_s^mF_s^n$ and
sum of objects of type $F_s^nE_s^m$.

In this section, we work in the category $\CA$
associated with $\Gg=\Gsl_2$: $I=\{s\}$ with $s\cdot s=2$,
$X=Y=\BZ$, $\alpha_s^\vee=1$ and $\alpha_s=2$.

We put $E=E_s$ and $F=F_s$.
We put $\eps=\eps_s$ and $\eta=\eta_s$.
Let $i\in\BZ_{\ge 0}$. We define by induction $\eps_m:E^mF^m\to\idun$ and
$\eta_m:\idun\to F^mE^m$ in $\CB_1$. We put $\eps_0=\eta_0=\id$ and
$\eps_m=\eps_{m-1}\circ (E^{m-1}\eps F^{m-1})$ and
$\eta_m=(F^{m-1}\eta E^{m-1})\circ \eta_{m-1}$.

 Given $a,b\in\BZ_{\ge 0}$,
we denote by $\CP(a,b)$ the set of partitions with at most $a$ non-zero
parts, all of 
which are at most $b$. Given $\mu=(\mu_1\ge\cdots\mu_a\ge 0)\in\CP(a,b)$,
we denote by $m_\mu(X_1,\ldots,X_a)=\sum_{\sigma}X_1^{\mu_{\sigma(1)}}\cdots
X_a^{\mu_{\sigma(a)}}$ the corresponding monomial symmetric function (here,
$\sigma$ runs over $\GS_a$ modulo the stabilizer of $\mu$).

\medskip
Let $m,n,i\in\BZ_{\ge 0}$ with $i\le m$ and $i\le n$ and let
$\lambda\in X$. Let $r=m-n+\lambda$. Assume $r<0$. We put
$$L(m,n,i,\lambda)=\bigoplus_{\substack{w\in \GS_m^{m-i}\\
w'\in {^{n-i}\GS}_n\\ \mu\in\CP(i,-r-i)}}
(T_w\otimes m_\mu(X_{n-i+1},\ldots,X_n) X_{n-i+1}^{i-1}
X_{n-i+2}^{i-2}\cdots X_{n-1} T_{w'})\BZ\subset
{^0H}_m\otimes_{{^0H}_i} {^0H}_n,$$
where the right action (resp. the left action) of ${^0H}_i$ on
${^0H}_m$ (resp. on ${^0H}_n$) is via 
$X_r\mapsto X_{r+m-i}$ and $T_r\mapsto T_{r+m-i}$
(resp. $X_r\mapsto X_{r+n-i}$ and $T_r\mapsto T_{r+n-i}$). The
sum is direct since $T_{w[1,i]}X_1^{i-1}\cdots X_{i-1}T_{w[1,i]}\not=0$
(cf \S \ref{se:nilaffineHecke}).

Note that $\bigoplus_{\mu\in\CP(i,-r-i)}m_\mu(X_1,\ldots,X_i)\BZ$ is the
subspace of
$\BZ[X_1,\ldots,X_i]^{\GS_i}$ of symmetric polynomials whose degree in any
of the variables is at most $-r-i$. It has dimension
$-r\choose i$.
Note that $L(m,n,0,\lambda)=\BZ$ and $L(m,n,i,\lambda)=0$ if
$i>0$ and $r=0$.

\smallskip
Let $\bar{L}(m,n,i,\lambda)=L(m,n,i,\lambda)
({^0H}^f_{m-i}\otimes ({^0H}^f_{n-i})^\opp)$,
a $(({^0H}_m^f\otimes ({^0H}_n^f)^\opp),
({^0H}^f_{m-i}\otimes ({^0H}_{n-i}^f)^\opp))$-subbimodule
of ${^0H}_m\otimes_{{^0H}_i} {^0H}_n$.

When needed, we will also consider the modules 
$L([a,b],[a',b'],i,\lambda)$ and $\bar{L}([a,b],[a',b'],i,\lambda)$ 
where $1\le a\le b\le m$ and 
$1\le a'\le b'\le n$, which are defined similarly.

\begin{lemma}
\label{le:projL}
The multiplication map induces an isomorphism
$$L(m,n,i,\lambda)\otimes ({^0H}^f_{m-i}\otimes ({^0H}^f_{n-i})^\opp)\iso
\bar{L}(m,n,i,\lambda).$$
The 
$(({^0H}_m^f\otimes ({^0H}_n^f)^\opp),
({^0H}_{m-i}\otimes ({^0H}_{n-i})^\opp))$-subbimodule
$L(m,n,i,\lambda)
({^0H}_{m-i}\otimes ({^0H}_{n-i})^\opp)$
of ${^0H}_m\otimes_{{^0H}_i} {^0H}_n$
is projective.
\end{lemma}

\begin{proof}
The first statement is clear.
The $(({^0H}_m^f\otimes ({^0H}_n^f)^\opp),
({^0H}_{m-i}\otimes ({^0H}_{n-i})^\opp))$-bimodule
$L(m,n,i,\lambda) ({^0H}_{m-i}\otimes ({^0H}_{n-i})^\opp)$
 is isomorphic to 
$-r\choose i$ copies of
$${^0H}_n^f \BZ[X_1,\ldots,X_{m-i}] \otimes
 \BZ[X_1,\ldots,X_{n-i}] X_{n-i+1}^{i-1}
X_{n-i+2}^{i-2}\cdots X_{n-1} {^0H}_n^f.$$
On the other hand,
${^0H}_d$ is projective as a 
$({^0H}_d^f,{^0H}_d)$-bimodule (cf \S \ref{se:nilaffineHecke}) and
the last statement of the lemma follows.
\end{proof}

Let $L'(m,n,i,\lambda)=\Hom_\BZ(L(n,m,i,-\lambda),\BZ)$ and
$$\bar{L}'(m,n,i,\lambda)=\Hom_{({^0H}_{m-i}^f)^\opp\otimes {^0H}_{n-i}^f}(
\bar{L}(n,m,i,-\lambda),{^0H}_{m-i}^f\otimes {^0H}_{n-i}^f).$$
The canonical isomorphism
$$L(n,m,i,-\lambda)\otimes_\BZ ({^0H}_{m-i}^f\otimes {^0H}_{n-i}^f)\iso
\bar{L}(n,m,i,-\lambda)$$
induces an isomorphism
$$\Hom_{\BZ}(L(n,m,i,-\lambda),{^0H}_{m-i}^f\otimes {^0H}_{n-i}^f) \iso
\bar{L}'(m,n,i,\lambda)$$
and composing with the canonical isomorphism
$$L'(m,n,i,\lambda)\otimes_\BZ ({^0H}_{m-i}^f\otimes {^0H}_{n-i}^f) \iso 
\Hom_{\BZ}(L(n,m,i,-\lambda),{^0H}_{m-i}^f\otimes {^0H}_{n-i}^f),$$
we obtain an isomorphism of right
$({^0H}_{m-i}^f\otimes ({^0H}_{n-i}^f)^\opp)$-modules
$$L'(m,n,i,\lambda)\otimes_\BZ ({^0H}_{m-i}^f\otimes {^0H}_{n-i}^f) \iso 
\bar{L}'(m,n,i,\lambda).$$

\medskip
Given $m,n\in\BZ_{\ge 0}$, we define by induction a map
$\sigma_{m,n}:E^mF^n\to F^nE^m$.
The maps $\sigma_{m,0}$ and $\sigma_{0,n}$ are identities. We put
$\sigma_{m,1}=(\sigma E^{m-1})\circ (E\sigma_{m-1,1})$ and
$\sigma_{m,n}=(F^{n-1}\sigma_{m,1})\circ(\sigma_{m,n-1}F)$.

\begin{lemma}
\label{le:sigmamn}
The map $\sigma_{m,n}$ is a morphism of 
$(H_m^f\otimes (H_n^f)^\opp)$-modules.
We have
$$\sigma_{m,1}=(E^mF\xrightarrow{\eta E^m F}FE^{m+1}F
\xrightarrow{F(T_1\cdots T_m)F} FE^{m+1}F\xrightarrow{FE^m\eps}FE^m)$$
and
$$\sigma_{1,n}=(EF^n\xrightarrow{EF^n\eta}EF^{m+1}E
\xrightarrow{E(T_1\cdots T_n)E} EF^{n+1}E\xrightarrow{\eps F^nE}F^nE).$$
Given $a,b\in\BZ_{\ge 0}$, we have a commutative diagram
$$\xymatrix{
E^aF^b\ar[r]^-{\eta\bullet} \ar[d]_-{\sigma_{a,b}} &
 FE^{a+1}F^b\ar[rr]^-{F\sigma_{a+1,b}} && F^{b+1}E^{a+1} \\
F^bE^a \ar[r]_-{F^b\eta\bullet} & F^{b+1}E^{a+1} 
\ar[urr]_{(T_b\cdots T_1)E^{a+1}}
}$$
\end{lemma}

\begin{proof}
We have a commutative diagram
$$\xymatrix{
E^mF\ar[r]^-{E\eta\bullet} \ar[d]_{\eta\bullet}
 & EFE^mF \ar[rr]^-{EF(T_1\cdots T_{m-1})F} &&
EFE^mF \ar[r]^-{\bullet\eps} & EFE^{m-1} \ar[r]^-{\eta\bullet} &
FE^2FE^{m-1} \ar[d]^-{FT\bullet}\\
FE^{m+1}F\ar[d]_-{FT\bullet} &&&&& FE^2FE^{m-1} \ar[d]^-{FE\eps\bullet} \\
FE^{m+1}F\ar[r]^-{FE^2\eta\bullet}\ar@/_2pc/[rr]_-{\id}&
FE^2FE^mF\ar[r]^-{FE\eps\bullet} &
FE^{m+1}F\ar[rr]^-{\bullet(T_1\cdots T_{m-1})F} &&
FE^{m+1}F\ar[r]^-{\bullet\eps} &FE^m
}$$
and the second statement follows
by induction. The third statement follows from the
second one by applying the Chevalley duality (cf \S \ref{se:symmetries}).

\smallskip
Let $i\in [1,m-1]$. Since $T_{i+1}T_1\cdots T_m=T_1\cdots T_mT_i$,
we have a commutative diagram
$$\xymatrix{
E^mF \ar[r]^-{\eta\bullet} \ar[dd]_{T_iF} &
FE^{m+1}F \ar[rr]^-{F(T_1\cdots T_m)F}
 \ar[drr]_{F(T_{i+1} T_1\cdots T_m)F\ \ \ \ }
 &&
FE^{m+1}F \ar[r]^-{\bullet\eps} & FE^m \ar[d]^{F T_i} \\
&&& FE^{m+1}F \ar[r]^-{\bullet\eps} & FE^m \\
E^m F \ar[r]^-{\eta\bullet} & FE^{m+1}F \ar[rur]_-{\ F(T_1\cdots T_m)F}
}$$
It follows that $\sigma_{m,1}$ commutes with the action of ${^0H}_m^f$ and
by induction we deduce that $\sigma_{m,n}$ commutes with ${^0H}_m^f$. The
commutation with $({^0H}_n^f)^\opp$ follows by applying the Chevalley 
duality.

\smallskip
We have a commutative diagram
$$\xymatrix{
E^aF\ar[r]^-{\eta\bullet} \ar[d]^-{\eta\bullet} \ar@/_3pc/[ddd]_-{\sigma_{a,1}}
 & FE^{a+1}F \ar[rrr]^-{F\sigma_{a+1,1}} \ar[d]_{F\eta\bullet} &&& F^2E^{a+1}\\
FE^{a+1}F \ar[dr]_-{F\eta\bullet} & F^2E^{a+2}F \ar[r]^-{F^2T\bullet} &
 F^2E^{a+2}F \ar[rr]^-{F^2E(T_1\cdots T_a)F} && 
 F^2E^{a+2}F \ar[u]^-{\bullet\eps} \\
& F^2E^{a+1}F \ar[r]^-{T\bullet} & F^2E^{a+1}F 
 \ar[urr]_-{F^2E(T_1\cdots T_a)F}\\
FE^a \ar[rrrr]_-{F\eta\bullet} &&&& F^2E^{a+1} \ar@/_3pc/[uuu]_{T\bullet}
}$$
hence, we obtain a commutative diagram
$$\xymatrix{
E^aF^b\ar[r]^-{\eta\bullet} \ar[d]_{\sigma_{a,1}\bullet} & FE^{a+1}F^b
\ar[rr]^-{F\sigma_{a+1,1}\bullet} &&
 F^2E^{a+1}F^{b-1} \ar[rr]^-{\bullet\sigma_{a+1,b-1}} && F^{b+1}E^{a+1} \\
FE^aF^{b-1} \ar[r]_-{F\eta\bullet} & F^2E^{a+1}F^{b-1} \ar[urr]_-{T\bullet}
\ar[rr]_-{F^2\sigma_{a+1,b-1}} && F^{b+1}E^{a+1} \ar[urr]_-{T\bullet}
}$$
The last part of the Lemma follows now by induction on $b$.
\end{proof}

Given $P\in\BZ[X_1,\ldots,X_n]$, we denote by $\deg_*(P)$ the maximum of
the degrees in any of the variables of $P$. Given $\mu$ a partition and
$l$ a non-negative integer, we denote by $\mu\cup\{l\}$ the partition
obtained by adding $l$ to $\mu$.

\begin{lemma}
\label{le:diffmonomial}
Let $a,b\in\BZ_{\ge 0}$, $P\in\BZ[X_1,\ldots,X_a]^{\GS_a}$ and
$Q\in\BZ[X_{a+1},\ldots,X_{a+b}]^{\GS_{[a+1,a+b]}}$.
Then,
$\deg_*\partial_{w[1,a+b]w[1,a]w[a+1,b]}(PQ)\le
\max(\deg_*(P)-b,\deg_*(Q)-a)$.

\smallskip
Let $\mu\in\CP(a,d)$ for some $d\in\BZ_{\ge 0}$ and let $l\ge d+a$. We have
$$\partial_{s_1\cdots s_a}(X_{a+1}^l m_\mu(X_1,\ldots,X_a))=
m_{\mu\cup\{l-a\}}(X_1,\ldots,X_{a+1})+R,$$
where $R$ is a symmetric polynomial with $\deg_*R<l-a$.
\end{lemma}

\begin{proof}
Let us first show by induction on $n\ge 1$
that given $a_1,\ldots,a_n\in\BZ_{\ge 0}$, we
have 
\begin{equation}
\label{eq:deg}
\deg_*(\partial_{w[1,n]}(X_1^{a_1}\cdots X_n^{a_n}))\le
\max(\{a_i\})-n+1.
\end{equation}
This clear for $n=1$. Applying a permutation of
$[1,n]$ if necessary, we
can assume that $a_n=\min(\{a_i\})$. Then,
\begin{align*}
\partial_{w[1,n]}(X_1^{a_1}\cdots X_n^{a_n})&=
(X_1\cdots X_n)^{a_n}\partial_{s_1\cdots s_{n-1}}
\partial_{w[1,n-1]}(X_1^{a_1-a_n}\cdots X_{n-1}^{a_{n-1}-a_n})\\
&=(X_1\cdots X_n)^{a_n}\partial_{s_1\cdots s_{n-1}}(R)
\end{align*}
where $R$ is a polynomial in $X_1,\ldots,X_{n-1}$ whose degree in 
$X_{n-1}$ is at most $\max(\{a_i\})-a_n-n+2$ by induction. It
follows that the degree in $X_n$ of $\partial_{s_1\cdots s_{n-1}}(R)$
is at most $\max(\{a_i\})-a_n-n+1$ and (\ref{eq:deg}) follows from the
fact that $\partial_{w[1,n]}(X_1^{a_1}\cdots X_n^{a_n})$ is a symmetric
polynomial.

We have
$\partial_{w[1,a+b]w[1,a]w[a+1,a+b]}(PQ)=
\partial_{w[1,a+b]}(PX_1^{a-1}\cdots X_{a-1}QX_{a+1}^{b-1}\cdots X_{a+b-1})$
and the first part of the lemma follows from (\ref{eq:deg}).

\smallskip
We prove the second part of the
lemma by induction on $a$. We write $k\subset \mu$ if there
is $i$ such that $\mu_i=k$ and we denote by $\mu\setminus k$ the partition
obtained by removing $k$ to $\mu$.
We have
$$\partial_{s_1\cdots s_a}(X_{a+1}^l m_\mu(X_1,\ldots,X_a))=
\sum_{k\subset\mu}\partial_{s_1}(X_1^k\partial_{s_2\cdots s_a}(
X_{a+1}^l m_{\mu\setminus k}(X_2,\ldots,X_a))).$$
By induction, we have
$$\partial_{s_2\cdots s_a}(
X_{a+1}^l m_{\mu\setminus k}(X_2,\ldots,X_a))=
X_2^{l-a+1}m_{\mu\setminus k}(X_3,\ldots,X_{a+1})+R,$$
where the degree in $X_2$ of $R$ is strictly less than $l-a+1$.
It follows that
$$\partial_{s_1\cdots s_a}(X_{a+1}^l m_\mu(X_1,\ldots,X_a))=
\sum_{k\subset\mu}X_1^{l-a}X_2^km_{\mu\setminus k}(X_3,\ldots,X_{a+1})+R'=
X_1^{l-a}m_{\mu}(X_2,\ldots,X_{a+1})+R',$$
where the degree in $X_1$ of $R'$ is strictly less than $l-a$. The lemma
follows.
\end{proof}

The following Lemma is clear.
\begin{lemma}
\label{le:filtrationiso}
Let $\CC$ be a $k$-linear category, $X,Y$ two objects of $\CC$,
$L$ and $L'$ two right $\End(X)$-modules and
$f:L\to\Hom(X,Y)$ and $f':L'\to\Hom(X,Y)$ two morphisms of
right $\End(X)$-modules. Let
$\phi:L\otimes_{\End(X)}X\to Y$ and
$\phi':L'\otimes_{\End(X)}X\to Y$ be the associated morphisms.

Consider finite filtrations on $L$ and on $L'$
such that $f(L^{<i})=f'(L^{\prime <i})$ for all $i$.
Assume there are
isomorphisms $L^{\le i}/L^{<i}\iso L^{\prime\le i}/L^{\prime <i}$ for all $i$
such that the following diagram commutes
$$\xymatrix{
L^{\le i}/L^{<i}\ar[dr]\ar[dd]_\sim \\
& \Hom(X,Y)/{L^{<i}\End(X)} \\
L^{\prime \le i}/L^{\prime <i}\ar[ur] \\
}$$
Then, $\phi$ is an isomorphism if and only if $\phi'$ is an isomorphism.
\end{lemma}

\begin{lemma}
\label{le:basiccommutation}
Assume $m-n+\lambda\le 0$. We have an isomorphism
$\sum_i \mathrm{act}\circ \Bigl(\id\otimes\bigl((F^{n-i}\eta_iE^{m-i})\circ
\sigma_{m-i,n-i}\bigr)\Bigr)$:
$$\bigoplus_{i=0}^{\min(m,n)} L(m,n,i,\lambda)\otimes_\BZ E^{m-i}F^{n-i}
\idun_{\lambda} (i(i-2m-\lambda))
\iso F^nE^m\idun_{\lambda}.$$
It induces an isomorphism of $({^0H}_m^f\otimes ({^0H}_n^f)^\opp)$-modules:
$$\bigoplus_{i=0}^{\min(m,n)} \bar{L}(m,n,i,\lambda)
\otimes_{{^0H}_{m-i}^f\otimes ({^0H}_{n-i}^f)^\opp} E^{m-i}F^{n-i}
\idun_{\lambda} (i(i-2m-\lambda)) \iso F^nE^m\idun_\lambda.$$

Assume $m-n+\lambda\ge 0$. We have an isomorphism
$\sum_i (\id\otimes(\sigma_{m-i,n-i}\circ (E^{m-i}\eps_i F^{n-i})))\circ
\mathrm{act}^*$:
$$E^mF^n\idun_\lambda\iso
\bigoplus_{i=0}^{\min(m,n)} L'(m,n,i,\lambda)\otimes_\BZ F^{n-i}E^{m-i}
\idun_{\lambda}(i(2n-\lambda-i)).$$
It induces an isomorphism of $({^0H}_m^f\otimes ({^0H}_n^f)^\opp)$-modules:
$$E^mF^n\idun_\lambda\iso
\bigoplus_{i=0}^{\min(m,n)} \bar{L}'(m,n,i,\lambda)
\otimes_{{^0H}_{m-i}^f\otimes ({^0H}_{n-i}^f)^\opp} F^{n-i}E^{m-i}
\idun_{\lambda}(i(2n-\lambda-i)).$$
\end{lemma}

\begin{proof}
Note first that the statements for $(m,n,\lambda)$ where $m-n+\lambda\le 0$
are transformed into the statements for $(n,m,-\lambda)$ by the Chevalley
involution $I$. It is immediate to check that the maps are graded and it is
enough to prove the Lemma in the non-graded setting.

\smallskip
Assume $m-n+\lambda\le 0$.
Note that the first statement is equivalent to the second one
(Lemma \ref{le:projL}), whose map
makes sense thanks to Lemma \ref{le:sigmamn}.
We will drop
the idempotents $\idun_\lambda$ to simplify notations. Note that
the result holds for $m=n=1$ as $\rho_{s,\lambda}$ is invertible by
definition.

Since 
$\bar{L}(m,n,i,\lambda)
\otimes_{{^0H}_{m-i}^f\otimes ({^0H}_{n-i}^f)^\opp}
({^0H}_{m-i}\otimes {^0H}_{n-i})$ is
projective as a
$({^0H}_m^f\otimes ({^0H}_n^f)^\opp,{^0H}_{m-i}\otimes
({^0H}_{n-i})^\opp)$-bimodule (Lemma \ref{le:projL}),
it is enough to show that the second map is an isomorphism after
multiplication by $T_{w[1,m]}\otimes T_{w[1,n]}$ (Lemma \ref{le:isoandTw0}).

\smallskip
We prove the Lemma by induction on $n+m$. Note that the Lemma holds trivially
when $n=0$ or $m=0$ as well as when $(m,n)=(1,1)$. So, we can assume $m+n\ge 3$.

\medskip
$\bullet$ Let us first consider the case $m-n+\lambda=0$. Applying the Chevalley
duality if necessary, we can assume that $n>1$.
By induction, we have isomorphisms
\begin{multline}
\label{eq:step0}
E^mF^n\oplus \bar{L}(m,n-1,1,\lambda-2)\otimes_{{^0H}_{m-1}^f\otimes
({^0H}_{n-2}^f)^\opp} E^{m-1}F^{n-1}\iso
F^{n-1}E^mF \\
\iso F^nE^m\oplus \bar{L}'(m,1,1,\lambda)\otimes_{{^0H}_{m-1}^f}
F^{n-1}E^{m-1}.
\end{multline}
By Lemma \ref{le:sigmamn}, we have a
commutative diagram
$$\xymatrix{
E^{m-1}F\ar[r]^-{\eta\bullet} \ar[d]_-{\sigma_{m-1,1}} &
 FE^mF\ar[r]^-{F\sigma_{m,1}} & F^2E^m \\
FE^{m-1}\ar[r]_-{F\eta\bullet} & F^2E^m\ar[ur]_-{TE^m}
}$$
It follows that the composition of maps in (\ref{eq:step0}) has one of
its components equal to
\begin{multline}
\label{eq:step00}
\mathrm{act}\circ
(T_{n-1}E^m)\circ (F^{n-1}\eta E^{m-1})\circ (F^{n-2}\sigma_{m-1,n-1}):\\
\bar{L}(m,n-1,1,\lambda-2)\otimes_{{^0H}_{m-1}^f\otimes
({^0H}_{n-2}^f)^\opp} E^{m-1}F^{n-1}\to F^nE^m.
\end{multline}
We have $(T_{w[1,m]}\otimes T_{w[1,n]}^\opp)\bar{L}(m,n-1,1,\lambda-2)=
(T_{w[1,m]}\otimes T_{w[1,n]})\BZ$ and it follows that the map in
(\ref{eq:step00}) vanishes after multiplication by 
$(T_{w[1,m]}\otimes T_{w[1,n]}^\opp)$. We deduce that the component
$\sigma_{m,n}:E^mF^n\to F^nE^m$ of the composition of maps in (\ref{eq:step0})
is an isomorphism.

\medskip
$\bullet$ We consider now the case $n=1$ and $m+\lambda\le 0$.
By induction,
we have an isomorphism
$$E^{m-1}FE\oplus L([2,m],1,1,\lambda+2)\otimes E^{m-1} \iso FE^m.$$
So, we have an isomorphism
\begin{equation}
\label{eq:step1}
E^mF\oplus L(1,1,1,\lambda)\otimes E^{m-1}\oplus 
L([2,m],1,1,\lambda+2)\otimes E^{m-1} \iso FE^m
\end{equation}
Taking the image under the Chevalley duality of the commutative
diagram of Lemma \ref{le:sigmamn}, we obtain a commutative diagram
$$\xymatrix{
E^{m-1}\ar[r]^-{\bullet\eta} \ar[dr]_-{\eta\bullet} &
 E^{m-1} FE \ar[r]^-{\sigma_{m-1,1}E} & FE^m \\
& FE^m \ar[ur]_-{F(T_1\cdots T_{m-1})}
}$$
It follows that the isomorphism (\ref{eq:step1}) induces an isomorphism
$(\sigma_{m,1},\mathrm{act}\circ (\id\otimes (\eta E^{m-1})),
(\mathrm{act}\circ (\id\otimes (\eta E^{m-2})))E)$:
$$E^mF\oplus \left(\bigoplus_{0\le i<-\lambda}X_1^iT_1\cdots T_{m-1}\BZ
\right)\otimes E^{m-1}\oplus L([2,m],1,1,\lambda+2)\otimes E^{m-1} \iso FE^m.$$
Let 
$$M=\left(\bigoplus_{0\le i<-\lambda}X_1^iT_1\cdots T_{m-1} {^0H}_{m-1}\right)
\oplus \bigoplus_{\substack{w\in\GS_{[2,m]}^{[2,m-1]}\\ l< -m-\lambda}}
T_w(X_m^l){^0H}_{m-1}.$$
This is a ${^0H}_m^f$-submodule of ${^0H}_m$.
We have
\begin{align*}
T_{w[1,m]}M&=T_{w[1,m]}\left(\sum_{i<-\lambda}
\partial_{s_{m-1}\cdots s_1}(X_1^i){^0H}_{m-1}
+\sum_{i<-m-\lambda}(X_m^i) {^0H}_{m-1}\right)\\
&=
T_{w[1,m]}\sum_{i\le -\lambda-m}(X_m^i){^0H}_{m-1}\\
&= T_{w[1,m]}\bar{L}(m,1,1,\lambda){^0H}_{m-1}.
\end{align*}

Note that $M$ is generated by 
$\dim_\BZ L(m,1,1,\lambda)$ elements as a right ${^0H}_{m-1}$-module.
Since $\bar{L}(m,1,1,\lambda){^0H}_{m-1}$ is a free right ${^0H}_{m-1}$-module
of rank $\dim_\BZ L(m,1,1,\lambda)$, it follows that
$M$ is a free right ${^0H}_{m-1}$-module of that rank.
We have an isomorphism
\begin{equation}
\label{eq:step2}
(\sigma_{m,1},\mathrm{act}\circ (\id\otimes (\eta E^{m-1}))):
E^mF\oplus M\otimes_{{^0H}_{m-1}}E^{m-1}\iso FE^m.
\end{equation}
The morphism
\begin{equation}
\label{eq:step3}
(\sigma_{m,1},\mathrm{act}\circ (\id\otimes (\eta E^{m-1}))):
E^mF\oplus \bar{L}(m,1,1,\lambda)\otimes_{{^0H}_{m-1}^f}E^{m-1}\iso FE^m
\end{equation}
becomes an isomorphism after multiplication by $T_{w[1,m]}$, since it
coincides with the multiplication by $T_{w[1,m]}$ of the isomorphism
(\ref{eq:step2}). It follows from Lemmas \ref{le:projL} and
\ref{le:isoandTw0} that  the
morphism (\ref{eq:step3}) is an isomorphism and the lemma is proven
when $n=1$.

\medskip
$\bullet$ We consider finally the case $n>1$ and $m-n+\lambda<0$.
We have an isomorphism
$$F\bigoplus_{i=0}^{n-1} L(m,n-1,i,\lambda)\otimes E^{m-i}F^{n-i-1}
\iso F^nE^m.$$
The case $n=1$ of the lemma gives isomorphisms
$$\left(L(m-i,1,1,\lambda-2(n-i-1))
\otimes E^{m-i-1}\right)F^{n-i-1}\oplus
E^{m-i}F^{n-i}\iso FE^{m-i}F^{n-i-1}.$$
Combining the previous two isomorphisms, we obtain a isomorphism
$$\bigoplus_{i=0}^{n} \bigl(L(m,[2,n],i,\lambda)\oplus
L(m,[2,n],i-1,\lambda)\otimes L(m-i+1,1,1,\lambda-2(n-i))\bigr)\otimes
E^{m-i}F^{n-i}\iso F^nE^m.$$
In that isomorphism, the map
$L(m,[2,n],i,\lambda)\otimes E^{m-i}F^{n-i}\to F^nE^m$ is
$\mathrm{act}\circ \Bigl(\id\otimes\bigl((F^{n-i}\eta_iE^{m-i})\circ
\sigma_{m-i,n-i}\bigr)\Bigr)$. It follows from Lemma \ref{le:sigmamn}
that the map
$$L(m,[2,n],i-1,\lambda)\otimes L(m-i+1,1,1,\lambda-2(n-i))\otimes
E^{m-i}F^{n-i}\to F^nE^m$$
is
$$\mathrm{act}\circ
\biggl(\id\otimes\Bigl(\mathrm{act}\circ\bigl((T_{n-i}\cdots T_1)E^{m-i+1}\bigr)
\Bigr)\biggr)
\circ\Bigl(\id\otimes \id\otimes \bigl((F^{n-i}\eta_iE^{m-i})\circ
\sigma_{m-i,n-i}\bigr)\Bigr).$$
Let $i>0$ and
$$M_i=L(m,[2,n],i,\lambda)\oplus
\left(\bigoplus_{l<-r+n-i}T_{n-i}\cdots T_1X_1^l\BZ\right)
\cdot L(m,[2,n],i-1,\lambda)\cdot
\left(\bigoplus_{1 \le j\le m-i} T_j\cdots T_{m-i}\BZ\right),$$
a subgroup of ${^0H}_m\otimes_{{^0H}_i}{^0H}_n$.
We have shown that there is an isomorphism
$$\mathrm{act}\circ \Bigl(\id\otimes\bigl((F^{n-i}\eta_iE^{m-i})\circ
\sigma_{m-i,n-i}\bigr)\Bigr):
\bigoplus_{i=0}^{n} M_i\otimes E^{m-i}F^{n-i}\iso F^nE^m.$$
Let 
$$N_i=
\bigoplus_{\substack{w'\in {^{[2,n-i+1]}\GS}_{[2,n]}\\ 
\mu\in\CP(i-1,-r-i)\\ l<-r+n-i}}
T_{n-i}\cdots T_1 X_1^l m_\mu(X_{n-i+2},\ldots,X_n)
X_{n-i+2}^{i-2}\cdots X_{n-1} T_{w'}\BZ$$
and 
$$N'_i=\bigoplus_{\substack{w'\in {^{[2,n-i]}\GS}_{[2,n]}\\
 \mu\in\CP(i,-r-1-i)}}
m_\mu(X_{n-i+1},\ldots, X_n)X_{n-i+1}^{i-1}\cdots X_{n-1} T_{w'}\BZ.$$
We have
$$M_i=
\left(\bigoplus_{w\in \GS_m^{m-i}}T_w\BZ\right)\otimes 
\left(
N'_i\oplus N_i\right).$$

We have
$$T_{w[n-i+1,n]}N'_i T_{w[1,n]}=
\sum_{ \mu\in\CP(i,-r-1-i)}m_\mu(X_{n-i+1},\ldots, X_n) T_{w[1,n]}\BZ
$$
and
\begin{multline*}
T_{w[n-i+1,n]}N_i T_{w[1,n]}=\\
=\sum_{\substack{\mu\in\CP(i-1,-r-i)\\ l<-r+n-i}}
\partial_{s_{n-1}\cdots s_{n-i+1}}(
\partial_{s_{n-i}\cdots s_1}(X_1^l) m_\mu(X_{n-i+2},\ldots,X_n)) T_{w[1,n]}\BZ\\
=\sum_{\substack{\mu\in\CP(i-1,-r-i)\\ l<-r+n-i}}
(\sum_{k\le l-n+i}
P_{k,l}(X_1,\ldots,X_{n-i})
R_{k,\mu}(X_{n-i+1},\ldots,X_n))
 T_{w[1,n]}\BZ
\end{multline*}
where $P_{k,l}$ is a symmetric polynomial
and $R_{k,\mu}=
\partial_{s_{n-1}\cdots s_{n-i+1}}(X_{n-i+1}^k m_\mu(X_{n-i+2},\ldots,X_n))$
satisfies $\deg_* R_{k,\mu}\le \max(k-i+1,-r-i-1)$
by Lemma \ref{le:diffmonomial}.

Let us fix $k$ and $l$.
By induction, the composite morphism
$$E^{m-i}F^{n-i}\xrightarrow{\sigma_{m-i,n-i}} F^{n-i}E^{m-i}
\xrightarrow{P_{k,l}E^{m-i}}F^{n-i}E^{m-i}$$
is equal to
$$\sum_{\substack{j\ge i\\ \mu'\in\CP(j-i,-r-j+i)}}
\!\!\!\!\!\!\!\!\!\!\!\!
((m_{\mu'}(X_{n-j+1},\ldots,X_{n-i})X_{n-j+1}^{j-i-1}\cdots
X_{n-i-1})E^{m-i})\circ
(\id\otimes((F^{n-j}\eta_{j-i}E^{m-j})\circ\sigma_{m-j,n-j}))
\circ f_{j,\mu'}$$
for some $f_{j,\mu'}:E^{m-i}F^{n-i}\to E^{m-j}F^{m-j}$.
We have
\begin{multline*}
T_{w[n-j+1,n]w[n-i+1,n]}m_{\mu'}(X_{n-j+1},\ldots,X_{n-i})
X_{n-j+1}^{j-i-1}\cdots X_{n-i-1}
R_{k,\mu}(X_{n-i+1},\ldots,X_n) T_{w[1,n]}=\\
=
\partial_{w[n-j+1,n]w[n-i+1,n]w[n-j+1,n-i]}(m_{\mu'}(X_{n-j+1},\ldots,X_{n-i})
R_{k,\mu}(X_{n-i+1},\ldots,X_n)) T_{w[1,n]}=\\
=S_{k,\mu,\mu'}(X_{n-j+1},\ldots,X_n)T_{w[1,n]},
\end{multline*}
where $S_{k,\mu,\mu'}$ is a symmetric polynomial and
$\deg_* S_{k,\mu,\mu'}\le -r-j$ by Lemma \ref{le:diffmonomial}. Note that
if $j=i$ and $k\not=-r-1$, then $\deg_* S_{k,\mu,\mu'}\le -r-i-1$.

Assume $l=-r+n-i-1$ and $k=l-n+i=-r-1$. We have
$$R_{k,\mu}=\partial_{s_{n-1}\cdots s_{n-i+1}}
(X_{n-i+1}^{-r-1} m_\mu(X_{n-i+2},\ldots,X_n))
=m_{\mu\cup\{-r-i\}}(X_{n-i+1},\ldots,X_n)+T(X_{n-i+1},\ldots,X_n),$$
where $T$ is a symmetric polynomial with $\deg_*T\le -r-i-1$
(Lemma \ref{le:diffmonomial}).

\smallskip
We have shown that the images of $L(m,n,i,\lambda)$ and of
$M_i$ in $\Hom(E^{m-i}F^{n-i},F^{(n)}E^{(m)})$ coincide modulo maps that 
factor through 
$$\bigoplus_{j>i} (T_{w[1,m]}T_{w[1,n]}^\opp)
L(m,n,i,\lambda)\otimes E^{m-j}F^{n-j}\to F^{(n)}E^{(m)}.$$
Using Lemma \ref{le:filtrationiso},
we deduce by descending induction on $i$ that the lemma holds, using
that $\dim_\BZ M_i=\dim_\BZ L(m,n,i,\lambda)$ as in the case
$n=1$ considered earlier.
\end{proof}

\begin{rem}
Let $\hat{\CB}_1$ the $k$-linear category
$\CB\times\CB^{\opp}$. Denote by
$F_s$ the object $E_s$ of $\CB^{\opp}$ and define
$\hat{h}:\Ob(\hat{\CB}_1)\to X,\ (M,N)\mapsto h(M)+h(N)$.
Consider the $2$-category $\hat{\FA}_1$ with set of objects $X$ and
$\CHom(\lambda,\lambda')=\hat{h}^{-1}(\lambda'-\lambda)$. The isomorphisms
of Lemma \ref{le:basiccommutation}, together with $\sigma_{st}$ for
$s\not=t$, are the first steps
to provide a direct construction of a composition on the
homotopy category of $\hat{\FA}_1$ (after adding maps
$(M\otimes E_s,F_s\otimes N)\to (M,N)$).
\end{rem}

\subsubsection{Decomposition of $[E_s^{(m)},F_t^{(n)}]$}

\begin{lemma}
\label{le:commutationEF}
Let $s\in I$ and $m,n\in\BZ_{\ge 0}$. Let 
$r=m-n+\langle\alpha_s^\vee,\lambda\rangle$.
We have the following isomorphisms in $\FA^i$
and in $\FA^{\bullet i}$:
$$E_s^{(m)}F_s^{(n)}\simeq
\bigoplus_{i\in \BZ_{\ge 0}} 
\begin{bmatrix}r\\i\end{bmatrix}_s
F_s^{(n-i)}E_s^{(m-i)} 
\text{ if }r\ge 0
$$
$$E_s^{(m)}F_s^{(n)}\oplus
\bigoplus_{i\in 1+2\BZ_{\ge 0}} 
\begin{bmatrix}i-1-r \\ i\end{bmatrix}_s
F_s^{(n-i)}E_s^{(m-i)} \simeq 
\bigoplus_{i\in 2\BZ_{\ge 0}} 
\begin{bmatrix}i-1-r\\i\end{bmatrix}_s
F_s^{(n-i)}E_s^{(m-i)} 
\text{ if }r<0
$$
$$F_s^{(n)}E_s^{(m)}\oplus
\bigoplus_{i\in 1+2\BZ_{\ge 0}} 
\begin{bmatrix}i-1+r \\ i\end{bmatrix}_s
E_s^{(m-i)}F_s^{(n-i)} 
\simeq\bigoplus_{i\in 2\BZ_{\ge 0}} 
\begin{bmatrix}i-1+r\\i\end{bmatrix}_s
E_s^{(m-i)}F_s^{(n-i)} 
\text{ if }r>0
$$

$$F_s^{(n)}E_s^{(m)}\simeq
\bigoplus_{i\in \BZ_{\ge 0}} 
\begin{bmatrix}-r\\i\end{bmatrix}_s
E_s^{(m-i)}F_s^{(n-i)} 
\text{ if }r\le 0
$$

Let $t\in I-\{s\}$ and $m,n\in\BZ_{\ge 0}$. We have the following isomorphisms
in $\FA^i$ and in $\FA^{\bullet i}$:
$$E_s^{(m)}F_t^{(n)}\simeq F_t^{(n)}E_s^{(m)}.$$
\end{lemma}

\begin{proof}
The first isomorphism follows from the isomorphism of
$({^0H}_m\otimes {^0H}_n^\opp)$-modules in Lemma \ref{le:basiccommutation}.
Assume $r<0$. Given $l\in\BZ_{>0}$, we have (cf e.g. \cite[1.3.1(e) p.9]{Lu})
$$\sum_i (-1)^i \begin{bmatrix}i-1-r\\i\end{bmatrix}\cdot
\begin{bmatrix}-r\\l-i\end{bmatrix}=0.$$
It follows that
\begin{align*}
\bigoplus_{\substack{i\in 1+2\BZ_{\ge 0} \\ j\ge 0}}
\begin{bmatrix}i-1-r\\i\end{bmatrix}\cdot
\begin{bmatrix}-r\\j\end{bmatrix} E^{(m-i-j)}F^{(n-i-j)}&\simeq
\bigoplus_{\substack{i\in 2\BZ_{>0} \\ j\ge 0}}
\begin{bmatrix}i-1-r\\i\end{bmatrix}\cdot
\begin{bmatrix}-r\\j\end{bmatrix} E^{(m-i-j)}F^{(n-i-j)}\oplus \\
&\oplus\bigoplus_{i\ge 1} \begin{bmatrix}-r\\i\end{bmatrix} E^{(m-i)}F^{(n-i)},
\end{align*}
hence
$$\bigoplus_{i\in 1+2\BZ_{\ge 0}}\begin{bmatrix}i-1-r\\i\end{bmatrix}
F^{(n-i)}E^{(m-i)}\simeq
\bigoplus_{i\in 2\BZ_{>0}}\begin{bmatrix}i-1-r\\i\end{bmatrix}
F^{(n-i)}E^{(m-i)}
\oplus\bigoplus_{i\ge 1} \begin{bmatrix}-r\\i\end{bmatrix} E^{(m-i)}F^{(n-i)}
$$
using the first isomorphism of the lemma for $(m-i,n-i)$.
The second isomorphism of the lemma
follows by applying again the first isomorphism.

The third and fourth isomorphism follow from the second and first by
applying the Chevalley involution.

\smallskip
The isomorphisms $\sigma_{st}$ induce an isomorphism
$E_s^mF_t^n\iso F_t^nE_s^m$ compatible with the action of
${^0H_m}\otimes {^0H_n}$ (the proof in Lemma
\ref{le:sigmamn} works when $s\not=t$).
It follows that
$E_s^{(m)}F_t^{(n)}\simeq F_t^{(n)}E_s^{(m)}$.
\end{proof}

\subsubsection{Decategorification}
Proposition \ref{pr:Serre}
shows that we have a morphism of algebras
$$U_1^+(\Gg)\to\CB^i_{\le 1},\ e_s^{(r)}\mapsto [E_s^{(r)}]$$
and, when $C$ is symmetrizable, to a morphism of $\BZ[v^{\pm 1}]$-algebras
$$U_v^+(\Gg)\to\CB^{\bullet i}_{\le 1},\ e_s^{(r)}\mapsto [E_s^{(r)}].$$

\smallskip
The defining relations for $\FA$ show that 
we have a monoidal functor:
$$U_1(\Gg)\to \FA^i_{\le 1},\ \ 
\lambda\mapsto\lambda,\ e_s^{(r)}\mapsto [E_s^{(r)}],\  f_s^{(r)}\mapsto
 [F_s^{(r)}]$$
and, when $C$ is symmetrizable,
a monoidal functor compatible with the $\BZ[v^{\pm 1}]$-structure:
$$U_v(\Gg)\to \FA^{\bullet i}_{\le 1},\ \ 
\lambda\mapsto\lambda,\ e_s^{(r)}\mapsto [E_s^{(r)}],\  f_s^{(r)}\mapsto
 [F_s^{(r)}].$$

\section{$2$-Representations}
We assume in this section that the set $I$ is finite.
All results are stated over $\Bk$ and are related to
representations of $\Gg$. They generalize immediately to the graded case over
$\Bk^\bullet$ and relate then to representations of $U_v(\Gg)$.

\subsection{Integrable representations}
\subsubsection{Definition}
\label{se:defintegrable}
Let $\FB$ be a $\Bk$-linear $2$-category.

Given $R:\FA\to\FB$ a $2$-functor, we have a collection
$\{R(\lambda)\}_{\lambda\in X}$ of objects of $\FB$.
We say that $R$ gives 
a $2$-representation of $\FA$
on $\{R(\lambda)\}$. If this makes sense,
we put $\CV=\bigoplus_{\lambda\in X} R(\lambda)$ and say that we
have a $2$-representation of $\FA$ on $\CV$.

\smallskip
The data of a strict 
$2$-functor $R:\FA\to\FB$ is the same as the data of
\begin{itemize}
\item a family $(\CV_\lambda)_{\lambda\in X}$ of objects of $\FB$
\item $1$-arrows
$E_{s,\lambda}:\CV_{\lambda}\to\CV_{\lambda+\alpha_s}$
and
$F_{s,\lambda}:\CV_{\lambda}\to\CV_{\lambda-\alpha_s}$ for $s\in I$
\item $x_{s,\lambda}\in\End(E_{s,\lambda})$ and
$\tau_{s,t,\lambda}\in\Hom(E_{s,\lambda+\alpha_t}E_{t,\lambda},
E_{t,\lambda+\alpha_s}E_{s,\lambda})$ for $s,t\in I$
\item an adjunction $(E_{s,\lambda},F_{s,\lambda+\alpha_s})$
\end{itemize}
such that
\begin{itemize}
\item relations (\ref{en:half1})-(\ref{en:half4}) in \S \ref{se:halfKM} hold
\item the maps $\rho_{s,\lambda}$ and $\sigma_{st}$ for $s\not=t$
are isomorphisms.
\end{itemize}

Note that there are canonical strict $2$-functors that are locally
equivalences
$$\bar{\FA}'\mMOD(\FB)\to \FA'\mMOD(\FB)\to \FA\mMOD(\FB).$$

\smallskip
From now on, we assume $\FB$ is a locally full $2$-subcategory of
$\FL in_{\Bk}$.

\begin{defi}
A $2$-representation $\FA\to\FB$ is {\em integrable} if $E_s$ and
$F_s$ are locally nilpotent for all $s$, \ie, for any $\lambda$ and
any object $M$ of the category $\CV_\lambda$, there is an integer $n$
such that $E_{s,\lambda+n\alpha_s}\cdots E_{s,\lambda+\alpha_s}
E_{s,\lambda}(M)=0$ and
$F_{s,\lambda-n\alpha_s}\cdots F_{s,\lambda-\alpha_s}
F_{s,\lambda}(M)=0$.
\end{defi}

Our main object of study is the $2$-category of integrable $2$-representations
of $\FA$ in $\Bk$-linear, abelian, triangulated and dg-categories.

Given $\CV$ a $2$-representation of $\FA$, then we endow
$\CV^\opp$ with the structure of a $2$-representation of $\FA$ by using
the Chevalley involution $I$, with $(\CV^\opp)_\lambda=(\CV_{-\lambda})^\opp$.

Let $\CV$ be an integrable $2$-representation of $\FA$ in $\FL in_\Bk$. 
There is an induced action of $\Ho^b(\FA)$ on $\Ho^b(\CV)$.

\begin{lemma}
\label{le:invertcusp}
Let $s\in I$.
\begin{itemize}
\item
Let $C\in \Ho^b(\CV)$.
If $\Hom_{\Ho^b(\CV)}(E_s^iM,C)=0$ for all $M\in \Ho^b(\CV)$ such that $F_sM=0$
and all $i\ge 0$, then $C=0$.
\item
Let $X$ be a $1$-arrow
of $\Ho^b(\FA)$ with a right dual. If $XE_s^i(M)=0$ 
for all $M\in \Ho^b(\CV)$ such that $F_sM=0$ and all $i\ge 0$,
then $X(N)=0$ for all $N\in \Ho^b(\CV)$.
\item
Let $f$ a $2$-arrow of $\Ho^b(\FA)$ between $1$-arrows with right duals.
 If $f(E_s^i M)$ is an isomorphism
for all $M\in \Ho^b(\CV)$ such that $F_sM=0$ and all $i\ge 0$,
then $f(N)$ is an isomorphism for all $N\in \Ho^b(\CV)$.
\end{itemize}
\end{lemma}

\begin{proof}
Let $i$ be a maximal integer such that $F_s^iC\not=0$. 
We have
$$\End(F_s^iC)\simeq \Hom(E_s^iF_s^i C,C)=0,$$
hence a contradiction and consequently $C=0$.

\smallskip
Let $X^\vee$ be a right dual of $X$. Let $M,N\in \Ho^b(\CV)$ such that $F_sM=0$
and let $i\ge 0$. We have
$$\Hom(E_s^i(M),X^\vee\cdot X(N))\simeq \Hom(X E_s^i(M), X(N))=0$$
and we deduce from the first statement of the Lemma that $X^\vee X(N)=0$,
hence $X(N)=0$.

\smallskip
The last assertion follows from the second one by taking for $X$ the
cone of $f$.
\end{proof}

\subsubsection{Simple $2$-representations}
\label{se:simple2rep}
We assume that the root datum is $Y$-regular, \ie, the image
of the embedding $I\to Y$ is linearly independent in $Y$
(cf \cite[\S 2.2.2]{Lu}). Let $X^+=\{\lambda\in X |
\langle\alpha_i^\vee,\lambda\rangle\in\BZ_{\ge 0}\ \text{ for all }
i\in I\}$. The set $X$ is endowed with a poset structure defined by
$\lambda\ge\mu$ if $\lambda-\mu\in\bigoplus_{i\in I}\BZ_{\ge 0}\alpha_i^\vee$.

\smallskip
Let $\lambda\in -X^+$. Consider the $2$-functor
$\CHom(\lambda,-):\FA\to \FL in_\Bk$ and let
$R:\FA\to\FL in_\Bk$ be
the $2$-subfunctor generated by the $F_{s,\lambda}$
for $s\in I$, \ie, $R(\mu)$ is the $\Bk$-linear full subcategory 
of $\CHom(\lambda,\mu)$ with
objects in $h^{-1}(\mu-\lambda+\alpha_s)F_s$.
We denote by $\CL(\lambda)$ the quotient $2$-functor, viewed
as a $\Bk$-linear category endowed with a decomposition
$\CL(\lambda)=\bigoplus_{\mu\in X}\CL(\lambda)_\mu$ and endowed
with an action of $\FA$.

Denote by $\bar{\idun}_\lambda$ the identity functor of $\CL(\lambda)_\lambda$.
It follows from Lemma \ref{le:basiccommutation} that 
$F_sE_s^{\langle\alpha_s^\vee,-\lambda\rangle+1}\idun_\lambda$
is isomorphic to a direct summand of
 $E_s^{\langle\alpha_s^\vee,-\lambda\rangle+1}F_s\idun_\lambda$. In
particular, $F_sE_s^{\langle\alpha_s^\vee,-\lambda\rangle+1}\bar{\idun}_\lambda
=0$. The isomorphism
$$\End(E_s^{\langle\alpha_s^\vee,-\lambda\rangle+1}\bar{\idun}_\lambda)
\simeq \Hom(
E_s^{\langle\alpha_s^\vee,-\lambda\rangle}\bar{\idun}_\lambda,
F_sE_s^{\langle\alpha_s^\vee,-\lambda\rangle+1}\bar{\idun}_\lambda)$$
shows that 
$$E_s^{\langle\alpha_s^\vee,-\lambda\rangle+1}\bar{\idun}_\lambda=0.$$

Since $F_sE_t\idun_\mu$ is a direct summand of $E_tF_s\idun_\mu$ plus a
multiple of $\idun_\mu$, it follows that
every object of $\CL(\lambda)$ is isomorphic to a direct summand of a sum of
objects of the form $E_{s_1}\cdots E_{s_n}\bar{\idun}_\lambda$ for
some $s_1,\ldots,s_n\in I$. In particular, every object of
$\CL(\lambda)_\lambda$ is isomorphic to a direct summand of a multiple of
$\bar{\idun}_\lambda$. Since $\End(\bar{\idun}_\lambda)$ is a quotient
of $\End(\idun_\lambda)$, it is commutative and
$\CL(\lambda)_\lambda$ is equivalent to a full subcategory of
$\End(\bar{\idun}_\lambda)\mproj$.

Note that when $C$ is a symmetrizable Cartan matrix, then
$\BC\otimes K_0(\CL(\lambda))$ is isomorphic to the simple
integrable representation of $\Gg$ with lowest weight $\lambda$
\cite[Corollary 10.4]{Kac}, or
it is $0$. We will show in \cite{Rou3} that it is indeed
non zero and determine $\End(\bar{\idun}_\lambda)$.

\subsubsection{Lowest weights}
\label{se:lowest}
Let $A$ be an $\End(\bar{\idun}_\lambda)$-algebra. Let
$\CV=\CL(\lambda)\otimes_{\End(\bar{\idun}_\lambda)} A$, given by
$\CV_\mu=\CL(\lambda)_\mu\otimes_{\End(\bar{\idun}_\lambda)}A$,
where the map $\End(\bar{\idun}_\lambda)\to Z(\CL(\lambda)_\mu)$ is
given by right multiplication.  The action of $\FA$ on $\CL(\lambda)$
extends to an action on $\CV$. Similarly, if
$\CA$ is a $\End(\bar{\idun}_\lambda)$-linear category, we have an
action of $\FA$ on 
$\CL(\lambda)\otimes_{\End(\bar{\idun}_\lambda)} \CA$.

\medskip
Let $\CV$ be a $2$-representation of $\FA$ in
$\FL in_{\Bk}$. Given $\lambda\in X$, we denote by
$\CV_\lambda^{\mathrm{lw}}$ the full subcategory of $\CV_\lambda$ of objects
$M$ such that $F_sM=0$ for all $s\in I$.

\begin{lemma}
If $\CV_\lambda^{\mathrm{lw}}\not=0$, then $\lambda\in -X^+$.
\end{lemma}

\begin{proof}
Assume there is $s\in I$ such that
$\langle \alpha_s^\vee,\lambda\rangle>0$ and let
 $M\in\CV_\lambda^{\mathrm{lw}}$.
Then, $M$ is a direct summand of $E_sF_s M=0$.
\end{proof}

Assume $\lambda\in -X^+$.
The canonical morphism of $2$-representations $\CHom(\lambda,-)\to \CV$ 
associated with $M\in\CV_\lambda^{\mathrm{lw}}$
factors through a morphism $R_M:\CL(\lambda)\to \CV$. Note that
$R_M(\bar{\idun}_\lambda)=M$. So, we have a morphism of algebras
$\End(\bar{\idun}_\lambda)\to \End(M)$ and this shows that the morphism
above extends to a morphism of $2$-representations
$R_M:\CL(\lambda)\otimes_{\End(\bar{\idun}_\lambda)}\End(M)\to \CV$.
We have also a canonical morphism of $2$-representations
$R_\lambda:\CL(\lambda)\otimes_{\End(\bar{\idun}_\lambda)}
\CV_\lambda^{\mathrm{lw}}\to \CV$ that extends $R_M$.

\begin{lemma}
\label{le:highestweight}
Let $\CV$ be a $2$-representation of $\FA$ in
$\FL in_{\Bk}$ and $\lambda\in -X^+$.
The morphism of $2$-representations
$$R_\lambda:
\CL(\lambda)\otimes_{\End(\bar{\idun}_\lambda)}
\CV_\lambda^{\mathrm{lw}} \to \CV$$
is fully faithful.
\end{lemma}

\begin{proof}
Let $\lambda\in -X^+$ and $M\in\CV_\lambda^{\mathrm{lw}}$.
Let $\CL_M(\lambda)=\CL(\lambda)\otimes_{\End(\bar{\idun}_\lambda)}\End(M)$.
Let $X$ be an object of $\CHom(\lambda,\mu)$ that is a product
of $E_s$'s. There is an
object $Y$ of $\CHom(\mu,\lambda)$ right dual to $X$.
We have a commutative diagram of canonical maps
$$\xymatrix{
\End_{\CL_M(\lambda)}(X \bar{\idun}_\lambda)\ar[r]^-\sim \ar[d] &
\Hom_{\CL_M(\lambda)}(\bar{\idun}_\lambda,YX\bar{\idun}_\lambda)\ar[d] \\
\End_\CV(XM)\ar[r]_-\sim & \Hom_\CV(M,YXM)
}$$
Since $YX\bar{\idun}_\lambda$ is isomorphic to a direct summand of
a multiple of 
$\bar{\idun}_\lambda$, the right vertical map is an isomorphism, hence
the left vertical map is an isomorphism as well. It follows that
$R_M$ is fully faithful, hence $R_\lambda$ is fully faithful as well.
\end{proof}

\begin{lemma}
\label{le:0onsimple}
Let $C$ be a $1$-arrow of $\Ho^b(\FA)$ with a right dual. If $C$ acts by
$0$ on $\Ho^b(\CL(\lambda))$ for all $\lambda\in -X^+$, then $C$ acts
by $0$ on $\Ho^b(\CV)$ for all integrable $2$-representations 
$\CV$ of $\FA$ in $\FL in_\Bk$.

Let $f$ be a $2$-arrow of $\Ho^b(\FA)$ between $1$-arrows with right duals.
 If $f$ is an isomorphism on $\Ho^b(\CL(\lambda))$ for all $\lambda\in -X^+$,
then $f$ is an isomorphism on $\Ho^b(\CV)$ for all integrable
 $2$-representations $\CV$ of $\FA$ in $\FL in_\Bk$.
\end{lemma}

\begin{proof}
Let $M\in\CV_{\lambda}$ such that $FM=0$. Lemma 
\ref{le:highestweight} provides a fully faithful morphism of
$2$-representations
$$R:\CV(\lambda)\otimes_{\End(\bar{\idun}_\lambda)}\End(M)\to
\CV$$
with $R(\bar{\idun}_\lambda)\simeq M$. We deduce that $C(E^iM)=0$
for all $i$. This holds also for $\CV$ replaced by $\Ho^b(\CV)$ and
Lemma \ref{le:invertcusp} shows that $C$ acts by $0$ on $\CV$.

The second statement follows by taking for $C$ the cone of $f$.
\end{proof}

\begin{prop}
\label{pr:highestweight}
Let $\CV$ be a $2$-representation of $\FA'$ in
$\FL in_{\Bk}$ and $\lambda\in -X^+$.
The morphism of $2$-representations
$$\sum_{\lambda\in -X^+}R_\lambda:
\bigoplus_{\lambda\in -X^+}\CL(\lambda)\otimes_{\End(\bar{\idun}_\lambda)}
\CV_\lambda^{\mathrm{lw}} \to \CV$$
is fully faithful.
\end{prop}

\begin{proof}
Let $\lambda\in -X^+$ and $M\in\CV_\lambda^{\mathrm{lw}}$.
Consider $\mu\in -X^+$ with $\mu\not=\lambda$ and
let $N\in \CV_\mu^{\mathrm{lw}}$.
Let $s_1,\ldots,s_m,t_1,\ldots,t_n\in I$ such that
$\alpha_{s_1}+\cdots+\alpha_{s_m}+\lambda=
\alpha_{t_1}+\cdots+\alpha_{t_n}+\mu$.
If $m=0$, then we have
$$\Hom(M,E_{t_1}\cdots E_{t_n}N)\simeq \Hom(F_{t_1}M,
E_{t_2}\cdots E_{t_n}N)=0.$$
Assume $m>0$.
Since $F_{s_1}E_{t_1}\cdots E_{t_n}N$ is isomorphic to a direct summand of
a sum of objects of the form
 $E_{t_1}\cdots E_{t_{i-1}}E_{t_{i+1}}\cdots E_{t_n}N$ for $t_i=s_1$, it
follows by induction on $m$ that
$$\Hom(E_{s_1}\cdots E_{s_m}M,
E_{t_1}\cdots E_{t_n}N)=0.$$
So, there are no non-zero maps between an object in the image of $R_\lambda$
and an object in the image of $R_\mu$. Lemma \ref{le:highestweight}
provides the conclusion.
\end{proof}

An immediate consequence of Proposition \ref{pr:highestweight} 
is a decomposition result for additive $2$-representations generated 
by lowest weight objects.

\begin{cor}
\label{th:isotypicadditive}
Assume $\CV$ is an idempotent complete integrable $2$-representation and
every object of $\CV$ is a direct summand
of $XM$ for some object $X$ of $\FA'$ and $M\in\CV$ with
$F_iM=0$ for all $i$.

Then, there is an equivalence of $2$-representations
$$\sum_{\lambda\in -X^+}R_\lambda:
\bigoplus_{\lambda\in -X^+}
\left(\CL(\lambda)\otimes_{\End(\bar{\idun}_\lambda)}
\CV_\lambda^{\mathrm{lw}}\right)^i \iso \CV.$$
\end{cor}

\subsubsection{Jordan-H\"older series}
We denote by $\FO^{\mathrm{int}}(\FB)$
the $1,2$-full subcategory of $2$-representations $\CV$ of $\FA'$ in $\FB$
which are integrable and such that $\{\lambda\in -X^+ | \CV_\lambda\not=0\}$ is
bounded below (\ie, there is an integer $n$ such that given a sequence
$\lambda_1>\lambda_2>\cdots>\lambda_r$ of elements of $-X^+$ with
$\CV_{\lambda_i}\not=0$ for all $i$, then $r\le n$).

\begin{thm}
\label{th:JordanHolder}
Let $\CV$ be an idempotent complete $2$-representation in
$\FO^{\mathrm{int}}(\FL in_\Bk)$. There is a filtration by
full $\Bk$-linear $2$-subrepresentations
$$0=\CV\{0\}\subset \CV\{1\}\cdots\subset\cdots \subset\CV\{n\}=\CV,$$
there are $\End(\bar{\idun}_{\lambda})$-linear categories
$\CM_{\lambda,l}$ for $\lambda\in -X^+$ and isomorphisms of
$2$-representations
$$\CV\{l\}/\CV\{l-1\}\iso
 \bigoplus_{\lambda\in -X^+}
\left(\CL(\lambda)\otimes_{\End(\bar{\idun}_{\lambda})}
\CM_{\lambda,l}\right)^i.$$
\end{thm}

\begin{proof}
We proceed by induction on the maximal length of a sequence
$\lambda_1<\cdots<\lambda_n$ of elements of $-X^+$ such that
$\CV_{\lambda_i}\not=0$.
Let $L$ be the set of minimal elements $\lambda\in -X^+$ such that
$\CV_\lambda\not=0$.
Proposition \ref{pr:highestweight} gives a fully faithful morphism of
 $2$-representations
$$\bigoplus_{\lambda\in L}\CL(\lambda)\otimes_{\End(\bar{\idun}_\lambda)}
\CV_\lambda^{\mathrm{lw}} \to \CV$$
that is an equivalence on $\lambda$-weight spaces for $\lambda\in L$.
By induction, its cokernel satisfies the conclusion of the Theorem and
we are done.
\end{proof}

This theorem extends to abelian and (dg) triangulated settings, cf
\cite{Rou3}.

\subsubsection{Bilinear forms}
Assume $\CV$ is a $2$-representation of $\FA'$ in 
$\FT ri_k$, where $k$ is a field endowed with a $\Bk$-algebra structure.

The action of $\FA'$ on $\CV$ induces an action
of $U_1(\Gg)$ on $K_0(\CV)$. The same holds for $2$-representations in
abelian or exact categories.

\smallskip
Assume $\CV$ is $\Ext$-finite, \ie,
$\dim_k \bigoplus_{i\in\BZ} \Hom_{\CV}(M,N[i])<\infty$ for all
$M,N\in\CV$.

We have a pairing on $K_0(\CV)$:
$$K_0(\CV)\times K_0(\CV)\to\BZ,\ \langle [M],[N]\rangle=
\sum_i (-1)^i \dim_k \Hom(M,N[i]).$$

We have
$$\langle e_s(v),v'\rangle=\langle v,f_s(v')\rangle
\text{ and }
\langle f_s(v),v'\rangle=\langle v,e_s(v')\rangle.$$

Note in particular that if $L$ is a field such that
the pairing is perfect on $L\otimes K_0(\CV)$, then
$L\otimes K_0(\CV)$ is a semi-simple representation of
$L\otimes_{\BZ} U_1(\Gg)$.

\subsection{Simple $2$-representations of $\Gsl_2$}
\label{se:sl2}
\subsubsection{Symmetrizing forms}
Fix a positive integer $n$.
Let $i$ be an integer with $0\le i\le n$.
We put $P_i=k[X_1,\ldots,X_i]$.
We denote by $H_{i,n}$ the subalgebra of ${^0H_n}$ generated by
$T_1,\ldots,T_{i-1}$ and $P_n^{\GS[i+1,n]}$. This is the
same as the subalgebra generated by ${^0 H}_i$ and
$P_n^{\GS_n}$.
 We have a
decomposition as abelian groups
$$H_{i,n}=
{^0H}_i^f\otimes_\BZ P_n^{\GS[i+1,n]}.$$
and a decomposition as algebras
\begin{equation}
\label{eq:Hin}
H_{i,n}={^0H}_i\otimes_\BZ \BZ[X_{i+1},\ldots,X_n]^{\GS[i+1,n]}.
\end{equation}

\begin{lemma}
\label{le:formHin}
The algebra $H_{i,n}$ has a symmetrizing 
form over $P_n^{\GS_n}$ 
\begin{align*}
t_i:H_{i,n} & \to P_n^{\GS_n}(2i(i-n)) \\
P\cdot T_w\cdot w[1,i]&\mapsto
\partial_{w[1,n]\cdot w[i+1,n]}(P)\delta_{w,w[1,i]}
\end{align*}
for $w\in\GS_i$  and  $P\in P_n^{\GS[i+1,n]}$.
\end{lemma}

\begin{proof}
The decomposition
(\ref{eq:Hin}) shows that $H_{i,n}$ has a symmetrizing
form over $P_n^{\GS[1,i]\times\GS[i+1,n]}$
given by $P T_w w[1,i]\mapsto \partial_{w[1,i]}(P)\delta_{w,w[1,i]}$ for
$w\in\GS_i$ and $P\in P_n^{\GS[i+1,n]}$.

The algebra $P_n$ has a symmetrizing form over $P_n^{\GS_n}$ given by
$\partial_{w[1,n]}$ and a symmetrizing form over
$P_n^{\GS[1,i]\times\GS[i+1,n]}$ given by
$\partial_{w[1,i]}\partial_{w[i+1,n]}$.
It follows from Lemma \ref{le:trans3} that
the algebra $P_n^{\GS[1,i]\times\GS[i+1,n]}$ has a symmetrizing form
over $P_n^{\GS_n}$ given by $\partial_{w[1,n]\cdot w[1,i]\cdot w[i+1,n]}$.
The lemma follows now from Lemma \ref{le:trans1}.
\end{proof}

Let $e_i(\cdots)$ (resp. $h_i(\cdots)$) denote the elementary (resp. complete)
symmetric functions and put $e_i=h_i=0$ for $i<0$.

\begin{lemma}
\label{le:basisrelativeP}
The morphism $\partial_{s_{n-1}\cdots s_{i+1}}$ is a symmetrizing form
for the $P_n^{\GS[i+1,n]}$-algebra
$P_n^{\GS[i+2,n]}$. The set
$\{X_{i+1}^j\}_{0\le j\le n-i-1}$ is a basis, with dual basis
$\{(-1)^j e_{n-i-1-j}(X_{i+2},\ldots,X_n)\}$.
\end{lemma}

\begin{proof}
The first statement follows as in the proof of Lemma \ref{le:formHin} from
Lemma \ref{le:trans3}.
We have
$$\partial_{s_m}(h_j(X_1,\ldots,X_m))=-h_{j-1}(X_1,\ldots,X_{m+1})
\text{ and }
\partial_{s_m}(e_j(X_{m+1},\ldots,X_n))=e_{j-1}(X_m,\ldots,X_n).$$

Let $k,j\in [0,n-i-1]$. We have
$e_k(X_{i+2},\ldots,X_n))=e_k(X_{i+1},\ldots,X_n)-
X_{i+1}e_{k-1}(X_{i+2},\ldots,X_n)$, hence
\begin{align*}
\partial_{s_{n-1}\cdots s_{i+1}}(X_{i+1}^je_k(X_{i+2},\ldots,X_n))=&
(-1)^{n+i+1}h_{j-n+i+1}(X_{i+1},\ldots,X_n)e_k(X_{i+1},\ldots,X_n)-\\
&-\partial_{s_{n-1}\cdots s_{i+1}}(X_{i+1}^{j+1}e_{k-1}(X_{i+2},\ldots,X_n))
\end{align*}
By induction, we obtain
\begin{align*}
\partial_{s_{n-1}\cdots s_{i+1}}(X_{i+1}^je_k(X_{i+2},\ldots,X_n))=&
(-1)^{n+i+1}(h_{j-n+i+1}e_k-h_{j-n+i+2}e_{k-1}+\cdots+\\
&+(-1)^k h_{j-n+i+1+k}e_0),
\end{align*}
where we wrote $e_j$ and $h_j$ for the functions in the variables
$X_{i+1},\ldots,X_n$.
It follows from the fundamental relation between elementary and complete
symmetric functions that
$$\partial_{s_{n-1}\cdots s_{i+1}}(X_{i+1}^je_k(X_{i+2},\ldots,X_n))=0
\text{ if } j+k\not=n-i-1$$
while 
$$\partial_{s_{n-1}\cdots s_{i+1}}(X_{i+1}^je_{n-i-1-j}(X_{i+2},\ldots,X_n))=
(-1)^j.$$
\end{proof}

\subsubsection{Induction and restriction}
\label{se:mincatadjunction}
We have the usual canonical adjoint pair
$(\Ind_{H_{i,n}}^{H_{i+1,n}}, \Res_{H_{i,n}}^{H_{i+1,n}})$. The symmetric forms
on the algebras $H_{i,n}$ and $H_{i+1,n}$ described in Lemma
\ref{le:formHin} provide an adjoint pair
$(\Res_{H_{i,n}}^{H_{i+1,n}},\Ind_{H_{i,n}}^{H_{i+1,n}})$ and we will now
describe the units and counits of that pair, in terms of morphisms of
bimodules.

\smallskip

The following proposition gives a Mackey decomposition for nil affine Hecke
algebras.
\begin{prop}
\label{pr:isoH}
Assume $i\le n/2$. We have an isomorphism of graded
$(H_{i,n},H_{i,n})$-bimodules

\begin{align*}
\rho_i:H_{i,n}\otimes_{H_{i-1,n}}H_{i,n}(2)
\oplus \bigoplus_{j=0}^{n-2i-1} H_{i,n}(-2j)&\iso H_{i+1,n} \\
(a\otimes a',a_0,\ldots,a_{n-2i-1})&\mapsto
aT_i a'+\sum_{j=0}^{n-2i-1} a_j X_{i+1}^j.
\end{align*}
Assume $i\ge n/2$. We have an isomorphism of graded
$(H_{i,n},H_{i,n})$-bimodules

\begin{align*}
\rho_i:H_{i,n}\otimes_{H_{i-1,n}}H_{i,n}(2)&\iso H_{i+1,n}
\oplus \bigoplus_{j=0}^{2i-n-1} H_{i,n}(2j+2) \\
a\otimes a'&\mapsto
(aT_i a',aa',aX_ia',\ldots,aX_i^{2i-n-1}a').
\end{align*}
\end{prop}

\begin{proof}
By \cite[Proposition 5.32]{ChRou}, we know that the maps above
are isomorphisms after applying $-\otimes_{P_n^{\GS_n}}k$, where
$k$ is any field. 
So, the maps are isomorphisms after applying
$-\otimes_{P_n^{\GS_n}}\BZ$.
The proposition follows now from Nakayama's Lemma.
\end{proof}

Let $\CB_i$ be a basis for $H_{i,n}$ over $P_n^{\GS_n}$ and
$\{b^\vee\}_{b\in\CB_i}$ be the dual basis.
The symmetrizing forms on $H_{i,n}$ and $H_{i+1,n}$
induce a canonical morphism of
$(H_{i,n},H_{i,n})$-bimodules, which is the Frobenius form of
$H_{i+1,n}$ as an $H_{i,n}$-algebra:
$$\eps_i:H_{i+1,n}\to H_{i,n}(-2(n-2i-1))$$
and a canonical morphism of
$(H_{i+1,n},H_{i+1,n})$-bimodules
$$\eta_i:H_{i+1,n}\to H_{i+1,n}\otimes_{H_{i,n}}H_{i+1,n}(2(n-2i-1)).$$
They give rise to the counit and unit of the adjoint pair
$(\Res_{H_{i,n}}^{H_{i+1,n}},\Ind_{H_{i,n}}^{H_{i+1,n}})$. Note that
$t_i\circ\eps_i=t_{i+1}$.

\smallskip
\begin{lemma}
\label{le:epsi}
Let $P\in P_n^{\GS[i+2,n]}$ and $w\in\GS_{i+1}$.
We have 
$$\eps_i(P T_w s_1\cdots s_i)=\begin{cases}
\partial_{s_{n-1}\cdots s_{i+1}}(P) T_{ws_1\cdots s_i}
& \text{ if }w\in \GS_is_i\cdots s_1 \\
0 & \text{ otherwise.}
\end{cases}$$

We have
\begin{align*}
\eps_i(P)&=
\partial_{s_{n-1}\cdots s_{i+1}}\left(P(X_1-X_{i+1})\cdots
 (X_i-X_{i+1})\right) \\
\text{and }\eps_i(PT_i)&=
-\partial_{s_{n-1}\cdots s_{i+1}}\left(P(X_1-X_{i+1})\cdots (X_{i-1}-X_{i+1})
\right).
\end{align*}

When $i<n/2$, we have
$$\eps_i(T_i)=\eps_i(X_{i+1}^j)=0 \text{ for }j<n-2i-1
\text{ and } \eps_i(X_{i+1}^{n-2i-1})=(-1)^{n+1}.$$

When $i\ge n/2$, we have
$$\eps_i(T_i)=(-1)^{n+1}X_i^{2i-n}\pmod{\sum_{j=0}^{2i-n-1}
P_n^{\GS[1,i]\times\GS[i+1,n]}X_i^j}.$$
\end{lemma}

\begin{proof}
Let us consider the first equality. Let $f:H_{i+1,n}\to H_{i,n}$ be the
$\BZ$-linear map sending $P T_w s_1\cdots s_i$ to 
the second term of the equality. Note that
$f(Pa)=Pf(a)$ for all $P\in P_n^{\GS[i+1,n]}$ and
$a\in H_{i+1,n}$.

Let $j<i$, let $P\in P_n^{\GS[i+2,n]}$ and let
$w\in\GS_{i+1}$.
If $w{\not\in}\GS_is_i\cdots s_1$, then
$$f(T_j P T_w s_1\cdots s_i)=0=T_jf(P T_w s_1\cdots s_i).$$
Assume now $w\in \GS_is_i\cdots s_1$. Then,
\begin{align*}
f(T_j P T_w s_1\cdots s_i)&=\partial_{s_{n-1}\cdots s_{i+1}}(s_j(P))
T_jT_{ws_1\cdots s_i}+
\partial_{s_{n-1}\cdots s_{i+1}s_j}(P)T_{ws_1\cdots s_i}\\
&=T_j\partial_{s_{n-1}\cdots s_{i+1}}(P) T_{ws_1\cdots s_i}\\
&=T_jf(P T_w s_1\cdots s_i).
\end{align*}
It follows that $f$ is left $H_{i,n}$-linear.
Since $t_i\circ f=t_{i+1}$, we obtain the first equality from
Lemma \ref{le:trans2}.

\smallskip
We have
$$s_i\cdots s_1=(X_1-X_{i+1})\cdots (X_i-X_{i+1})T_i\cdots T_1
\mod F^{<(*,2i)}$$
hence $\eps_i(P)=
\partial_{s_{n-1}\cdots s_{i+1}}\left(P(X_1-X_{i+1})\cdots
 (X_i-X_{i+1})\right)$.

We have
$$T_is_i\cdots s_1=-T_i s_{i-1}\cdots s_1=
-(X_1-X_{i+1})\cdots (X_{i-1}-X_{i+1})T_i\cdots T_1
\mod F^{<(*,2i)}$$
hence $\eps_i(PT_i)=
-\partial_{s_{n-1}\cdots s_{i+1}}\left(P(X_1-X_{i+1})\cdots (X_{i-1}-X_{i+1})
\right)$.

\smallskip
The vanishing statements follow immediately from degree considerations.

\smallskip
Let $P=X_{i+1}^{n-2i-1} (X_1-X_{i+1})(X_2-X_{i+1})\cdots (X_i-X_{i+1})$.
We have
$$P=\sum_{j=0}^i (-1)^j X_{i+1}^{n-2i-1+j}e_{i-j}(X_1,\ldots,X_i).$$
We have $\partial_{s_{n-1}\cdots s_{i+1}}(X_{i+1}^r)=0$ for
$r<n-i-1$. It follows that
$$\eps_i(X_{i+1}^{n-2i-1})=\partial_{s_{n-1}\cdots s_{i+1}}(P)=(-1)^i
\partial_{s_{n-1}\cdots s_{i+1}}(X_{i+1}^{n-i-1})=(-1)^{n+1}$$
by Lemma \ref{le:basisrelativeP}.

\smallskip
Assume $i\ge n/2$. We have
\begin{align*}
\eps_i(T_i)&=-\partial_{s_{n-1}\cdots s_{i+1}}(
(X_1-X_{i+1})\cdots(X_{i-1}-X_{i+1}))\\
&=\sum_{j=n-i-1}^{i-1}(-1)^{j+1}
\partial_{s_{n-1}\cdots s_{i+1}}(X_{i+1}^j)
e_{i-1-j}(X_1,\ldots,X_{i-1}).
\end{align*}
Since 
$e_{k+1}(X_1,\ldots,X_{i-1})=e_{k+1}(X_1,\ldots,X_i)-
X_i e_k(X_1,\ldots,X_{i-1})$, we see
by induction that
$e_k(X_1,\ldots,X_{i-1})\in (-1)^kX_i^k+\sum_{j<k}
P_i^{\GS_i}X_i^j$.
It follows that
$$\eps_i(T_i)=(-1)^{n+1}X_i^{2i-n}\pmod{\sum_{j=0}^{2i-n-1}
P_n^{\GS[1,i]\times\GS[i+1,n]}X_i^j}.$$
\end{proof}

\begin{lemma}
\label{le:etai}
We have
$$\eta_i(1)=T_i\cdots T_1 s_1\cdots s_i \pi + \cdots +
T_1 s_1\cdots s_i \pi T_i \cdots T_2+ s_1\cdots s_i \pi T_i\cdots T_1,$$
where $\pi=\sum_{j=0}^{n-i-1} (-1)^j 
e_{n-i-1-j}(X_{i+2},\ldots,X_n)\otimes X_{i+1}^j$.

Let $P\in P_n^{\GS_i}$. We have
$$m((1\otimes P\otimes 1\otimes 1)\eta_i(1))=(-1)^{n+1}
\partial_{s_1\cdots s_i}(P
(X_{i+1}-X_{i+2})\cdots (X_{i+1}-X_n))$$
and
$$m((1\otimes T_{i+1}P\otimes 1\otimes 1)\eta_i(1))=(-1)^n
\partial_{s_1\cdots s_i}(P
(X_{i+1}-X_{i+3})\cdots (X_{i+1}-X_n)).$$
When $i>n/2-1$, we have
$$m((1\otimes X_{i+1}^j\otimes 1\otimes 1)\eta_i(1))=
m((1\otimes T_{i+1}\otimes 1\otimes 1)\eta_i(1))=0
\text{ for }j<2i-n+1$$
$$\text{and }
m((1\otimes X_{i+1}^{2i-n+1}\otimes 1\otimes 1)\eta_i(1))=(-1)^{n+1}.$$
When $i\le n/2-1$, we have
$$m((1\otimes T_{i+1}\otimes 1\otimes 1)\eta_i(1))=(-1)^nX_{i+2}^{n-2i-2}
\pmod{\sum_{j=0}^{n-2i-3}P_n^{\GS[1,i+1]\times
\GS[i+2,n]}X_{i+2}^j}.$$
\end{lemma}

\begin{proof}
Let $\CB=\{X_{i+1}^j\}_{0\le j\le n-i-1}$,
a basis for $\BZ[X_{i+1},\ldots,X_n]^{\GS[i+2,n]}$ over
$\BZ[X_{i+1},\ldots,X_n]^{\GS[i+1,n]}$ with dual
basis $\CB^\vee=\{(-1)^j e_{n-i-1-j}(X_{i+2},\ldots,X_n)\}$ for
the symmetrizing form $\partial_{s_{n-1}\cdots s_{i+1}}$
(Lemma \ref{le:basisrelativeP}).
Let $\CB$ be a basis for $\BZ[X_{i+1},\ldots,X_n]^{\GS[i+2,n]}$ over
$\BZ[X_{i+1},\ldots,X_n]^{\GS[i+1,n]}$ and $\CB^\vee$ the dual basis for
the symmetrizing form $\partial_{s_{n-1}\cdots s_{i+1}}$.
Let $\pi=\sum_{a\in\CB}a^\vee\otimes a$ be the Casimir element.
Let
$R=\{1,T_i,\ldots,T_i\cdots T_1\}$, a basis of
${^0H}_{i+1}^f$ over ${^0H}_i^f$. Its dual basis for
the Frobenius form 
$$T_w\mapsto
\begin{cases}
T_{ws_1\cdots s_i}
& \text{ if }w\in \GS_is_i\cdots s_1 \\
0 & \text{ otherwise}
\end{cases}$$
is given by $\{1^\vee=T_i\cdots T_1,\ldots,(T_i\cdots T_2)^\vee=T_1,
(T_i\cdots T_1)^\vee=1\}$.
It follows from Lemmas \ref{le:epsi} and \ref{le:trans2} that
$$T_ws_1\cdots s_i\mapsto
\begin{cases}
T_{ws_1\cdots s_i}
& \text{ if }w\in \GS_is_i\cdots s_1 \\
0 & \text{ otherwise.}
\end{cases}$$
extends to a Frobenius form for the $\left({^0H}_i\otimes
\BZ[X_{i+1},\ldots,X_n]^{\GS[i+2,n]}\right)$-algebra $H_{i+1,n}$ for which the basis
dual to $R$ is $\{h^\vee s_1\cdots s_i\}_{h\in R}$.
Then,
$\{ah\}_{a\in\CB,h\in R}$ is a basis of $H_{i+1,n}$ as an $H_{i,n}$-module.
Furthermore, the dual basis for the Frobenius form $\eps_i$ is
$\{h^\vee s_1\cdots s_i a^\vee \}_{a\in\CB,h\in R}$ (cf Lemma
\ref{le:epsi}).
So, we have
$$\eta_i(1)=T_i\cdots T_1 s_1\cdots s_i \pi + \cdots +
T_1 s_1\cdots s_i \pi T_i \cdots T_2+ s_1\cdots s_i \pi T_i\cdots T_1.$$
We deduce that
$$T_{w[1,i+2]}\eta_i(1)=T_{w[1,i+2]} s_1\cdots s_i \pi T_i\cdots T_1=
(-1)^i T_{w[1,i+2]} \pi T_i\cdots T_1.$$
Let $b\in H_{i+2,n}^{H_{i,n}}$. Define
$$f(b)=m((1\otimes b\otimes 1\otimes 1)\eta_i(1))=
\sum_{a\in\CB,h\in R} h^\vee s_1\cdots s_i a^\vee b a h\in
H_{i+2,n}.$$
We have
$$T_{w[1,i+2]}f(b)=(-1)^i T_{w[1,i+2]}\sum_{a\in\CB}a^\vee b a T_i\cdots T_1.$$
Since $\deg(\pi)=2(n-i-1)$, it follows that
$$T_{w[1,i+2]}f(b)=0 \text{ for } b\in F^{<(2(2i-n+1),*)}.$$

\smallskip
We have
$$\left(\BQ(X_1,\ldots,X_n)^{\GS[i+3,n]}\rtimes\GS_{i+2}\right)^
{\BQ(X_1,\ldots,X_n)^{\GS[i+2,n]}\rtimes\GS_{i+1}}=
\BQ(X_1,\ldots,X_n)^{\GS[1,i+1]\times\GS[i+3,n]}$$
and
$$H_{i+2,n}^{H_{i+1,n}}=P_n^{\GS[1,i+1]\times\GS[i+3,n]}.$$
We deduce that given $b\in H_{i+2,n}^{H_{i,n}}$,
then $f(b)\in P_n$. Note that
left multiplication by $T_{w[1,i+2]}$ is injective on
$P_n$.

\smallskip
We have $m(\pi)=(X_{i+2}-X_{i+1})\cdots (X_n-X_{i+1})$ by
Lemma \ref{le:Casimir}.
Let $P\in P_n^{\GS_i}$. We have
$T_{w[1,i+2]}f(P)=(-1)^iT_{w[1,i+2]}\partial_{s_1\cdots s_i}(P
(X_{i+2}-X_{i+1})\cdots (X_n-X_{i+1}))$,
hence 
$$f(P)=(-1)^{n+1}\partial_{s_1\cdots s_i}(P
(X_{i+1}-X_{i+2})\cdots (X_{i+1}-X_n)).$$

\smallskip
We have
$$T_{w[1,i+2]}f(T_{i+1}P)=(-1)^iT_{w[1,i+2]}\sum_{j=0}^{n-i-2}
(-1)^j e_{n-i-2-j}(X_{i+3},\ldots,X_n)PX_{i+1}^j T_i\cdots T_1$$
hence
$$f(T_{i+1}P)=(-1)^n\partial_{s_1\cdots s_i}(P
(X_{i+1}-X_{i+3})\cdots (X_{i+1}-X_n)).$$

\smallskip
Assume $i>n/2-1$.
The vanishing statements are immediate consequences of the previous
two equalities of the Lemma.

We have
\begin{align*}
f(X_{i+1}^{2i-n+1})&=(-1)^{n+1}\partial_{s_1\cdots s_i}(X_{i+1}^{2i-n+1}
(X_{i+1}-X_{i+2})\cdots (X_{i+1}-X_n))\\
&= (-1)^{n+1}\partial_{s_1\cdots s_i}(X_{i+1}^i)=(-1)^{n+1}.
\end{align*}

Assume $i\le n/2-1$. We have
\begin{align*}
f(T_{i+1})&=(-1)^n\partial_{s_1\cdots s_i}((X_{i+1}-X_{i+3})
\cdots (X_{i+1}-X_n))\\
&=
\sum_{j=i}^{n-i-2}(-1)^{i-j}\partial_{s_1\cdots s_i}(X_{i+1}^j)
e_{n-i-2-j}(X_{i+3},\ldots,X_n).
\end{align*}
By induction, we see that
$e_k(X_{i+3},\ldots,X_n)\in (-1)^kX_{i+2}^k+
\sum_{j<k}\BZ[X_{i+2},\ldots,X_n]^{\GS[i+2,n]}X_{i+2}^j$.
Consequently,
$$f(T_{i+1})=(-1)^nX_{i+2}^{n-2i-2}
\pmod{\sum_{j=0}^{n-2i-3}P_n^{\GS[1,i+1]\times
\GS[i+2,n]}X_{i+2}^j}.$$
\end{proof}

As a consequence of Lemmas \ref{le:epsi} and \ref{le:etai}, we obtain a
description of the units and counits $\eta_i$ and $\eps_i$
through the isomorphisms of Proposition \ref{pr:isoH}.

\begin{prop}
\label{pr:adjmin}
If $i<n/2$ then we have a commutative diagram
$$\xy
(-50,0) 
*+{H_{i,n}\otimes_{H_{i-1,n}}H_{i,n}(2)
\oplus H_{i,n}\oplus H_{i,n}(-2)\oplus\cdots\oplus H_{i,n}(-2(n-2i-1))}="1",
(40,0) *+{H_{i+1,n}}="2",
(40,-20) *+{H_{i,n}(-2(n-2i-1))}="3",
 \ar^{\eps_i}"2";"3",
 \ar_-{\sim}^-{\rho_i}"1";"2",
 \ar_{(0,0,\ldots,0,(-1)^{n+1})\ \ \ \ }"1";"3",
\endxy$$

If $i\ge n/2$ then the image of $\eps_i\circ\rho_i$ in 
$$\Hom_{H_{i,n},H_{i,n}}(H_{i,n}\otimes_{H_{i-1,n}}H_{i,n}(2),
H_{i,n}(2(2i-n+1)))/\bigoplus_{j=0}^{2i-n-1}
(a\otimes a'\mapsto aX_i^j a')\cdot Z(H_{i,n})_{2(2i-n-j)}$$
is equal to the image of the map $a\otimes a'\mapsto (-1)^{n+1}aX_i^{2i-n}a'$.

\smallskip
If $i\le n/2-1$ then the image of $\rho_{i+1}\circ\eta_i$ in
$$\Hom_{H_{i+1,n},H_{i+1,n}}(H_{i+1,n},H_{i+2,n}(2(n-2i-2)))/
\bigoplus_{j=0}^{n-2i-3}
X_{i+2}^j Z(H_{i+1,n})_{2(n-2i-2-j)}$$
is equal to $(-1)^n X_{i+2}^{n-2i-2}$.

\smallskip
If $i> n/2-1$ then we have a commutative diagram
$$\xy
(-50,0) *+{H_{i+1,n}}="1",
(40,0) *+{H_{i+1,n}\otimes_{H_{i,n}}H_{i+1,n}(2(n-2i-1))}="2",
(40,-20) *+{H_{i+2,n}(2(n-2i-2)) \oplus H_{i+1,n}(2(n-2i-1))\oplus 
H_{i+1,n}(2(n-2i+1))\oplus\cdots\oplus H_{i+1,n}}="3",
\ar^-{\eta_i}"1";"2",
\ar_-{(0,0,\ldots,0,(-1)^{n+1})\ \ \ \ }"1";"3",
\ar_-{\sim}^-{\rho_{i+1}}"2";"3",
\endxy$$
\end{prop}

\subsubsection{$\Gsl_2$-action}
\label{se:sl2min}
Let $\tilde{\CL}(-n)_\lambda=H_{(n+\lambda)/2,n}\mfree$ for 
$\lambda\in\{-n,-n+2,\ldots,n-2,n\}$. We define
$E=\bigoplus_{i=0}^{n-1}\Ind_{H_{i,n}}^{H_{i+1,n}}$ and
$F=\bigoplus_{i=0}^{n-1}\Res_{H_{i,n}}^{H_{i+1,n}}$.
We have a canonical adjunction $(E,F)$.
Multiplication
by $X_{i+1}$ induces an endomorphism of $\Ind_{H_{i,n}}^{H_{i+1,n}}$ and
taking the sum over all $i$, we obtain an endomorphism $x$ of $E$.
Similarly, $T_{i+1}$ induces an endomorphism of
$\Ind_{H_{i,n}}^{H_{i+2,n}}$ and we obtain an endomorphism $\tau$ of $E^2$.
Propositions \ref{pr:isoH} and \ref{pr:adjmin} show that this endows
$\tilde{\CL}(-n)=\bigoplus_\lambda \tilde{\CL}(-n)_\lambda$ with an
action of $\bar{\FA}'$.

Let $R=R_M:\CL(-n)\otimes_{\End(\bar{\idun}_{-n})}P_n^{\GS_n}
\to \tilde{\CL}(-n)$ be the
morphism of $2$-representations associated with $M=P_n^{\GS_n}\in
\tilde{\CL}(-n)_{-n}$ (cf \S \ref{se:lowest}).

\begin{prop}
\label{pr:minsl2}
The canonical map $\End(\bar{\idun}_{-n})\to P_n^{\GS_n}$ is an
isomorphism and $R$ induces an isomorphism of $2$-representations
of $\FA$
$$\CL(-n)\iso \tilde{\CL}(-n).$$
In particular, the action of $\FA$ on 
$\CL(-n)$ induces an action of $\bar{\FA}'$.
\end{prop}

\begin{proof}
The canonical map $\bar{\idun}_{-n}\to F^{(n)}E^{(n)}\bar{\idun}_{-n}$
is an isomorphism by Lemma \ref{le:basiccommutation}. It follows that
$E^{(n)}$ induces an isomorphism $\End(\bar{\idun}_{-n})\iso
\End(E^{(n)}\bar{\idun}_{-n})$. We have a commutative diagram of
canonical morphisms of $\End(\bar{\idun}_{-n})$-algebras
$$\xymatrix{
& \End(\bar{\idun}_{-n}) \ar[dl]_-{\sim} \ar[d] \ar[dr] \\
\End(E^{(n)}\bar{\idun}_{-n})\ar[r] & \End(E^{(n)}M) &
 P_n^{\GS_n} \ar[l]_-\sim \ar@/^2pc/[ll]_{ }
}$$
so the canonical map
$\End(\bar{\idun}_{-n})\to P_n^{\GS_n}$ is a split surjection of
 $\End(\bar{\idun}_{-n})$-algebras, hence it is an isomorphism.
The proposition follows.
\end{proof}

\subsection{Construction of representations}
In this section, we show that, for integrable representations, certain
axioms are consequences of others.

\subsubsection{Biadjointness}
\begin{thm}
\label{th:leftadjoint}
The canonical strict $2$-functor $\FA\to
\FA'$ induces an equivalence from the $2$-category of integrable
$2$-representations of $\FA'$ to the $2$-category of
integrable $2$-representations of $\FA$.
\end{thm}

\begin{proof}
It is enough to consider the case $\Gg=\Gsl_2$ and an integrable
$2$-representation $\CV$ of $\FA$.
The theorem will follow from the fact that
the maps $\tilde{\eps}_{\lambda}$
are the counits of an adjoint pair
$(F\Id_{\CV_\lambda},\Id_{\CV_\lambda}E)$.
It is enough to show that
\begin{equation}
\label{eq:compo}
(E\tilde{\eps})\circ(\hat{\eta} E) \text{ and }
(\tilde{\eps} F)\circ(F\hat{\eta}) \text{ are invertible}.
\end{equation}
Note that this holds for
$\CV=\CL(\lambda)$ by Proposition \ref{pr:adjmin}.
Lemma \ref{le:0onsimple} shows that the first invertibility in 
(\ref{eq:compo}) holds for any $\CV$.
Now, applying that result to $\CV^\opp$
endowed with the action of $\FA$ induced by $I$, we obtain that
the second invertibility in (\ref{eq:compo}) holds as well.
\end{proof}

A consequence of Theorem \ref{th:leftadjoint} is an extension of Lemma
 \ref{le:0onsimple}.

\begin{lemma}
\label{le:0onsimple2}
Let $C$ be a $1$-arrow of $\Ho^b(\FA)$. If $C$ acts by
$0$ on $\Ho^b(\CL(\lambda))$ for all $\lambda\in -X^+$, then $C$ acts
by $0$ on $\Ho^b(\CV)$ for all integrable $2$-representations 
$\CV$ of $\FA$ in $\FL in_\Bk$.

Let $f$ be a $2$-arrow of $\Ho^b(\FA)$.
 If $f$ is an isomorphism on $\Ho^b(\CL(\lambda))$ for all $\lambda\in -X^+$,
then $f$ is an isomorphism on $\Ho^b(\CV)$ for all integrable
$2$-representations 
$\CV$ of $\FA$ in $\FL in_\Bk$.
\end{lemma}

\subsubsection{Braid group action}
We follow the construction of \cite[\S 6]{ChRou}.
Let $s\in I$ and $\lambda\in X$.
Let $l=\langle\alpha_s^\vee,\lambda\rangle$.
We define a complex 
$\Theta_{s,\lambda}\in \Comp(\CHom_{\FA}(\lambda,\lambda-l\alpha_s))$ by
$\Theta_{s,\lambda}^r=F_s^{(l+r)}E_s^{(r)}$ for $r\ge 0$ and
$\Theta_{s,\lambda}^r=0$ for $r<0$. Since 
$b_{n-1}b_n=b_nb_{n-1}=b_n$, it follows that
$F_s^{l+r}\eta_s E_s^{r}:F_s^{l+r}E_s^{r}\to F_s^{l+r+1}E_s^{r+1}$ restricts
to a map
$$d^r:F_s^{(l+r)}E_s^{(r)}\to F_s^{(l+r+1)}E_s^{(r+1)}.$$
Since $b'_2b_2=0$, it follows that $d^{r+1}\circ d^r=0$ and $d$
defines the differential of $\Theta_{s,\lambda}$.

\smallskip
Let $\CV$ be an integrable $2$-representation of $\FA$ in $\FL in_\Bk$.
We define an endofunctor
$\Theta_s$ of $\Comp(\CV)$. 
Given $\lambda\in X$, we define $\Theta_s:\Comp^b(\CV_\lambda)\to
\Comp^b(\CV_{\sigma_s(\lambda)})$ as the total (direct sum)
complex associated with the complex of functors
$\Theta_{s,\lambda}\in \Comp(\CHom_{\FA}(\lambda,\sigma_s(\lambda)))$.

\begin{thm}
\label{th:refl}
The functor $\Theta_s$ induces a self-equivalence of $\Ho^b(\CV)$.
\end{thm}

\begin{proof}
Note that it is enough to consider the case $\Gg=\Gsl_2$. The functor
$\Theta_s$ has left and right adjoints.
The theorem holds when $\CV=\CL(-n)\otimes_{P_n^{\GS_n}}k$ for any field $k$
by \cite[Theorem 6.4]{ChRou}. So, it holds for
$\CL(-n)\otimes_{P_n^{\GS_n}}\BZ$, hence for $\CL(-n)$ by
Nakayama's Lemma.
The conclusion follows now from Lemma \ref{le:0onsimple2}.
\end{proof}

\begin{conj}[Chuang-R.] The functors $\Theta_s$ satisfy braid relations.
\end{conj}

\subsubsection{$\Gsl_2$-categorifications}
We recall the definition of \cite[\S 5.2.1]{ChRou}.
\label{se:defsl2cat}
Let $k$ be a field.
\begin{defi}
Let $\CV\in\FA b_k^f$.
An $\Gsl_2$-categorification on $\CV$ is the data of
\begin{itemize}
\item an adjoint pair
$(E,F)$ of exact functors $\CV\to\CV$
\item $X\in\End(E)$ and $T\in\End(E^2)$
\end{itemize}
such that
\begin{itemize}
\item the actions of $[E]$ and $[F]$ on $K_0(\CV)$ give a locally finite
representation of $\Gsl_2$
\item classes of simple objects are weight vectors
\item $F$ is isomorphic to a left adjoint of $E$
\item $X$ has a single eigenvalue
\item the action on $E^n$ of $X_i=E^{n-i}XE^{i-1}$ for $1\le i\le n$
and of $T_i=E^{n-i-1}TE^{i-1}$ for $1\le i\le n-1$
induce an action of an affine Hecke algebra with $q\not=1$,
a degenerate affine Hecke
algebra or a nil affine Hecke algebra of $\GL_n$.
\end{itemize}
\end{defi}

Note that the three types of actions (affine Hecke with $q\not=1$,
degenerate affine Hecke and nil affine Hecke) are equivalent by
Theorems \ref{th:equivdegnil} and \ref{th:equivaffnil}. The
endomorphism $T$ needs to be changed, as follows:
$$\xymatrix{
\text{affine} \ar@{<->}[r] & \text{nil} \\
T\ar@{|->}[r] & (qEX-XE)T+q}
\hskip2cm
\xymatrix{
\text{degenerate} \ar@{<->}[r] & \text{nil} \\
T\ar@{|->}[r] & (EX-XE+1)T+1}
$$
Note also that, in the nil case, if $a$ is the eigenvalue of $X$, then
by replacing $X$ by $X-a$ one reaches the case where $0$ is the eigenvalue of
$X$. As a consequence, given an $\Gsl_2$-categorification, one can construct
a new categorification by modifying $X$ and $T$ as above so that the action
of $X$ and $T$ induce an action of the nil affine Hecke algebra
${^0 H}_n$ on $\End(E^n)$ and $X$ is locally nilpotent.

\smallskip
In \cite{ChRou}, the case of nil affine Hecke algebras wasn't
considered. The equivalence of the definitions explained above shows that the
results of \cite{ChRou} generalize to this setting. It can also be seen
directly that all constructions, results and proofs in \cite{ChRou}
involving degenerate affine Hecke algebras carry over to nil affine Hecke
algebras. A key point is the commutation
relation between $T_i$ and a polynomial: that relation is the same
for the degenerate affine Hecke algebra and the nil affine Hecke algebra.
The definition of $c_n^\tau$ \cite[\S 3.1.4]{ChRou}
needs to be modified: we define
$c_n=T_{w[1,n]}$. Note that $T_{w[1,n]}^2=0$ for $n\ge 2$.
Given $M$ a projective $k({^0H}_n^f)$-module, we have
 $c_n M=\{m\in M~|~T_w m=0 \text{ for all }w\in\GS_n-\{1\}\}$.

\begin{rem}
We haven't included the parameters $a$ and $q$ in the definition, as they
are not needed here.
\end{rem}

Let $k$ be a field and $\CV\in\FA b_k^f$
endowed with an $\Gsl_2$-categorification. Let $V=\BC\otimes
K_0(\CV)$. The weight space decomposition $V=\bigoplus_{\lambda\in\BZ}
V_\lambda$ induces a decomposition
$\CV=\bigoplus_\lambda \CV_\lambda$, where
$\CV_\lambda=\{M\in\CV | [M]\in V_\lambda\}$ \cite[Proposition 5.5]{ChRou}.
Let $x=X$ and
$$\tau=\begin{cases}
(qEX-XE)^{-1}(T-q) & \text{ affine case}\\
(EX-XE+1)^{-1}(T-1) & \text{ degenerate affine case}\\
T & \text{ nil affine case.}
\end{cases}$$

\begin{thm}
\label{th:comparisonsl2}
The construction above defines an integrable
 $2$-representation of $\FA(\Gsl_2)$ on $\CV$.

Conversely, a integrable $2$-representation
of $\FA(\Gsl_2)$ on $\CV$ gives rise to
an $\Gsl_2$-categorification on $\CV$.

This provides an equivalence between the $2$-category of
$\Gsl_2$-categorifications and the $2$-category of integrable
$2$-representations of $\FA(\Gsl_2)$ in $\FA b_k^f$.
\end{thm}

\begin{proof}
By \cite[Theorem 5.27]{ChRou}, the maps $\rho_{s,\lambda}$ are
invertible and the result follows.
\end{proof}

In the isotypic case, we have a stronger result:
\begin{thm}
Let $k$ be a field and $\CV\in\FA b_k^f$.
Assume given an $\Gsl_2$-categorification on $\CV$ such that
$\BC\otimes K_0(\CV)$ is a multiple of an irreducible representation
of $\Gsl_2(\BC)$. Then, the construction of Theorem \ref{th:comparisonsl2}
 gives rise
to a $2$-representation of $\bar{\FA}'(\Gsl_2)$ on $\CV$.
\end{thm}

\begin{proof}
Theorems \ref{th:leftadjoint} and
\ref{th:comparisonsl2} provide an action of 
$\FA'$. Let $\lambda\in X$ be minimum such that $\CV_\lambda\not=0$.
Note that the theorem holds for $\CL(\lambda)$ by
Proposition \ref{pr:minsl2}.

Let $N\in\CV_{\lambda+2i}$ for some $i\ge 0$. Let
$N'$ be the cokernel of $\eps_i(N):E^iF^iN\to N$. We have 
$F^i N'=0$, hence $[N']=0$ in $K_0(\CV)$ since
the only non-zero elements of $\BC\otimes K_0(\CV)$ killed by
$[F]$ are in the $\lambda$-weight space. So, $N'=0$ and we deduce
that $N$ is a quotient of $E^i(F^iN)$.

Let $M\in\CV_{\lambda}$. Proposition
\ref{pr:highestweight} provides a fully faithful morphism of
$2$-representations
$$R:\CL(\lambda)\otimes_{\End(\bar{\idun}_\lambda)}\End(M)\to
\CV$$
with $R(\bar{\idun}_\lambda)\simeq M$. Since the theorem
holds for $\CL(\lambda)$, it follows that the relations defining
$\bar{\FA}'$ hold when applied to $E^iM$, for every $i$. It follows that
they hold for every quotient of $E^iM$. We deduce that the relations
hold on $\CV$.
\end{proof}

\subsubsection{Involution $\iota$}
\label{se:iota}
Let $\CV$ be an integrable $2$-representation of $\FA'$ 
in $\FL in_\Bk$.

Let $(\CV^\iota)_\lambda=\CV_{-\lambda}$, let
$E^\iota_s=F_s$ and $F^\iota_s=E_s$.
Let $x_s^\iota\in\End(E_s^\iota)$ corresponding
to $x_s\in\End(E_s)\iso\End(F_s)^\opp$ and let
$\tau_{st}^\iota\in\Hom(E_s^\iota E_t^\iota,E_t^\iota E_s^\iota)$
corresponding to $-\tau_{st}\in\Hom(E_sE_t,E_tE_s)\iso
\Hom(F_sF_t,F_tF_s)$.

The adjunction $(F_s,E_s)$ gives an
adjoint pair $(E^\iota_s,F^\iota_s)$: $\eta_s^\iota=\eta_s^l$ and
$\eps_s^\iota=\eps_s^l$.

\begin{prop}
\label{pr:iota}
The construction above defines a $2$-representation of
$\FA'$ on $\CV^\iota$.
\end{prop}

\begin{proof}
The relations (\ref{en:half1})-(\ref{en:half4}) in \S \ref{se:halfKM} are
clear.
Let us show that the maps $\rho_{s,\lambda}$ on $\CV^\iota$ are isomorphisms.
Thanks to Lemma \ref{le:0onsimple2}, it is enough to do
so for $\CV=\CL(-n)$ for some $n>0$.

Given a field $k$, consider the canonical $2$-representation of
$\FA'$ on $\CW=\bigoplus_i
\left( H_{i,n}\otimes_{P_n^{\GS_n}}k\right)\mMod$. The category
$\CW^\iota$ is endowed with a structure of $\Gsl_2$-categorification.
It follows from 
Theorem \ref{th:comparisonsl2} that the maps $\rho_{s,\lambda}$ are
isomorphisms for $\CW^\iota$.

We conclude now as in the proof of Proposition \ref{pr:isoH} that 
the maps $\rho_{s,\lambda}$ are isomorphisms for $\CL(-n)^\iota$.

\smallskip
We are left with proving the invertibility
of $\sigma_{st}$ for $s\not=t$. This is a consequence of Theorem
\ref{th:critsigmast} below. 
\end{proof}

Note that $\CV\mapsto\CV^\iota$ induces a strict endo $2$-functor of
the $2$-category of integrable $2$-representations of $\FA'$ that is a
$2$-equivalence.

\subsubsection{Relation $[E_s,F_t]=0$ for $s\not=t$}
\label{se:rank2}
Let $\{\CV\}_{\lambda\in X}$ be a family of $\Bk$-linear categories endowed
with the data of
\begin{itemize}
\item functors
$E_s:\CV_{\lambda}\to\CV_{\lambda+\alpha_s}$
and
$F_s:\CV_{\lambda}\to\CV_{\lambda-\alpha_s}$ for $s\in I$
\item $x_s\in\End(E_s)$ and
$\tau_{st}\in\Hom(E_sE_t,E_tE_s)$ for $s,t\in I$
\item an adjunction $(E_s,F_s)$ for $s\in I$
\end{itemize}
such that
\begin{itemize}
\item relations (\ref{en:half1})-(\ref{en:half4}) in \S \ref{se:halfKM} hold
\item the maps $\rho_{s,\lambda}$ are isomorphisms.
\end{itemize}

\begin{thm}
\label{th:critsigmast}
The data above defines a $2$-representation of $\FA'$ on
$\CV=\bigoplus_\lambda\CV_\lambda$.
\end{thm}

Theorem \ref{th:leftadjoint} provides maps $\eps^l_s$ and $\eta^l_s$ and
we only have to show the
invertibility of the maps $\sigma_{st}$ for any $s\not=t\in I$.
Note that the construction of \S \ref{se:iota} provide a category
$\CV^\iota$ satisfying the same properties as the category $\CV$.

\smallskip
Let $s\not=t\in I$. 
We write $Q_{ts}(u,v)=\sum_{a,b}q_{ab}u^av^b$ with $q_{a,b}\in \Bk$.
Let $\lambda\in X$ and $r\ge 0$. 
Consider the morphism
$$\psi:{^0H}_{r+1}\to\End(E_sE_t^r\idun_\lambda)$$
$$h\mapsto
(E_sE_t^r\xrightarrow{\eta\bullet}F_tE_tE_sE_t^r
\xrightarrow{F_t\tau_{ts}\bullet} F_tE_sE_t^{r+1}
\xrightarrow{F_t E_sh} F_tE_sE_t^{r+1}
\xrightarrow{F_t\tau_{st}\bullet} F_tE_tE_sE_t^r
\xrightarrow{\eps^l\bullet} E_sE_t^r).$$

\begin{lemma}
\label{le:specialtrace}
Let $a\le -\langle \alpha_t^\vee,\lambda\rangle -r-1$. We have
$$\psi(X_{r+1}^a T_{w[1,r+1]})=
\begin{cases}
T_{w[1,r]}q_{m_{ts},0} (-1)^{\langle\alpha_t^\vee,\lambda\rangle+
m_{ts}+1} & \text{ if }a=-\langle \alpha_t^\vee,\lambda\rangle -r-1 \\
0 & \text{ otherwise.}
\end{cases}
$$
\end{lemma}

\begin{proof}
Assume first $r=0$. We have
$$\psi(X^a)=
\sum_{\alpha,\beta\ge 0} q_{\beta,\alpha}
(\eps^l\circ (F_tX^{a+\beta})\circ\eta)X^\alpha$$
If $\eps^l\circ (F_tX^{a+\beta})\circ\eta\not=0$, then
$a+\beta\ge -\langle\alpha_t^\vee,\lambda\rangle+m_{ts}-1$, hence
$\beta=m_{ts}$ and $a=-\langle\alpha_t^\vee,\lambda\rangle-1$, and
$\eps^l\circ (F_tX^{a+\beta})\circ\eta=
(-1)^{\langle\alpha_t^\vee,\lambda\rangle+m_{ts}+1}$. The result follows.

\medskip
We assume now $r>0$. 
We have
\begin{multline*}
(\tau_{st}E_t)\circ (E_s\tau_{tt})\circ (\tau_{ts}E_t)=\\
=(E_t\tau_{ts})\circ (\tau_{tt}E_s)\circ(E_t\tau_{st})+
\sum_{\substack{\beta\ge 0\\ m_{ts}>\alpha_1+\alpha_2\ge 0}}
q_{\alpha_1+\alpha_2+1,\beta}(X^{\alpha_1}E_sE_t)(E_t X^{\beta}E_t)
(E_tE_s X^{\alpha_2}).
\end{multline*}
Let
$$f_a=(E_sE_t^r\xrightarrow{\tau_{st}\bullet}E_tE_sE_t^{r-1}
\xrightarrow{\eta\bullet} F_tE_t^2E_sE_t^{r-1}
\xrightarrow{F_t (X^aE_t\circ T)\bullet} F_tE_t^2E_sE_t^{r-1}
\xrightarrow{\eps^l\bullet} E_tE_sE_t^{r-1}
\xrightarrow{\tau_{ts}\bullet} E_sE_t^r)$$
and
$$g_b=E_sE_t^r
\xrightarrow{\eta\bullet} F_tE_tE_sE_t^r
\xrightarrow{F_t X_{r+1}^b\bullet} F_tE_tE_sE_t^r
\xrightarrow{\eps^l\bullet} E_sE_t^r.$$

We have $X_{r+1}^a T_{w[1,r+1]}=T_{w[1,r]}X_{r+1}^a T_r\cdots T_1$, hence
$$\psi(X_{r+1}^a T_{w[1,r+1]})=T_{w[1,r]}f_a T_{r-1}\cdots T_1+
\sum_{\substack{\beta\ge 0\\ m_{ts}>\alpha_1+\alpha_2\ge 0}}
q_{\alpha_1+\alpha_2+1,\beta}T_{w[1,r]} g_{a+\alpha_1}
\partial_{s_1\cdots s_{r-1}}(X_r^{\alpha_2})
(X^\beta E_t^r).$$

We have a commutative diagram (Lemma \ref{le:sigmamn} and Chevalley
duality)
\begin{equation}
\label{eq:traceT}
\xymatrix{
E_t\ar[r]^-{\eta E_t}\ar[d]_{E_t\eta}  & F_tE_t^2 \ar[d]^{F_tT} \\
E_tF_tE_t \ar[r]_-{\sigma E_t} & F_t E_t^2 \ar[r]^-{\eps^l E_t} & E_t
}
\end{equation}

\medskip
$\bullet\ $Assume first $\langle\alpha_t^\vee,\lambda\rangle+2r-m_{ts}<0$.
The diagram (\ref{eq:traceT}) shows that the composition
$$E_tE_sE_t^{r-1} \xrightarrow{\eta\bullet}
F_tE_t^2E_sE_t^{r-1}\xrightarrow{F_tT\bullet}
F_tE_t^2E_sE_t^{r-1}\xrightarrow{\eps^l\bullet} E_tE_sE_t^{r-1}$$
vanishes.
Since $X_2^a T=TX_1^a+\sum_{c=0}^{a-1}X_2^c X_1^{a-1-c}$, it follows that
$$E_tE_sE_t^{r-1} \xrightarrow{\eta\bullet}
F_tE_t^2E_sE_t^{r-1}\xrightarrow{F_t(X^aE_t\circ T)\bullet}
F_tE_t^2E_sE_t^{r-1}\xrightarrow{\eps^l\bullet} E_tE_sE_t^{r-1}$$
equals
$$\sum_{c=0}^{a-1}(\eps^l\circ (F_tX^c)\circ\eta)
X^{a-1-c}E_sE_t^{r-1},$$
hence
$$T_{w[1,r]}f_a T_{r-1}\cdots T_1
=\sum_{\substack{0\le c\le a-1\\ 0\le\alpha\le m_{ts}\\
0\le\beta\le m_{st}}}
T_{w[1,r]} q_{\alpha,\beta} X_{E_s}^\beta
(\eps^l\circ (F_tX^c)\circ\eta)
\partial_{s_1\cdots s_{r-1}}(X_r^{a-1-c+\alpha}) 
$$

If $\eps^l\circ (F_tX^c)\circ\eta\not=0$, then
$c\ge -\langle\alpha_t^\vee,\lambda\rangle-2r+m_{ts}-1$. If
$\partial_{s_1\cdots s_{r-1}} (X_r^{a-1-c+\alpha})\not=0$, then
$a-1-c+\alpha\ge r-1$. If both of those terms are non zero, then
$a\ge -\langle\alpha_t^\vee,\lambda\rangle-r+(m_{ts}-\alpha)-1$, hence
$a=-\langle\alpha_t^\vee,\lambda\rangle-r-1$,
$\alpha=m_{ts}$ and $a-1-c=r-1-m_{ts}$. In particular, we have
$r>m_{ts}$. So, we have 
$$T_{w[1,r]}f_a T_{r-1}\cdots T_1=\begin{cases}
T_{w[1,r]}q_{m_{ts},0} (-1)^{\langle\alpha_t^\vee,\lambda\rangle+
m_{ts}+1} & \text{  if } a=-\langle\alpha_t^\vee,\lambda\rangle-r-1
\text{ and } r>m_{ts} \\
0 & \text{ otherwise.}
\end{cases}$$

If $g_{a+\alpha_1} \partial_{s_1\cdots s_{r-1}}(X_r^{\alpha_2})\not=0$, then
$a+\alpha_1\ge -\langle\alpha_t^\vee,\lambda\rangle-2r+m_{ts}-1$ and
$\alpha_2\ge r-1$, hence 
$a\ge -\langle\alpha_t^\vee,\lambda\rangle -r-2+(m_{ts}-\alpha_1-\alpha_2)
\ge -\langle\alpha_t^\vee,\lambda\rangle -r-1$. We obtain
$a=-\langle\alpha_t^\vee,\lambda\rangle-r-1$,
$\alpha_2=r-1$ and $\alpha_1+\alpha_2=m_{ts}-1$.
In particular, $r\le m_{ts}$. So, we have
\begin{multline*}
\sum_{\substack{\beta\ge 0\\ \alpha_1+\alpha_2\ge 0}}
q_{\alpha_1+\alpha_2+1,\beta}T_{w[1,r]} g_{a+\alpha_1}
\partial_{s_1\cdots s_{r-1}}(X_r^{\alpha_2})
(X^\beta E_t^r)
=\\
=\begin{cases}
T_{w[1,r]}q_{m_{ts},0} (-1)^{\langle\alpha_t^\vee,\lambda\rangle+
m_{ts}+1} & \text{  if } a=-\langle\alpha_t^\vee,\lambda\rangle-r-1
\text{ and } r\le m_{ts} \\
0 & \text{ otherwise.}
\end{cases}
\end{multline*}
So, we have shown that
$$\psi(X_{r+1}^a T_{w[1,r+1]})
\begin{cases}
T_{w[1,r]}q_{m_{ts},0} (-1)^{\langle\alpha_t^\vee,\lambda\rangle+
m_{ts}+1} & \text{ if }a=-\langle \alpha_t^\vee,\lambda\rangle -r-1 \\
0 & \text{ otherwise.}
\end{cases}
$$

\medskip
$\bullet\ $Assume now $\langle\alpha_t^\vee,\lambda\rangle+2r-m_{ts}\ge 0$.
We can assume that $\langle\alpha_t^\vee,\lambda\rangle+r<0$, for
otherwise the lemma is empty. So, we have $r>m_{ts}$.

If $\partial_{s_1\cdots s_{r-1}}(X_r^{\alpha_2})\not=0$, then
$\alpha_2\ge r-1$, hence $m_{ts}\ge r$, which is impossible. So,
$$\sum_{\substack{\beta\ge 0\\ \alpha_1+\alpha_2\ge 0}}
q_{\alpha_1+\alpha_2+1,\beta}T_{w[1,r]} g_{a+\alpha_1}
\partial_{s_1\cdots s_{r-1}}(X_r^{\alpha_2})
(X^\beta E_t^r)=0.$$

Let $\mu=\lambda+r\alpha_t+\alpha_s$.
The diagram (\ref{eq:traceT}) shows that
there are elements $z_i\in Z(\CV_\mu)$ with
$z_{\langle \alpha_t^\vee,\mu\rangle}=(-1)^{\langle \alpha_t^\vee,\mu\rangle+1}$
such that
$$(E_t\idun_{\mu-\alpha_t}
\xrightarrow{\eta\bullet} F_tE_t^2\idun_{\mu-\alpha_t}
\xrightarrow{F_t T\bullet} F_tE_t^2\idun_{\mu-\alpha_t}
\xrightarrow{\eps^l\bullet} E_t\idun_{\mu-\alpha_t})=
\sum_{i=0}^{\langle \alpha_t^\vee,\mu\rangle}z_iX^i.$$
So,
\begin{multline*}
T_{w[1,r]}f_a T_{r-1}\cdots T_1
=\sum_{\substack{0\le c\le a-1\\ 0\le\alpha\le m_{ts}\\
0\le\beta\le m_{st}}}
T_{w[1,r]} q_{\alpha,\beta} X_{E_s}^\beta
(\eps^l\circ (F_tX^c)\circ\eta)
\partial_{s_1\cdots s_{r-1}}(X_r^{a-1-c+\alpha})+ \\
+
\sum_{\substack{0\le\alpha\le m_{ts}\\ 0\le\beta\le m_{st}}}
T_{w[1,r]} q_{\alpha,\beta} X_{E_s}^\beta
\sum_{i=0}^{\langle \alpha_t^\vee,\mu\rangle} z_i
\partial_{s_1\cdots s_{r-1}}(X_r^{a+i+\alpha}).
\end{multline*}
We have
$$a-1+m_{ts}\le -\langle\alpha_t^\vee,\lambda\rangle-r-2+m_{ts}\le r-2$$
hence $\partial_{s_1\cdots s_{r-1}}(X_r^{a-1-c+\alpha})=0$
for all $a$, $c\ge 0$ and $\alpha\le m_{ts}$.

We have
$a+\langle\alpha_t^\vee,\mu\rangle+m_{ts}\le r-1$.
If $\partial_{s_1\cdots s_{r-1}}(X_r^{a+i+\alpha})\not=0$, then
$a=-\langle \alpha_t^\vee,\lambda\rangle-r-1$, 
$i=\langle\alpha_t^\vee,\mu\rangle$ and $\alpha=m_{ts}$.

We have shown that 
$$T_{w[1,r]}f_a T_{r-1}\cdots T_1=
\begin{cases}
T_{w[1,r]}q_{m_{ts},0} (-1)^{\langle\alpha_t^\vee,\lambda\rangle+
m_{ts}+1} & \text{  if } a=-\langle\alpha_t^\vee,\lambda\rangle-r-1
\text{ and } r>m_{ts} \\
0 & \text{ otherwise.}
\end{cases}$$
The lemma follows.
\end{proof}

\begin{proof}[Proof of Theorem \ref{th:critsigmast}]

Let $N\in\CV_\lambda$ such that $F_tN=0$.
Define
$$L=\bigoplus_{\substack{w\in\GS_{r+1}^r\\ i<-\langle\alpha_t^\vee,\lambda
\rangle-r}} T_wX_{r+1}^i\BZ \text{ and }
L'=\bigoplus_{\substack{w\in{^r\GS_{r+1}}\\ i<-\langle\alpha_t^\vee,\lambda
\rangle-r}} X_{r+1}^iT_w\BZ.$$
We have $L\simeq L(r+1,1,1,\lambda)$ (cf \S \ref{se:relationsl2}).
We have an isomorphism (Lemma \ref{le:basiccommutation})
$$\mathrm{act}\circ (\id\otimes \eta E_t^r):
L\otimes_{\BZ} E^r_tN\iso F_tE_t^{r+1}N.$$
Similarly, applying Lemma \ref{le:basiccommutation} to $\CV^\iota$,
we obtain an isomorphism
$$(\id\otimes \eps^l E_t^r)\circ \mathrm{act}^*:
F_tE_t^{r+1}N\iso
L^{\prime *}\otimes_\BZ E_t^rN.$$

We have a commutative diagram
$$\xymatrix{
L\otimes E_sE_t^r\ar[r]_-{\mathrm{act}\circ (\id\otimes E_s\eta\bullet)}
 \ar[d]_{\id\otimes\eta\bullet}&
E_sF_tE_t^{r+1}\ar@/^2pc/[rrr]^-{\sigma_{st}\bullet} \ar[r]_-{\eta\bullet} &
F_tE_tE_sF_tE_t^{r+1} \ar[r]_-{F_t\tau_{ts}\bullet} &
F_tE_sE_tF_tE_t^{r+1} \ar[r]_-{F_tE_s\eps\bullet} &
F_tE_sE_t^{r+1} \\
L\otimes F_tE_tE_sE_t^r\ar[r]^-{F_t\tau_{ts}\bullet} &
L\otimes F_tE_sE_t^{r+1} \ar[rr]^-{F_tE_sE_t\eta\bullet}
 \ar@/_2pc/[rrr]_-{\id} &&
L\otimes F_tE_sE_tF_tE_t^{r+1} \ar[r]^-{F_tE_s\eps\bullet} &
L\otimes F_tE_sE_t^{r+1} \ar[u]_{\mathrm{act}}
}$$
and a commutative diagram
$$\xymatrix{
F_tE_sE_t^{r+1}\ar[r]_-{F_tE_s\eta^l\bullet} \ar[d]_{\mathrm{act}^*}
 \ar@/^2pc/[rrr]^-{\sigma_{ts}^\iota\bullet} &
F_tE_sE_tF_tE_t^{r+1}\ar[r]_-{F_t\tau_{st}\bullet} &
F_tE_tE_sF_tE_t^{r+1}\ar[r]_-{\eps^l\bullet} &
E_sF_tE_t^{r+1}\ar[d]^-{(\id\otimes E_s\eps^l\bullet)\circ \mathrm{act}^*}\\
L^{\prime *}\otimes F_tE_sE_t^{r+1}\ar[d]_-{F_tE_s\eta^l\bullet}
 \ar[r]^-{\id}&
L^{\prime *}\otimes F_tE_sE_t^{r+1}\ar[r]^-{F_t\tau_{st}\bullet} &
L^{\prime *}\otimes F_tE_tE_sE_t^r\ar[r]^-{\eps^l\bullet} &
L^{\prime *}\otimes E_s E_t^r \\
L^{\prime *}\otimes F_tE_sE_tF_tE_t^{r+1}
\ar[ur]_-{F_tE_sE_t\eps^l\bullet} 
}$$

We have a commutative diagram
$$\xymatrix{
L\otimes E_sE_t^r\ar[rr]^-{\mathrm{act}\circ (\id\otimes E_s\eta\bullet)}
 \ar[d]_{\eta\bullet}&&
E_sF_tE_t^{r+1}\ar[r]^-{\sigma_{st}\bullet} & 
F_tE_sE_t^{r+1}\ar[r]^-{\sigma_{ts}^\iota\bullet} & E_sF_tE_t^{r+1}
\hskip 0.6cm
 \ar[r]^-{(\id\otimes E_s\eps^l\bullet)\circ \mathrm{act}^*} &
\hskip 0.3cm
L^{\prime *}\otimes E_sE_t^r \\
L\otimes F_tE_tE_sE_t^r \ar[rr]_-{F_t\tau_{ts}\bullet} &&
L\otimes F_tE_sE_t^{r+1} \ar[r]_-{\mathrm{act}} & 
F_tE_sE_t^{r+1} \ar[r]_-{\mathrm{act}^*} & 
L^{\prime *}\otimes F_tE_sE_t^{r+1} \ar[r]_-{F_t\tau_{st}\bullet} & 
L^{\prime *}\otimes E_sE_t^r \ar[u]_-{\eps^l\bullet} 
}$$

We will show that the top horizontal composition in the diagram
above is an isomorphism when applied to $N$:

$$f:L\otimes E_sE_t^rN\iso L^{\prime *}\otimes E_sE_t^rN.$$
It is enough to show
that the map $\gamma$ obtained from $f$
by left multiplication by $T_{w[1,r+1]}$ is invertible, as in the
proof of Lemma \ref{le:basiccommutation}. Lemma
\ref{le:specialtrace} shows that the map
$$ E_sE_t^{(r)}\xrightarrow{X^a\otimes \id}
T_{w[1,r+1]} (L\otimes E_sE_t^r)\xrightarrow{\gamma} 
T_{w[1,r+1]} (L^{\prime *}\otimes E_sE_t^r)
\xrightarrow{\langle X^{a'},-\rangle} E_sE_t^{(r)}$$
is $0$ for $a+a'<-\langle\alpha_t^\vee,\lambda\rangle-r-1$ and it is
an isomorphism for $a+a'=-\langle\alpha_t^\vee,\lambda\rangle-r-1$.
So, $\gamma$ is an isomorphism and $f$ as well.
Consequently, the composition
$\sigma_{ts}^\iota\circ\sigma_{st}$ is an isomorphism when applied
to $E_t^r N$. We conclude from Lemma \ref{le:0onsimple2}
that it is an isomorphism on all objects of $\CV$.

\smallskip
We apply now the result above to $\CV^\iota$: it shows that
$\sigma_{ts}^\iota$ has a left inverse. So,
$\sigma_{ts}^\iota$ is invertible, hence $\sigma_{st}$ is invertible
as well.
\end{proof}

\subsubsection{Control from $K_0$}
\label{se:controlK0}

\begin{thm}
\label{th:criteriumK0}
Consider a root datum with associated Kac-Moody algebra
$\Gg$ and associated ring $\Bk$.

Let $k$ be a field that is a $\Bk$-algebra and $\CV\in\FA b_k^f$.

Assume given
\begin{itemize}
\item
an adjoint pair $(E_s,F_s)$ of exact functors
$\CV\to\CV$ for every $s\in I$
\item
$x_s\in\End(E_s)$ and $\tau_{st}\in\Hom(E_sE_t,E_tE_s)$ for every $s,t\in I$.
\item a decomposition $\CV=\bigoplus_{\lambda\in X} \CV_\lambda$.
\end{itemize}

We assume that
\begin{itemize}
\item $F_s$ is isomorphic to a left adjoint of $E_s$
\item $E_s(\CV_\lambda)\subset \CV_{\lambda+\alpha_s}$ and
$F_s(\CV_\lambda)\subset \CV_{\lambda-\alpha_s}$
\item $\{[E_s],[F_s]\}_{s\in I}$ induce an integrable representation of
$\Gg$ on $V=K_0(\CV)$
\item relations (\ref{en:half1})-(\ref{en:half4})
in \S \ref{se:halfKM} hold
\end{itemize}

Then, the data above
defines an integrable action of $\FA(\Gg)$ on $\CV$.
\end{thm}

\begin{proof}
This is a consequence of Theorems \ref{th:comparisonsl2} and
\ref{th:critsigmast}.
\end{proof}

\subsubsection{Type $A$}
\label{se:typeA}
Let $k$ be a field.
Let $q\in k^\times$ and let $I$ be a subset of $k$. Assume $0\not\in I$ if 
$q\not=1$ and consider the corresponding Lie algebra
$\Gsl_{I_q}$ as in \S \ref{se:typeAgraphs}.

Let $\CV$ be a $k$-linear category.
Consider 
\begin{itemize}
\item an adjoint pair $(E,F)$ of endofunctors of $\CV$
\item $X\in\End(E)$ and $T\in\End(E^2)$.
\end{itemize}
Assume there are decompositions
$E=\bigoplus_{i\in I} E_i$ and
$F=\bigoplus_{i\in I} F_i$,
where $X-i$ is locally nilpotent on $E_i$ and $F_i$.

\smallskip
When $q=1$, we put
$x_i=X-i$ (acting on $E_i$) and
$$\tau_{ij}= \begin{cases}
(E_iX-XE_j+1)^{-1}(T-1) & \text{ if }i=j \\
(E_iX-XE_j)T+1 & \text{ if }i=j+1 \\
\frac{E_iX-XE_j}{E_iX-XE_j+1}(T-1)+1 & \text{ otherwise}
\end{cases}$$
(restricted to $E_iE_j$).

When $q\not=1$, we put
$x_i=i^{-1}X$ (acting on $E_i$) and
$$\tau_{ij}= \begin{cases}
i(qE_iX-XE_j)^{-1}(T-q) & \text{ if }i=j \\
q^{-1}i^{-1}(E_iX-XE_j)T+i^{-1}(1-q^{-1})XE_j & \text{ if }i=qj \\
\frac{E_iX-XE_j}{qE_iX-XE_j}(T-q)+1 & \text{ otherwise}
\end{cases}$$
(restricted to $E_iE_j$).

\smallskip

Assume that
there is a decomposition $\CV=\bigoplus_{\lambda\in X}\CV_\lambda$
such that
\begin{itemize}
\item $E_i(\CV_\lambda)\subset \CV_{\lambda+\alpha_i}$ and
$F_i(\CV_\lambda)\subset \CV_{\lambda-\alpha_i}$
\item $E_i$ and $F_i$ are locally nilpotent
\item  when $\langle\alpha_s^\vee,\lambda\rangle\ge 0$, the map
$\sigma_{ss}+\sum_{i=0}^{\langle\alpha_s^\vee,\lambda\rangle-1}\eps_s\circ
(x_s^i F_s):
E_s F_s(M)\to F_s E_s(M)\oplus M^{\langle\alpha_s^\vee,\lambda\rangle}$ is
invertible for $M\in\CV_\lambda$
\item  when $\langle\alpha_s^\vee,\lambda\rangle\le 0$, the map
$\sigma_{ss}+\sum_{i=0}^{-1-\langle\alpha_s^\vee,\lambda\rangle}
(F_s x_s^i)\circ \eta_s:
E_s F_s(M)\oplus
M^{-\langle\alpha_s^\vee,\lambda\rangle}
\to F_s E_s(M)$ is invertible for $M\in\CV_\lambda$.
\end{itemize}

\begin{thm}
\label{th:slcatlinear}
The data above defines an action of $\FA_\BZ(\Gsl_{I_q})\otimes k$ on $\CV$.
\end{thm}

\begin{proof}
The $x_i$'s and $\tau_{ij}$'s satisfy the relations
(\ref{en:half1})-(\ref{en:half4}) in \S \ref{se:halfKM} thanks to
Propositions \ref{pr:isowithgraph1} and \ref{pr:isowithgraphq}.
The invertibility of $\sigma_{st}$ for $s\not=t$ follows from 
Theorem \ref{th:critsigmast}.
\end{proof}

\subsubsection{$\Gsl$-categorifications}
\label{se:slcat}
Let $k$ be a field.
Let $q\in k^\times$ and let $I$ be a subset of $k$. Assume $0\not\in I$ if 
$q\not=1$ and consider the corresponding Lie algebra
$\Gsl_{I_q}$ as in \S \ref{se:typeAgraphs}.

Let $\CV\in\FA b_k^f$.

\begin{defi}[Chuang-Rouquier]
An $\Gsl_{I_q}$-categorification on $\CV$
is the data of
\begin{itemize}
\item an adjoint pair
$(E,F)$ of exact functors $\CV\to\CV$
\item $X\in\End(E)$ and $T\in\End(E^2)$
\item a decomposition $\CV=\bigoplus_{\lambda\in X} \CV_\lambda$.
\end{itemize}
Given $i\in k$, let $E_i$ (resp. $F_i$)
be the generalized $i$-eigenspace of $X$ acting on $E$ (resp. $F$).
We assume that
\begin{itemize}
\item $E=\bigoplus_{i\in I}E_i$
\item the action of $\{[E_i],[F_i]\}_{i\in I}$ on $K_0(\CV)$ gives an
integrable representation of $\Gsl_{I_q}'$
\item $E_i(\CV_\lambda)\subset \CV_{\lambda+\alpha_i}$ and
$F_i(\CV_\lambda)\subset \CV_{\lambda-\alpha_i}$
\item $F$ is isomorphic to a left adjoint of $E$
\item the action on $E^n$ of $X_i=E^{n-i}XE^{i-1}$ for $1\le i\le n$
and of $T_i=E^{n-i-1}TE^{i-1}$ for $1\le i\le n-1$
induce an action of
\begin{itemize}
\item an affine Hecke algebra if $q\not=1$
\item a degenerate affine Hecke if $q=1$.
\end{itemize}
\end{itemize}
\end{defi}

Consider an $\Gsl_{I_q}$-categorification on $\CV$.

\begin{thm}
\label{th:slcat2rep}
Assume given an $\Gsl_{I_q}$-categorification on $\CV$. 
The construction of \S \ref{se:typeA}
gives rise to an action of $\FA_\BZ(\Gsl_{I_q})\otimes k$ on $\CV$.

Conversely, an integrable action of $\FA_\BZ(\Gsl_{I_q})\otimes k$
 on $\CV$ gives rise to
an $\Gsl_{I_q}$-categorification on $\CV$.
\end{thm}

\begin{proof}
 The morphisms $\rho_{s,\lambda}$ are invertible by
\cite[Theorem 5.27]{ChRou}. The
theorem follows now from Theorem \ref{th:slcatlinear}.
\end{proof}

\begin{rem}
A setting for categorifications of $\Gsl_2$ \cite{Lau} and
$\Gsl_n$ \cite{KhoLau3} has been proposed recently. While they do not check
its compatibility with the older definition above, its $2$-representations
should give a full $2$-subcategory of those above, related to $\bar{\FA}'$.
\end{rem}

Note that it is straightforward to define a notion of
$\Gsl_{I_q}'$-categorifications:

\begin{defi}
An $\Gsl_{I_q}'$-categorification on $\CV$
is the data of
\begin{itemize}
\item an adjoint pair
$(E,F)$ of exact functors $\CV\to\CV$
\item $X\in\End(E)$ and $T\in\End(E^2)$.
\end{itemize}
Given $i\in k$, let $E_i$ (resp. $F_i$)
be the generalized $i$-eigenspace of $X$ acting on $E$ (resp. $F$).
We assume that
\begin{itemize}
\item $E=\bigoplus_{i\in I}E_i$
\item the action of $\{[E_i],[F_i]\}_{i\in I}$ on $K_0(\CV)$ gives an
integrable representation of $\Gsl_{I_q}'$
\item classes of simple objects are weight vectors
\item $F$ is isomorphic to a left adjoint of $E$
\item the action on $E^n$ of $X_i=E^{n-i}XE^{i-1}$ for $1\le i\le n$
and of $T_i=E^{n-i-1}TE^{i-1}$ for $1\le i\le n-1$
induce an action of
\begin{itemize}
\item an affine Hecke algebra if $q\not=1$
\item a degenerate affine Hecke if $q=1$.
\end{itemize}
\end{itemize}
\end{defi}

When $I_q$ has no component of type $\tilde{A}_n$, then the notion of
$\Gsl_{I_q}'$-categorification coincides with that of
$\Gsl_{I_q}$-categorification
(put $\CV_\lambda=\{M\in\CV | [M]\in V_\lambda\}$).

\end{document}